\documentclass[reqno,10pt]{amsart}
\usepackage{amssymb,amsmath,amsthm,amsfonts,color}

\usepackage{mathrsfs,dsfont,comment,mathscinet}

\usepackage{ulem}
\usepackage{enumerate,esint}
\usepackage[no label, final]{showlabels}
\usepackage[left=2.8cm,right=2.8cm,top=2.8cm,bottom=2.8cm]{geometry}
\usepackage{mathtools}
\usepackage{graphicx,tikz}
\usepackage[format=hang,labelfont=bf]{caption}

\usepackage{epstopdf}
\usepackage[colorlinks=true, pdfstartview=FitV, linkcolor=blue, 
            citecolor=blue, urlcolor=blue]{hyperref}

\allowdisplaybreaks
\numberwithin{equation}{section}
\theoremstyle{plain}
\newtheorem{theorem}{Theorem}[section]
\newtheorem{proposition}[theorem]{Proposition}
\newtheorem{lemma}[theorem]{Lemma}

\newtheorem{corollary}[theorem]{Corollary}
\theoremstyle{definition}

\newtheorem{remark}[theorem]{Remark}

\newtheorem*{theorem*}{Theorem}

\makeatletter
\def\th@plain{%
  \thm@notefont{}
  \itshape 
}
\def\th@definition{%
  \thm@notefont{}
  \normalfont 
}
\makeatother

\newcommand\R{\mathbb R}

\newcommand\N{\mathbb N}
\newcommand\Z{\mathbb{Z}}
\renewcommand\S{\mathbb S}

\renewcommand{\d}{\mathrm{d}}

\newcommand{\eps}{\varepsilon}
\newcommand{\vth}{\vartheta}

\newcommand{\restr}[1]{|_{#1}}

\def\Xint#1{\mathchoice
{\XXint\displaystyle\textstyle{#1}}%
{\XXint\textstyle\scriptstyle{#1}}%
{\XXint\scriptstyle\scriptscriptstyle{#1}}%
{\XXint\scriptscriptstyle\scriptscriptstyle{#1}}%
\!\int}
\def\XXint#1#2#3{{\setbox0=\hbox{$#1{#2#3}{\int}$ }
\vcenter{\hbox{$#2#3$ }}\kern-.6\wd0}}

\def\dashint{\Xint-}

\DeclareFontFamily{U}{mathx}{\hyphenchar\font45}
\DeclareFontShape{U}{mathx}{m}{n}{
      <5> <6> <7> <8> <9> <10>
      <10.95> <12> <14.4> <17.28> <20.74> <24.88>
      mathx10
      }{}
\DeclareSymbolFont{mathx}{U}{mathx}{m}{n}
\DeclareFontSubstitution{U}{mathx}{m}{n}
\DeclareMathAccent{\widecheck}{0}{mathx}{"71}
\DeclareMathAccent{\wideparen}{0}{mathx}{"75}

\newcommand{\loc}{\mathrm{loc}}
\newcommand{\tr}{\mathrm{tr}}
\newcommand{\cof}{\mathrm{cof}}
\newcommand{\sym}{\mathrm{sym}}
\renewcommand{\deg}{\mathrm{deg}}
\renewcommand{\div}{\mathrm{div}}
\newcommand{\inc}{\mathrm{inc}}
\newcommand{\dist}{\mathrm{dist}}
\newcommand\wk{\rightharpoonup}
\newcommand{\rtt}{\R^{3\times 3}}
\newcommand{\rtwtw}{\R^{2 \times 2}}

\newcommand*\closure[1]{\overline{#1}}

\newcommand{\lebt}{\mathscr{L}^2}
\newcommand{\leb}{\mathscr{L}^3}

\definecolor{bblue}{HTML}{3C3C9F}
\newcommand\BBB{\color{black}}

\newcommand\MMM{\color{black}}

\title [Dimension reduction for magnetoelastic plates]
{Linearized von K\'arm\'an theory for incompressible magnetoelastic plates}
\author[M. Bresciani] {Marco Bresciani} 
\address{Institute of Analysis and Scientific Computing, TU Wien, 
Wiedner Hauptstrasse 8-10, 1040 Vienna, Austria}
\email{marco.bresciani@tuwien.ac.at}
\subjclass[2010]{}

\begin{document} 
\vskip .2truecm
\begin{abstract}
We study the asymptotic behaviour, in the sense of $\Gamma$-convergence, of a thin incompressible magnetoelastic plate, as its thickness goes to zero. We focus on the linearized von K\'{a}rm\'{a}n regime. The model features a mixed Eulerian-Lagrangian formulation, as magnetizations are defined on the deformed configuration.
\end{abstract}
\maketitle

\normalem

\section{Introduction}
\label{sec:intro}

A crucial question in materials science is to characterize the unprecedented mechanical behavior of multifunctional materials, originating from the strong interplay between their elastic response and other effects, including polarizability or magnetizability, solid-solid phase change, or heat transfer. These couplings, often nonlinear, make these materials active, as comparably large strains can be induced via electromagnetic or thermal stimuli. This unique feature is at the basis of the vast range of applications of active materials for innovative devices, such as sensors, actuators, and semiconductors.

The modeling of multifunctional materials  is a very active field of research at the triple point between mathematics, physics, and materials science. In the case of \emph{small strains}, a variety of materials including magnetoelastics, active polymers, shape-memory alloys, and piezoelectrics have been extensively addressed and the corresponding mathematical theory is already quite developed \cite{mielke.roubicek}. However, most real-word phenomena involve \emph{large strains}, which cannot be effectively encompassed within the small-strain regime. At the purely mechanical level, passing from small to large strains requires leaving linear theories and resorting to nonlinear theories instead. In the case of active materials, the boost in complexity is even more evident, for their energetic formulations simultaneously involve both energy terms defined on the original stress-free configuration (\emph{Lagrangian}) and energy contributions arising in the deformed state (\emph{Eulerian}). 

The present contribution concerns the study of magnetoelastic materials. Such a material is characterized by a full  coupling between his mechanical properties and magnetization effects. This  feature is due to the presence of small magnetic domains in the material \cite{hubert.schafer} that, given an external magnetic field, tend to orientate themselves according to the latter, producing a magnetically induced deformation of the body. Conversely, any mechanical deformation modifies the orientation of the anisotropy directions of these domains, with the effect of changing the magnetic response of the material. We refer to \cite{brown} for the physical foundations of magnetoelasticity.

The mathematics of magnetoelasticity involves two quantities, the deformation $\boldsymbol{y}$ and the magnetization $\boldsymbol{m}$. If $\Omega \subset \R^3$ represents the reference configuration of the body, the deformation is given by a map $\boldsymbol{y}\colon \Omega \to \R^3$. In the case of small strains, the magnetization $\boldsymbol{m}$ is also defined on the reference configuration $\Omega$, as in micromagnetics \cite{desimone,gioia.james,james.kinderlehrer.2}. This approximation is valid if either the set $\Omega$ represents a ``very large'' sample of material or the deformation $\boldsymbol{y}$ is ``very close'' to the identity. Instead, in the case of large strains, the magnetization is defined on the deformed configuration, namely as a map $ \boldsymbol{m} \colon \boldsymbol{y}(\Omega)\to \R^3$. Thus, the deformation $\boldsymbol{y}$ is a \emph{Lagrangian variable}, while the magnetization $\boldsymbol{m}$ is an \emph{Eulerian variable}. In this last setting, the existence of energy minimizers have been proven first in \cite{rybka.luskin}, where a second-order regularization was introduced, and then in \cite{kruzik.stefanelli.zeman},  where the case of an incompressible material was considered. {\MMM Subsequently, the existence of minimizers under feasible assumptions have been proven in \cite{barchiesi.henao.moracorral}.}

In this paper, we investigate the problem of dimension reduction for thin magnetoelastic plates subjected to large deformations. Our goal is to identify an approximate two-dimensional description of the system, starting from a  three-dimensional one, as the thickness of the plate goes to zero. Such  effective  models are widely used in the applied sciences, since they significantly reduce the computational complexity while preserving the main features of the original three-dimensional structures. 

The problem of dimension reduction in the context of nonlinear elasticity has been a central topic of research in the last decades. In the calculus of variations, the approximation is understood in the sense of $\Gamma$-convergence \cite{braides,dalmaso} and usually relies on quantitative rigidity estimates \cite{friesecke.james.mueller1}. The case of plates has been extensively studied. Scaling the elastic energy by different powers of the thickness of the plate, a hierarchy of regimes and of corresponding limiting theories has been established \cite{friesecke.james.mueller2}. In particular, it has been shown in \cite{friesecke.james.mueller2} that, for a sufficiently small order of magnitude of the applied loads, one recovers in the limit the von K\'arm\'an model for plates, classically derived by means of formal asymptotic expansion and heuristic considerations \cite[Chapter 5]{ciarlet2}. In the context of micromagnetics, the analysis of the thin-film limit of magnetic plates has  been addressed in  \cite{carbou,gioia.james}. {\MMM We also mention the recent work \cite{davoli.difratta.praetorius.ruggeri}.}

In \cite{kruzik.stefanelli.zanini}, the dimension reduction of magnetoelastic plates in the Kirchoff-Love regime and the corresponding evolution are studied, in a purely Lagrangian setting, within the framework of linearized elasticity. 
The problem of dimension reduction for magnetoelastic plates undergoing large-strain deformations have been considered in \cite{liakhova}, under some a priori constraint on the jacobian of deformations, and subsequently numerically investigated in \cite{liakhova.luskin.zhang,luskin.zhang}.  The first rigorous derivation of a two-dimensional model,  starting from  the membrane regime,  has been obtained for non-simple materials in the recent contribution \cite{davoli.kruzik.piovano.stefanelli}.   

In the present work, we consider the same mixed Eulerian-Lagrangian variational formulation in \cite{kruzik.stefanelli.zeman}. Consider a plate $\Omega_h \coloneqq S \times (-h/2,h/2)$, with section $S \subset \R^2$  and thickness $h>0$, subjected to an elastic deformation $\boldsymbol{w}\colon \Omega_h \to \R^3$ and a magnetization $\boldsymbol{m}\colon \boldsymbol{w}(\Omega_h)\to \R^3$. As in \cite{kruzik.stefanelli.zeman}, we impose the constraint of \emph{magnetic saturation} which, up to normalization, reads $|\boldsymbol{m}|=1$  in $\boldsymbol{w}(\Omega_h)$. This is physically reasonable for sufficiently low constant temperature. The corresponding magnetoelastic energy is  described by the following functional:
\begin{equation}
\label{eqn:intro-Fh}
{\MMM \mathcal{E}_h}(\boldsymbol{w},\boldsymbol{m}) := \frac{1}{h^\beta} \int_{\Omega_h}W^\inc(\nabla \boldsymbol{w},\boldsymbol{m}\circ \boldsymbol{w})\,\d \boldsymbol{X}+{\MMM \alpha}\int_{\boldsymbol{w}(\Omega_h)}|\nabla\boldsymbol{m}|^2\,\d \boldsymbol{\xi}+\frac{1}{2}\int_{\R^3}|\nabla \psi_{\boldsymbol{m}}|^2\,\d \boldsymbol{\xi}.
\end{equation}
The first term represents the \emph{elastic energy}, which is rescaled according to the \emph{linearized von K\'arm\'an regime} \cite{friesecke.james.mueller2}. {\MMM More precisely, for deformations $\boldsymbol{w}\in W^{1,p}(\Omega_h;\R^3)$, we assume $\beta>2p$ and, as in \cite{kruzik.stefanelli.zeman,rybka.luskin}, we suppose $p>3$.} This {\MMM last} assumption is merely technical and ensures that every deformation $\boldsymbol{w}$ admits a representative which is continuous up to the boundary and, in particular, that  the deformed set $\boldsymbol{w}(\Omega_h)$ is defined without ambiguity.  {\MMM An alternative  choice could be to work  as in \cite{barchiesi.desimone} and consider $p=3$. In this case, deformations with strictly positive Jacobian still admit a continuous representative \cite[Theorem 5.14]{fonseca.gangbo}, which however, in general, cannot be extended  up  to the boundary.}  We also mention the setting in \cite{barchiesi.henao.moracorral}, where the authors assume $p>2$ and replace the the deformed set $\boldsymbol{w}(\Omega_h)$ by the topological image $\text{im}_T(\boldsymbol{w},\Omega_h)$ (see \eqref{eqn:def-config-reduced-domain} below). 
Unfortunately, no strategy seems to be available for the limiting case $p=2$, which is of particular  modeling interest.  The analysis presented in this paper represents, to the author's knowledge, the first $\Gamma$-convergence study of magnetoelastic plates in the linearized von K\'arm\'an regime: we thus tackle here the more regular case $p>3$. Further integrability assumptions will be the subject of forthcoming investigations.  

Following the modeling  approach  proposed in \cite{conti.dolzmann,li.chermisi}, the elastic energy density $W^\text{inc} \colon \rtt \times \S^2 \to \R$, where $\S^2$ denotes the unit sphere in $\R^3$, is defined by setting $W^\text{inc}(\boldsymbol{F},\boldsymbol{\nu})=W(\boldsymbol{F},\boldsymbol{\nu})$ if $\det \boldsymbol{F}= 1$ and $W^\text{inc}(\boldsymbol{F},\boldsymbol{\nu})=+\infty$ otherwise, where $\boldsymbol{F}\in \rtt$ and $\boldsymbol{\nu}\in \S^2$. This embodies the assumption of \emph{incompressibility} of the material.
The nonlinear magnetoelastic energy density $W \colon \rtt \times \S^2 \to \R$ satisfies frame-indifference and normalization hypotheses (see \eqref{eqn:frame_indifference}-\eqref{eqn:normalization} below) analogous to the classical ones in nonlinear elasticity, combined with growth conditions and regularity assumptions (see \eqref{eqn:mixed_growth}-\eqref{eqn:local-smooth} below) modeled on the ones usually considered in dimension reduction problems \cite{friesecke.james.mueller2}. In particular, we assume $W$ to have global $p$-growth, but quadratic growth close to the set of rotations. This is crucial in order to be able to perform the second-order approximation of $W$ that identifies the von K\'{a}rm\'{a}n model.
We do not assume any particular structure of the coupling, but we only require feasible compatibility conditions (see \eqref{eqn:coupling_C}-\eqref{eqn:coupling_omega} below).   

The second term in \eqref{eqn:intro-Fh}, represents the \emph{exchange energy}{\MMM, where $\alpha>0$ is the exchange constant}. This contribution is of Eulerian type and involves the gradient of magnetizations. Thus, magnetizations are Sobolev functions defined on the deformed domains.  We stress that the set $\boldsymbol{w}(\Omega_h)$  is not necessarily open if $\boldsymbol{w}$ is not an homeomorphism. Hence, we  replace it with a suitable open subset $\Omega_h^{\boldsymbol{w}}$ (see \eqref{eqn:def-config-reduced-domain} and Lemma \ref{lem:def-config-reduced-domain} below) and  we assume $\boldsymbol{m} \in W^{1,2}(\Omega_h^{\boldsymbol{w}};\S^2)$.

The last term in \eqref{eqn:intro-Fh} stands for the \emph{magnetostatic energy}. {\MMM This term usually comprises a  constant $\mu_0>0$, the vacuum permeability, that here, for simplicity, is assumed to be equal to one.} The function $\psi_{\boldsymbol{m}}\colon \R^3 \to \R$ represents the stray field {\MMM potential} which is defined as the weak solution, unique up to additive constants (see Lemma \ref{lem:maxwell} below), of the Maxwell equation:
\begin{equation}
    \label{eqn:intro-maxwell}
    \Delta \psi_{\boldsymbol{m}}=\text{div}(\chi_{\boldsymbol{w}(\Omega_h)}\boldsymbol{m})\:\:\text{in $\R^3$.}
\end{equation}
We mention that the magnetostatic energy usually involves additional terms such as the \emph{anisotropy energy} \cite{alouges.difratta,rybka.luskin} and the 
\emph{Dzyaloshinskii–Moriya interaction energy} \cite{davoli.difratta,davoli.difratta.praetorius.ruggeri} that, on first approximation, we are neglecting. Also, we are not considering the effect of {\MMM applied forces}, as well as the one given by the presence of an external magnetic field through the so-called \emph{Zeeman energy}.

Our main contributions  are  contained in Theorems \ref{thm:comp-lower-bound} and \ref{thm:optimality-lower-bound}, and characterize the limiting behaviour of the system, as the thickness of the plate goes to zero. 
The enunciation of these results requires the introduction of some notation, the specification of our assumptions, as well as rescaling, therefore we postpone it to Section \ref{sec:setting}.  We present a simplified unified statement below. Recall the definition of the functional {\MMM $\mathcal{E}_h$} in \eqref{eqn:intro-Fh}.

\begin{theorem*}[{\MMM Main results}]
\emph{ 
The asymptotic behaviour, as $h \to 0^+$, of the functionals {\MMM $\left(h^{-1} \mathcal{E}_h\right)$} is described, in the sense of $\Gamma$-convergence, by the functional:
\begin{equation}
    \label{eqn:intro-E}
    \begin{split}
	 E({\MMM \boldsymbol{u},}v,\boldsymbol{\lambda})&\coloneqq {\MMM \frac{1}{2}\int_S Q^\mathrm{inc}_2(\sym\nabla'\boldsymbol{u},\boldsymbol{\lambda})\,\d \boldsymbol{x}' +}\frac{1}{24}\int_S Q^\mathrm{inc}_2((\nabla')^2 v,\boldsymbol{\lambda})\,\d \boldsymbol{x}'\\
	 &\,+{\MMM \alpha} \int_S |\nabla' \boldsymbol{\lambda}|^2\,\d \boldsymbol{x}' + \frac{1}{2}\int_S |\lambda^3|^2\,\d \boldsymbol{x}',
    \end{split}
\end{equation}
defined {\MMM for $\boldsymbol{u}\in W^{1,2}(S;\R^2)$, $v \in W^{2,2}(S)$ and $\boldsymbol{\lambda}\in W^{1,2}(S;\S^2)$. Here, given $\boldsymbol{\nu}\in \S^2$, the function $Q^\mathrm{inc}_2(\cdot,\boldsymbol{\nu})$ is a quadratic form on $\rtwtw$ constructed from the second-order approximation of $W(\cdot,\boldsymbol{\nu})$  close to the identity,}  while $\nabla'$ denotes the gradient with respect to the variable $\boldsymbol{x}' \in S$.}
\end{theorem*}

We point out that Theorems \ref{thm:comp-lower-bound} and \ref{thm:optimality-lower-bound} extend the results previously obtained in the context of nonlinear elasticity and micromagnetics. More explicitly, in absence of magnetizations, we recover the theory of dimension reduction for incompressible elastic plates in the linearized von K\'arm\'an regime established in \cite{li.chermisi}, while, in the case of deformations given by the identity map,  our model reduces to  the thin film-limit of micromagnetic bodies presented in \cite{gioia.james}.

Note that the functionals {\MMM $\mathcal{E}_h$} and $E$ in \eqref{eqn:intro-Fh} and \eqref{eqn:intro-E}, respectively, have different domains of definition, therefore we cannot expect a proper $\Gamma$-convergence statement. However, in Corollary \ref{cor:gamma-convergence}, we show that Theorems \ref{thm:comp-lower-bound} and \ref{thm:optimality-lower-bound} can be rigorously reformulated within the framework of $\Gamma$-convergence. This  latter result is adapted from \cite{lewicka.mora.pakzad}. 

{\MMM Applied forces and external magnetic fields can be included in the analysis. While the energetic contribution given by mechanical forces is classically set on the reference configuration according to the assumption of \emph{dead loads}, the one determined by external magnetic fields, the so-called \emph{Zeeman energy}, is necessarily defined in the actual space. As a consequence of our main results, sequences of minimizers of the energy $\mathcal{E}_h$  in \eqref{eqn:intro-Fh}, augmented by the functional representing the applied loads, converge, in a suitable sense, to minimizers of the functional $E$ in \eqref{eqn:intro-E}, completed by the corresponding limiting functional. The precise statement is presented in Corollary \ref{cor:convergence-minimizers}. 
The latter is included here just to attest the significance of our main results and it is stated without proof. The convergence of minimizers will be treated in detail in a forthcoming publication.}

As customary for $\Gamma$-convergence, our main results can be  subdivided into three parts:  compactness of sequences of states with equibounded energies, identification of a common lower bound for the energy of converging sequences of states, and existence of recovery sequences for arbitrary limiting states. The first two parts are contained in Theorem \ref{thm:comp-lower-bound}, while the third one is given by Theorem \ref{thm:optimality-lower-bound}.

In analogy with \cite{friesecke.james.mueller2},  compactness is obtained up to composition with rigid motions. The proof of the compactness of deformations relies on the adaptation of standard techniques of dimension reduction in nonlinear elasticity \cite{friesecke.james.mueller2} to our growth assumption (see \eqref{eqn:mixed_growth} below). In this regime, the deformations $\boldsymbol{w}_h$  converge to the identity. The compactness of magnetizations, instead, is proved using an original approach, which constitutes the main novelty of the present contribution. Indeed, since the magnetizations $\boldsymbol{m}_h$ are defined on the deformed sets $\boldsymbol{w}_h(\Omega_h)$, which are  unknown, the analysis is quite delicate. Note that, in general, the sets $\boldsymbol{w}_h(\Omega_h)$ are not regular domains and hence usual operations like the extension to the whole space or the identification of traces on the boundary are not allowed for the magnetizations $\boldsymbol{m}_h$. Moreover, in our case, the uniform convergence techniques developed in \cite{kruzik.stefanelli.zeman,rybka.luskin} are not available, since the deformed sets are crushing onto the section $S$, which has Lebesgue measure zero. Our approach is based on careful considerations on the geometry of the deformed sets. Combining the uniform convergence estimate of  the deformations  $\boldsymbol{w}_h$ towards the identity map with some elementary properties of the topological degree, we prove that for every $h$ sufficiently small, up to rigid motions, the deformed sets $\boldsymbol{w}_h(\Omega_h)$ contain a cylinder of height of order $h$ whose section is obtained by shrinking $S$. On this cylinder, the compactness of magnetizations can be deduced by standard methods. The limiting object, locally identified by this procedure, turns out to be globally well-defined. The same approach is used to deduce the compactness of the compositions    
$\boldsymbol{m}_h \circ \boldsymbol{w}_h$, properly rescaled. This is another subtle point, since the magnetizations $\boldsymbol{m}_h$ are not necessarily continuous. The issue is overcome by considering the restrictions of magnetizations to the  previously determined cylinder, which is  a regular domain,  and  extending them to the whole space. Then, we obtain the desired convergence by exploiting some Lusin-type property of Sobolev maps (see Proposition \ref{eqn:lip-prop} below). These techniques require an extensive use of the area formula (see Proposition \ref{eqn:area-formula} below). {\MMM In particular, we need admissible deformations to be injective. This fact, combined with the incompressibility constraint, entails that admissible deformations are volume-preserving.}

The proof of the existence of a lower bound for the elastic energies is similar to the corresponding one in the case of elasticity \cite{friesecke.james.mueller2}, once the convergence of the magnetizations is established. Incompressibility is treated by adopting the same strategy in \cite{conti.dolzmann,li.chermisi}.
The lower bound for the exchange energies is obtained by considering a family of cylinders contained in the deformed sets that exhaust them, in the sense of measure. Concerning the magnetostatic energy, we employ our geometric considerations about the deformed domains to prove the convergence of the right-hand sides of the equations \eqref{eqn:intro-maxwell} determined by $\boldsymbol{w}_h$ and $\boldsymbol{m}_h$. Then, we adapt the results in \cite{gioia.james} to prove the compactness of the  corresponding solutions and the convergence of the magnetostatic energies. 

The existence of recovery sequences is proved in two steps. First, we construct the sequence of deformations following the ansatz in \cite{friesecke.james.mueller2} and we deal with the incompressibility constraint by means of the techniques developed in \cite{li.chermisi} (see also \cite{conti.dolzmann}). The resulting deformations are given by perturbations of the identity and, in turn, by a result in \cite{ciarlet}, they are globally injective. Subsequently, we construct the sequence of magnetizations, taking into account that these have to be defined on the corresponding deformed sets. The convergence of the elastic energies follows, once more, similarly to the classical case \cite{friesecke.james.mueller2}. The convergence of the exchange energies is straightforward and the one of the magnetostatic energies is deduced by  analogous arguments to the ones of the lower bound.

The  key novelty of the present contribution is to  move beyond the small-strain assumption and  tackle instead a mixed Eulerian-Lagrangian formulation. In the same direction, we mention again the recent paper \cite{davoli.kruzik.piovano.stefanelli}. Note that, unlike other contributions \cite{davoli.kruzik.piovano.stefanelli,rybka.luskin}, we do not consider any second-order gradient term in the energy functional. This is possible because of the peculiar  properties  of the von K\'{a}rm\'{a}n regimes, in which deformations converge to the identity. A central ingredient for our analysis is the fact that  the convergence rate is of order bigger than one, namely of order $\beta/p-1$, where, we recall, $\beta>2p$. {\MMM If $\beta\leq 2p$, the the inner approximation of  the deformed sets $\boldsymbol{w}(\Omega_h)$  by means of cylinders of height comparable with $h$  is no more directly applicable. Thus, the techniques presented here cannot be extended to other regimes. 

We mention that we recently extended our results to the compressible case. This will be addressed in a forthcoming paper.
}

The paper is structured as follows. In Section \ref{sec:prelim}, we recall preliminary results that will be instrumental for our arguments. In Section \ref{sec:setting}, we establish the precise setting of the problem and state our main results with their corollaries.  Theorems \ref{thm:comp-lower-bound} and \ref{thm:optimality-lower-bound} are proved in Sections \ref{sec:comp-lb} and \ref{sec:opt}, respectively. Finally, in Section \ref{sec:gamma}, we prove the rephrasing of our main results in the language of $\Gamma$-convergence.

\subsection*{Notation and conventions}

We will work mainly in the three dimensions. For every $\boldsymbol{a}\in \R^3$, we set $\boldsymbol{a}'\coloneqq(a^1,a^2)^\top \in \R^2$, so that $\boldsymbol{a}=((\boldsymbol{a}')^\top,a^3)^\top$. The null vector in $\R^3$ is denoted by $\boldsymbol{0}$, thus $\boldsymbol{0}'$ is the null vector in $\R^2$.
The same notation applies to space variables. We denote by $\nabla'$ the gradient with respect to the first two variables, namely $\nabla'\coloneqq(\partial_1,\partial_2)^\top$, with obvious extension in the case of vector-valued functions. {\MMM Accordingly, the divergence and the Laplace operator with respect to the first two variables are denoted by $\div'$ and $\Delta'$, respectively.}
Given $\boldsymbol{a},\boldsymbol{b}\in \R^3$, their tensor product is defined as the matrix $\boldsymbol{a}\otimes \boldsymbol{b}\coloneqq(a^i b^j)_{i,j=1,2,3}\in \rtt$. For every matrix $\boldsymbol{A}\in \rtt$, we define $\boldsymbol{A}''\coloneqq (A^{ij})_{i,j=1,2}\in \rtwtw$. We denote the identity matrix and the null matrix in $\rtt$ by $\boldsymbol{I}$ and $\boldsymbol{O}$, respectively. Thus, $\boldsymbol{I}''$ and $\boldsymbol{O}''$ stand for the corresponding matrices in $\rtwtw$.  The identity map on $\R^3$ is denoted by $\boldsymbol{id}$. 

The group of proper rotations in $\R^N$ is denoted by $SO(N)$. The set of symmetric and skew-symmetric matrices in $\R^{N \times N}$ are respectively given by $\mathrm{Sym}(N)$ and $\mathrm{Skew}(N)$. The Lebesgue measure on $\R^N$ is denoted by $\mathscr{L}^N$, and the characteristic function of sets $A \subset \R^N$ is denoted by $\chi_A$. The optimal Lipschitz constant of a Lipschitz function $\boldsymbol{v}$ is denoted by $\mathrm{Lip}(\boldsymbol{v})$. 
We use standard notation for Lebesgue and Sobolev spaces, i.e. $L^p$ and $W^{m,p}$, and their local counterparts. We will also consider functions in the space, sometimes named after Beppo Levi, given by $V^{1,2}(\R^N)\coloneqq \{\varphi \in L^2_\loc(\R^N):\;\nabla \varphi \in L^2(\R^N;\R^N)\}$. Given an open set $\Omega \subset \R^N$ and an embedded submanifold $\mathcal{M} \subset \R^M$, we define the space of manifold-valued Sobolev maps as $W^{m,p}(\Omega;\mathcal{M})\coloneqq \{\boldsymbol{v} \in W^{m,p}(\Omega;\R^M):\; \boldsymbol{v}(\boldsymbol{x}) \in \mathcal{M}\; \text{for a.e. $\boldsymbol{x}\in \Omega$}\}$. In the following, $\mathcal{M}$ is going to be either the unit sphere $\S^2 \coloneqq \{\boldsymbol{x}\in \R^3: |\boldsymbol{x}|=1\}$ in $\R^3$ or the  special orthogonal group $SO(3)\subset \rtt$.

We will make use of the Landau symbols `$o$' and `$O$'. When referred to vectors or matrices, these are to be understood with respect to the maximum of their components.
We will adopt the common convention of denoting by $C$ a positive constant that can change from line to line and that can be computed in terms of known quantities. Sometimes, we are going to underline its dependence on certain quantities using parentheses. Also, we will identify functions defined on the plane with functions defined on the three-dimensional space that are independent on the third variable. Finally, all statements involving $h$ without specifying the range for this parameter, are to be understood to hold for $h>0$  and sufficiently small. Moreover, in absence of any specification, convergences are intended up to subsequences and for $h \to 0^+$.    

\section{Preliminary results}
\label{sec:prelim}

We briefly recall some facts and  notions that are going to be used later.
Let $\Omega \subset \R^N$ be {\MMM a bounded Lipschitz domain} and consider a map $\boldsymbol{v} \in W^{1,p}(\Omega;\R^N)$ with $p>N$. Any such map admits a representative which is continuous up to the boundary and has the \emph{Lusin property} $(N)$ {\MMM\cite[Corollary 1]{marcus.mizel}}, i.e. maps sets of zero Lebesgue measure to set of zero Lebesgue measure. {\MMM\emph{Henceforth, for such maps, we will always tacitly consider this representative.}} If $\boldsymbol{v}$ satisfies $\det \nabla \boldsymbol{v}>0$ almost everywhere in $\Omega$, then it also has the \emph{Lusin property} $(N^{-1})$ {\MMM \cite[Theorem 5.32]{fonseca.gangbo}}, i.e. the preimage via $\boldsymbol{v}$ of any set of zero Lebesgue measure  has Lebesgue measure zero.

We will use the following version of the {\MMM area and change-of-variable formulas \cite[Theorem 2]{marcus.mizel}.}

{\MMM
\begin{proposition}[Area  and change-of-variable formulas]
\label{prop:area-formula}
Let $\Omega \subset \R^N$ be a bounded Lipschitz domain and let $\boldsymbol{v} \in W^{1,p}(\Omega;\R^N)$ with $p>N$. Then, for every measurable set $A \subset \Omega$ , the multiplicity function
$$\boldsymbol{\xi} \mapsto \iota(\boldsymbol{v},A,\boldsymbol{\xi}):=\#\{\boldsymbol{x} \in A:\;\boldsymbol{v}(\boldsymbol{x})=\boldsymbol{\xi}\}$$
is measurable and the following area formula holds:
\begin{equation}
    \label{eqn:area-formula}
    \int_A | \det \nabla \boldsymbol{v}(\boldsymbol{x})|\,\d\boldsymbol{x}=\int_{\R^N} \iota(\boldsymbol{v},A,\boldsymbol{\xi})\,\d\boldsymbol{\xi}.
\end{equation}
Moreover, for  every  $f\colon \boldsymbol{v}(A)\to \R$ such that either $\boldsymbol{x}\mapsto f(\boldsymbol{v}(\boldsymbol{x}))\,|\det \nabla \boldsymbol{v}(\boldsymbol{x})|$ is integrable over $A$ or $\boldsymbol{\xi}\mapsto f(\boldsymbol{\xi})\,\iota(\boldsymbol{v},A,\boldsymbol{\xi})$ is integrable over $\boldsymbol{v}(A)$, the following change-of-variable formula holds:
\begin{equation}
    \label{eqn:change-of-var}
    \int_A f (\boldsymbol{v}(\boldsymbol{x}))\,|\det \nabla  \boldsymbol{v}(\boldsymbol{x})|\,\d \boldsymbol{x}= \int_{\boldsymbol{v}(A)} f(\boldsymbol{\xi}) \,\iota(\boldsymbol{v},A,\boldsymbol{\xi})\,\d \boldsymbol{\xi}.
\end{equation}
\end{proposition}
}

{\MMM
Note that, since $\boldsymbol{v}$ has the Lusin property $(N)$, the image $\boldsymbol{v}(A)$ of every measurable set $A \subset \Omega$ is measurable \cite[Corollary 2]{marcus.mizel}. Also, the map $f$ has to be Borel-measurable in order to have that the composition $f \circ \boldsymbol{v}$ is measurable. If $\det \nabla \boldsymbol{v}>0$ almost evereywhere, so that $\boldsymbol{v}$ has the Lusin property $(N^{-1})$, then it sufficies that $f$ is measurable.}

We say that $\boldsymbol{v}$ is \emph{almost everywhere injective} if there exists a set $B \subset \Omega$ of zero Lebesgue measure such that the map $\boldsymbol{v} \restr{\Omega \setminus B}$ is injective. {\MMM In this case, for every $A \subset \Omega$ measurable, we have $\iota(\boldsymbol{v},A,\cdot)=\chi_{\boldsymbol{v}(A)}$ almost everywhere in $\R^3$.}

{\MMM
\begin{remark}[Ciarlet-Ne\v{c}as condition]
The notion of  almost everywhere injectivity and the area formula are linked by the famous \emph{Ciarlet-Ne\v{c}as condition} \cite{ciarlet.necas} which reads
\begin{equation}
	\label{eqn:CN-intro}
	\int_\Omega |\det \nabla \boldsymbol{v}|\,\d\boldsymbol{x} \leq \mathscr{L}^N(\boldsymbol{v}(\Omega)).
\end{equation}
Indeed, under the assumptions of Proposition \ref{prop:area-formula}, the following holds \cite[page 185]{ciarlet.necas}: if $\boldsymbol{v}$ is almost everywhere injective, then $\boldsymbol{v}$ satisfies the  \eqref{eqn:CN-intro}; viceversa, if $\boldsymbol{v}$ satisfies\eqref{eqn:CN-intro} and $\det \nabla \boldsymbol{v} \neq 0$ almost everywhere in $\Omega$, then $\boldsymbol{v}$ is almost everywhere injective. This equivalence is particularly important in connection with the stability of almost everywhere injectivity with respect to the weak convergence in $W^{1,p}(\Omega;\R^N)$. We refer to \cite{ciarlet.necas} for details. Note that the Ciarlet-Ne\v{c}as condition  \eqref{eqn:CN-intro} is preserved under weak convegence as a consequence of the weak continuity of the Jacobian determinant and the Morrey embedding. 
\end{remark}
}

In our analysis, we will  use the following Lusin-type property of Sobolev maps {\MMM \cite{acerbi.fusco}}, which allows to approximate them with Lipschitz maps.

\begin{proposition}[Lusin-type property of Sobolev maps]
\label{prop:lipschitz-truncation}
Let $\boldsymbol{v} \in W^{1,r}(\R^N;\R^N)$ with $1 \leq r< \infty$. Then, for every $\lambda>0$ there exists a {\MMM measurable} set $F_\lambda \subset \R^N$  such that $\boldsymbol{v}\restr{F_\lambda}\colon F_\lambda \to \R^N$ is Lipschitz continuous with $\mathrm{Lip}(\boldsymbol{v}\restr{F_\lambda}) \leq C(N,r)\lambda$ and we have
\begin{equation*}
\mathscr{L}^N(\R^N \setminus F_\lambda) \leq \frac{C(N,r)}{\lambda^r} \int_{\{|\nabla \boldsymbol{v}|>\lambda/2\}}|\nabla \boldsymbol{v}|^r\,\d \boldsymbol{x},
\end{equation*}
where the constant $C(N,r)>0$ depends only on $N$ and $r$.
\end{proposition}

We will also employ the seminal rigidity estimate by Friesecke, James and  M\"uller \cite[Theorem 3.1]{friesecke.james.mueller1}. The proof for $p=2$ is given in the original paper, for the adaptation of the latter to arbitrary $1<p<\infty$, we refer to \cite[Section 2.4]{conti.schweizer}.

\begin{proposition}[Rigidity estimate]
\label{prop:rigidity-estimate}
Let $\Omega \subset \R^N$ be a bounded Lipschitz domain, \BBB and let $1<p<\infty$. Then, for every $\boldsymbol{v} \in W^{1,p}(\Omega;\R^N)$ there exists a constant rotation $\boldsymbol{Q} \in SO(N)$ such that
\begin{equation}
    \label{eqn:rigidity-intro}
    \int_\Omega |\nabla \boldsymbol{v} - \boldsymbol{Q}|^p\,\d\boldsymbol{x} \leq C(\Omega,N,p) \int_\Omega \dist^p(\nabla \boldsymbol{v};SO(N))\,\d\boldsymbol{x},
\end{equation}
where the constant $C(\Omega,N,p)>0$ depends only on $\Omega$, $N$ and $p$.
\end{proposition}

\begin{remark}[Rigidity constant]
\label{rem:rigidity-constant}
The constant in the rigidity estimate can be chosen uniformly for  domains which are homeomorphic through translations and homotheties. {\MMM Namely, in \eqref{eqn:rigidity-intro}, we can assume that $C(\sigma\, \Omega + \boldsymbol{c},N,p)=C(\Omega,N,p)$ for every $\sigma>0$ and $\boldsymbol{c}\in \R^N$. To see this, given $\widetilde{\boldsymbol{v}} \in W^{1,p}(\sigma \,\Omega + \boldsymbol{c};\R^N)$,  it is sufficient to apply the rigidity estimate to  $\boldsymbol{v}\coloneqq \sigma^{-1}\widetilde{\boldsymbol{v}}(\sigma \cdot + \boldsymbol{c}) \in W^{1,p}(\Omega;\R^N)$ and to use the change-of-variable formula.  More generally, the rigidity constant can be chosen uniformly for  domains which are homeomorphic via Bilipschitz maps with uniformly controlled Lipschitz constants.}
\end{remark}

{\MMM
\begin{remark}[Constant rotation]
\label{rem:constant-rotation}
A careful inspection of the proof of Proposition \ref{prop:rigidity-estimate} shows that the constant rotation $\boldsymbol{Q}\in SO(N)$ doest not depend on the exponent $p$. Namely, for every $1 < q \leq p$ there holds
\begin{equation*}
	\int_\Omega |\nabla \boldsymbol{v} - \boldsymbol{Q}|^q\,\d\boldsymbol{x} \leq C(\Omega,N,q) \int_\Omega \dist^q(\nabla \boldsymbol{v};SO(N))\,\d\boldsymbol{x}.
\end{equation*}
\end{remark}
}

We  recall  some elementary facts about the topological degree of a continuous map \cite{fonseca.gangbo,deimling,outerelo.ruiz}. Let $\Omega \subset \R^N$ be open and bounded, and let $\boldsymbol{v} \in C^0(\closure{\Omega};\R^N)$. The topological degree of $\boldsymbol{v}$ on $\Omega$, also known as the Brouwer degree, is a map $\deg(\boldsymbol{v},\Omega,\cdot) \colon \R^N \setminus \boldsymbol{v}(\partial \Omega) \to \Z$ which can be defined axiomatically by means of the following properties \cite[page 39]{outerelo.ruiz}:
\begin{align*}
    \textbf{(normalization)}& \hspace{5mm} \deg(\boldsymbol{id}\restr{\Omega}, \Omega, \boldsymbol{\xi})=1 \hskip 3pt \text{for every $\boldsymbol{\xi} \in \Omega$;}\\
    \textbf{(additivity)}& \hspace{5 mm} \text{if $A_1,A_2\subset \Omega$ are open with $A_1 \cap A_2 = \emptyset$, then}\\
    & \hspace{5mm } \text{$\deg(\boldsymbol{v},A_1 \cup A_2,\boldsymbol{\xi})=\deg(\boldsymbol{v},A_1,\boldsymbol{\xi})+ \deg(\boldsymbol{v},A_2,\boldsymbol{\xi})$} \\
    & \hspace{5mm } \text{for every $\boldsymbol{\xi} \in \R^N \setminus \boldsymbol{v}(\partial A_1 \cup \partial A_2)$;}\\
    \textbf{(homotopy invariance)}& \hspace{5mm} \text{if $\boldsymbol{H} \in C^0([0,1]\times \closure{\Omega};\R^N)$ and $\boldsymbol{\gamma}\colon [0,1]\to \R^N$ satisfies $\boldsymbol{\gamma}(t) \notin \boldsymbol{H}(\{t\} \times \partial \Omega)$}\\
   & \hspace{5mm} \text{for every $0 \leq t \leq 1$, then $\deg(\boldsymbol{H}(0,\cdot),\Omega,\boldsymbol{\gamma(0)})=\deg(\boldsymbol{H}(t,\cdot),\Omega,\boldsymbol{\gamma}(t))$ for}\\
    & \hspace{5mm}  \text{every $0 \leq t \leq 1$};\\
    \textbf{(solvability)}& \hspace{5mm} \text{if $\deg(\boldsymbol{v},\Omega,\boldsymbol{\xi})\neq 0$ for some $\boldsymbol{\xi} \in \R^N \setminus \boldsymbol{v}(\partial \Omega)$, then there exists $\boldsymbol{x} \in \Omega$ such}\\
    & \hspace{5mm} \text{that $\boldsymbol{\xi}=\boldsymbol{v}(\boldsymbol{x})$}.
\end{align*}
Starting from these, several other properties can be deduced. We state the ones that are going to be used in the following:
\begin{align*}
    \textbf{(continuity)} \hspace{5mm}& \text{$\deg(\boldsymbol{v},\Omega,\cdot)$ is continuous on $\R^N \setminus \boldsymbol{v}(\partial \Omega)$};\\
    \textbf{(local constantness)} \hspace{5mm}& \text{$\deg(\boldsymbol{v},\Omega,\cdot)$ is constant on each connected component of $\R^N \setminus \boldsymbol{v}(\partial \Omega)$};\\
    \textbf{(stability)} \hspace{5mm}& \text{if $\widetilde{\boldsymbol{v}}\in C^0(\closure{\Omega};\R^N)$ satisfies $||\widetilde{\boldsymbol{v}}-\boldsymbol{v}||_{C^0(\closure{\Omega};\R^N)}<\dist(\boldsymbol{\xi};\boldsymbol{v}(\partial\Omega))$ for some}\\
    \hspace{5mm}& \text{$\boldsymbol{\xi} \in \R^N \setminus \boldsymbol{v}(\partial \Omega)$, then $\boldsymbol{\xi} \in \R^N \setminus \widetilde{\boldsymbol{v}}(\partial \Omega)$ and $\deg(\widetilde{\boldsymbol{v}},\Omega,\boldsymbol{\xi})=\deg(\boldsymbol{v},\Omega,\boldsymbol{\xi})$}.
\end{align*}
In the case of regular maps, the topological degree can be computed explicitly. Namely, if $\boldsymbol{v} \in C^1(\closure{\Omega};\R^N)$ and $\boldsymbol{\xi} \in \R^N \setminus \boldsymbol{v}(\partial \Omega)$ is a regular value, i.e. $\boldsymbol{\xi} \notin \boldsymbol{v}(\{{\MMM \det \nabla}\boldsymbol{v}=0\})$, then 
\begin{equation*}
    \deg(\boldsymbol{v},\Omega,\boldsymbol{\xi})=\sum_{\boldsymbol{x} \in \Omega:\;\boldsymbol{v}(\boldsymbol{x})=\boldsymbol{\xi}} \mathrm{sgn}({\MMM \det \nabla} \boldsymbol{v}(\boldsymbol{x})).
\end{equation*}
In particular, if $\boldsymbol{v}$ is injective, then for every $\boldsymbol{\xi} \in \boldsymbol{v}(\Omega)$ we obtain $\deg(\boldsymbol{v},\Omega,\boldsymbol{\xi}) \in \{-1,1\}$. Actually, this last property holds even if $\boldsymbol{v}$ is just continuous \cite[Theorem 3.35]{fonseca.gangbo}.

If $\boldsymbol{v}\in C^1(\closure{\Omega};\R^N)$ and $\boldsymbol{\xi} \in \R^N \setminus \boldsymbol{v}(\partial \Omega)$ is a regular value, then the degree can be also computed by means of the following integral formula
\begin{equation}
    \label{eqn:degree-integral-formula}
    \deg(\boldsymbol{v},\Omega,\boldsymbol{\xi})=\int_\Omega \psi \circ \boldsymbol{v} \,{\MMM \det \nabla}\boldsymbol{v}\,\d \boldsymbol{x},
\end{equation}
where $\psi \in C^\infty_c(\R^N)$ is any non-negative function  with integral equal to one and whose support is contained in the connected component $V$ of $\R^N \setminus \boldsymbol{v}(\partial \Omega)$ with $\boldsymbol{\xi}\in V$ \cite[Proposition 2.1]{deimling}. As already noted in \cite[page 317]{ball}, if $\Omega \subset \R^N$ is a bounded Lipschitz domain, then formula \eqref{eqn:degree-integral-formula} holds for $\boldsymbol{v} \in W^{1,p}(\Omega;\R^N)$ with $p>N$ and for every $\boldsymbol{\xi} \in \R^N \setminus \boldsymbol{v}(\partial \Omega)$ as well. 

We conclude this section by introducing two sets associated to the deformations. We refer to \cite{kroemer} for more details and for a comprehensive treatment of the topological properties of deformations in the context of nonlinear elasticity. Let $\Omega \subset \R^N$ be open and bounded, and let $\boldsymbol{v} \in C^0(\closure{\Omega};\R^N)$. The corresponding \emph{deformed configuration} and \emph{topological image} are respectively defined as
\begin{equation}
\label{eqn:def-config-reduced-domain}
    \Omega^{\boldsymbol{v}}\coloneqq \boldsymbol{v}(\Omega)\setminus \boldsymbol{v}(\partial \Omega), \hskip 14 pt \mathrm{im}_T(\boldsymbol{v},\Omega)\coloneqq \{\boldsymbol{\xi}\in \R^N\setminus\boldsymbol{v}(\partial \Omega):\;\deg(\boldsymbol{v},\Omega,\boldsymbol{\xi})\neq 0   \}.
\end{equation}
The first set was considered in \cite{davoli.kruzik.piovano.stefanelli,kruzik.stefanelli.zeman,rybka.luskin}. The second one was introduced in \cite{sverak} and then studied by several authors {\MMM \cite{barchiesi.henao.moracorral,bouchala.hencl.molchanova,henao.moracorral,mueller.spector,mueller.spector.tang,tang}}. For the class of deformations that we are going to consider, these two sets turn out to coincide. 

\begin{lemma}[Deformed configuration and topological image]
\label{lem:def-config-reduced-domain}
Let $\Omega \subset \R^N$ be a bounded Lipschitz domain and let $\boldsymbol{v} \in W^{1,p}(\Omega;\R^N)$ with $p>N$ {\MMM satisfy $\det \nabla \boldsymbol{v}>0$ almost everywhere in $\Omega$}. Then, the corresponding deformed configuration and topological image coincide, namely $\Omega^{\boldsymbol{v}}=\mathrm{im}_T(\boldsymbol{v},\Omega)$. In particular, both sets are open. Moreover, $\mathscr{L}^N(\Omega^{\boldsymbol{v}})=\mathscr{L}^N(\boldsymbol{v}(\Omega))$.
\end{lemma}
\begin{proof}
By the continuity property of the degree, $\mathrm{im}_T(\boldsymbol{v},\Omega)$ is open in $\R^N \setminus \boldsymbol{v}(\partial \Omega)$ and, in turn, in $\R^N$. By the solvability property of the degree,  $\mathrm{im}_T(\boldsymbol{v},\Omega) \subset \Omega^{\boldsymbol{v}}$. For the opposite inclusion, consider $\boldsymbol{\xi}_0 \in \Omega^{\boldsymbol{v}}$ and denote by $V$ the connected component of $\R^N \setminus \boldsymbol{v}(\partial \Omega)$ with $\boldsymbol{\xi}_0 \in V$. Since $\R^N \setminus \boldsymbol{v}(\partial \Omega)$ is open, so is $V$. In particular, $B(\boldsymbol{\xi}_0,\eps)\subset \subset V$ for some $\eps>0$. Consider any $\boldsymbol{x}_0 \in \Omega$ such that $\boldsymbol{v}(\boldsymbol{x}_0)=\boldsymbol{\xi}_0$. By the continuity of $\boldsymbol{v}$, there exists $\delta>0$ such that $B(\boldsymbol{x}_0,\delta)\subset \subset \Omega$ and $\boldsymbol{v}(B(\boldsymbol{x}_0,\delta))\subset  B(\boldsymbol{\xi}_0,\eps)$. {\MMM Let $\psi \in C^\infty_c(\R^N)$ be such that $\psi \geq 0$, $\mathrm{supp}\,\psi = \closure{B}(\boldsymbol{\xi}_0,\eps) \subset V$ and $\int_{\R^N} \psi\,\d\boldsymbol{\xi}=1$}. Applying \eqref{eqn:degree-integral-formula}, we compute
\begin{equation*}
    \deg(\boldsymbol{v},\Omega,\boldsymbol{\xi}_0)=\int_\Omega \psi \circ \boldsymbol{v}\,{\MMM \det \nabla} \boldsymbol{v}\,\d\boldsymbol{x} \geq \int_{B(\boldsymbol{x}_0,\delta)} \psi \circ \boldsymbol{v}\,{\MMM \det \nabla} \boldsymbol{v}\,\d\boldsymbol{x}>0,
\end{equation*}
where we used that $\psi \circ \boldsymbol{v}>0$ in $B(\boldsymbol{x}_0,\delta)$ and ${\MMM \det \nabla} \boldsymbol{v}>0$ almost everywhere in $\Omega$. Thus $\boldsymbol{\xi}_0 \in \mathrm{im}_T(\boldsymbol{v},\Omega)$, and, by the arbitrariness of $\boldsymbol{\xi}_0$, we deduce $\Omega^{\boldsymbol{v}} \subset \mathrm{im}_T(\boldsymbol{v},\Omega)$. Therefore $\Omega^{\boldsymbol{v}} = \mathrm{im}_T(\boldsymbol{v},\Omega)$ and, in particular, $\Omega^{\boldsymbol{v}}$ is open. Since $\Omega$ is a Lipschitz domain, $\mathscr{L}^N(\partial \Omega)=0$. Then,  $\mathscr{L}^N(\boldsymbol{v}(\partial \Omega))=0$ by the Lusin property $(N)$, and we conclude $\mathscr{L}^N(\Omega^{\boldsymbol{v}})=\mathscr{L}^N(\boldsymbol{v}(\Omega))$.
\end{proof}


\section{Setting of the problem and main results}
\label{sec:setting}

\subsection{Setting of the problem}
\label{subsec:setting}
Let $\Omega_h:=S \times hI \subset \R^3$ be the reference configuration of a thin magnetoelastic plate. Here, $S \subset \R^2$ is a bounded, connected, Lipschitz domain representing the section of the plate, $h>0$ specifies its thickness and $I:=(-1/2,1/2)$. 

We consider deformations $\boldsymbol{w} \in W^{1,p}(\Omega_h;\R^3)$ with $p>3$ that {\MMM satisfy $\det \nabla \boldsymbol{w}>0$ almost everywhere and are almost everywhere injective in $\Omega_h$}. Therefore,  we can assume that $\boldsymbol{w}$ is continuous up to the boundary and that it satisfies both Lusin properties $(N)$ and $(N^{-1})$ (see Section \ref{sec:prelim}). 
Magnetizations are given by maps $\boldsymbol{m} \in W^{1,2}(\Omega_h^{\boldsymbol{w}};\S^2)$, where $\Omega_h^{\boldsymbol{w}}$ is defined as in \eqref{eqn:def-config-reduced-domain}. Note that, by Lemma \ref{lem:def-config-reduced-domain},  the set $\Omega_h^{\boldsymbol{w}}$ is open and differs from $\boldsymbol{w}(\Omega_h)$ at most by a set of zero Lebesgue measure. 

Following \cite{kruzik.stefanelli.zeman}, the magnetoelastic energy associated to an elastic deformation $\boldsymbol{w}:\Omega_h\to \mathbb{R}^3$ and a magnetization $\boldsymbol{m}:\Omega_h^{\boldsymbol{w}}\to \mathbb{S}^2$ will be encoded by the following energy functional:
\begin{equation}
\label{eqn:energy_Fh}
{\MMM \mathcal{E}_h}(\boldsymbol{w},\boldsymbol{m}) := \frac{1}{h^\beta} \int_{\Omega_h}W^\inc(\nabla \boldsymbol{w},\boldsymbol{m}\circ \boldsymbol{w})\,\d \boldsymbol{X}+{\MMM \alpha}\int_{\Omega_h^{\boldsymbol{w}}}|\nabla\boldsymbol{m}|^2\,\d \boldsymbol{\xi}+\frac{1}{2}\int_{\R^3}|\nabla \psi_{\boldsymbol{m}}|^2\,\d \boldsymbol{\xi}.
\end{equation}
The first term in \eqref{eqn:energy_Fh} represents the \emph{elastic energy}. Note that, since $\boldsymbol{w}$ is continuous and $\boldsymbol{m}$ is measurable, their composition is measurable. Moreover, by the Lusin property $(N^{-1})$, this composition is well-defined, meaning that its equivalence class does not depend on the choice of the representative of $\boldsymbol{m}$. Here, we focus on the linearized von K\'arm\'an regime, where we adopt the terminology of \cite{friesecke.james.mueller2}. {\MMM More precisely, the elastic energy is rescaled by $h^\beta$ for some fixed $\beta>2p$}.

As in \cite{conti.dolzmann} and \cite{li.chermisi}, we enforce an  incompressibility constraint through the elastic energy density $W^\inc \colon \rtt \times \S^2 \to [0,+\infty]$ (the superscript `inc' stands for `incompressible') by setting
\begin{equation*}
    W^\inc(\boldsymbol{F},\boldsymbol{\nu})=
    \begin{cases}
        W(\boldsymbol{F},\boldsymbol{\nu}) & \text{if } \det \boldsymbol{F}=1,\\
        + \infty &  \text{if }\det \boldsymbol{F}\neq 1,
    \end{cases}
\end{equation*}
for all $\boldsymbol{F} \in \rtt$ and $\boldsymbol{\nu} \in \S^2$, where the map $W \colon \rtt \times \S^2 \to [0,+\infty]$ is a continuous nonlinear magnetoelastic energy density. In particular, we require the following : 
\begin{align}
	\textbf{(frame indifference)}& \hskip 12pt \text{$W(\boldsymbol{R}\boldsymbol{F},\boldsymbol{R}\,\boldsymbol{\nu})=0$ for every $\boldsymbol{R}\in SO(3)$, $\boldsymbol{F}\in \rtt$, $\boldsymbol{\nu}\in\S^2$;} \hspace*{10mm}\label{eqn:frame_indifference}\\
	\textbf{(normalization)}& \hskip 12pt \text{$W(\boldsymbol{I},\boldsymbol{\nu})=0$ for every $\boldsymbol{\nu}\in\S^2$;} \label{eqn:normalization}\\
	&\begin{aligned}
		\mathllap{\textbf{(growth)}} &\hskip 12pt \text{there exists a constant $C>0$ such that}  \\
		&\hspace*{15mm} W(\boldsymbol{F},\boldsymbol{\nu})\geq C\,\dist^2(\boldsymbol{F};SO(3))\vee \dist^p(\boldsymbol{F};SO(3))\\
		&\hskip 12pt \text{for every $\boldsymbol{F}\in \rtt$, $\boldsymbol{\nu}\in \S^2$;}
	\end{aligned} \label{eqn:mixed_growth}\\
	&\begin{aligned}
		\mathllap{\textbf{(regularity)}} &\hskip 12pt \text{there exists $\delta>0$ such that for all $\boldsymbol{\nu}\in \S^2$ the function $W(\cdot,\boldsymbol{\nu})$}\\
		&\hskip 12pt \text{ is of class $C^2$ on the set $\{\boldsymbol{F}\in \rtt:\: \dist(\boldsymbol{F};SO(3))<\delta\}.$} \label{eqn:local-smooth}
	\end{aligned}
\end{align}

In view of \eqref{eqn:local-smooth} and by the normalization condition \eqref{eqn:normalization}, we have the following second-order Taylor expansion centered at the identity
\begin{equation}
    \label{eqn:taylor-expansion}
    W(\boldsymbol{I}+\boldsymbol{G},\boldsymbol{\nu})=\frac{1}{2}Q_3(\boldsymbol{G},\boldsymbol{\nu})+\omega(\boldsymbol{G},\boldsymbol{\nu}),
\end{equation}
for every $\boldsymbol{G}\in \rtt$ with $|\boldsymbol{G}|<\delta$ and for every  $\boldsymbol{\nu}\in \S^2$.
In the equation above
\begin{equation}
\label{eqn:definition-Q3}
Q_3(\boldsymbol{G},\boldsymbol{\nu}) \coloneqq\mathbb{C}^{\boldsymbol{\nu}}\boldsymbol{G}:\boldsymbol{G},
\end{equation}
where the fourth-order tensor $\mathbb{C}^{\boldsymbol{\nu}}\in \R^{3 \times 3 \times 3 \times 3}$ is defined by
\begin{equation}
    \mathbb{C}^{\boldsymbol{\nu}} \coloneqq \partial^2_{\boldsymbol{F}} W(\boldsymbol{I},\boldsymbol{\nu}),
\end{equation}
and, for every $\boldsymbol{\nu} \in \mathbb{S}^2$, there holds $\omega(\boldsymbol{G},\boldsymbol{\nu})=o(|\boldsymbol{G}|^2)$, as $|\boldsymbol{G}|\to 0^+$.
Note that, by \eqref{eqn:normalization},  for every $\boldsymbol{\nu} \in \S^2$, the tensor $\mathbb{C}^{\boldsymbol{\nu}}$ is positive definite. Hence, so is the quadratic form $Q_3(\cdot,\boldsymbol{\nu})$ which, in turn, is convex. 

We assume $W$ to be such that the following two facts hold:
\begin{equation}
    \label{eqn:coupling_C}
    \text{the map $\boldsymbol{\nu}\mapsto \mathbb{C}^{\boldsymbol{\nu}}$ {\MMM from $\S^2$ to $\R^{3 \times 3 \times 3 \times 3}$} is continuous,}
\end{equation}
\begin{equation}
    \label{eqn:coupling_omega}
    \text{$\overline{\omega}(t)\coloneqq\sup \left \{\frac{\omega(\boldsymbol{G},\boldsymbol{\nu})}{|\boldsymbol{G}|^2}:\:\boldsymbol{G}\in \rtt, |\boldsymbol{G}|\leq t,\:\boldsymbol{\nu}\in \S^2 \right \}=o(t^2)$, as $t\to 0^+$.}
\end{equation} 

Frame-indifference and normalization hypotheses are standard in nonlinear elasticity. The assumption \eqref{eqn:frame_indifference} corresponds to the fact that the energetic description of the system does not depend on the position of the observer and it was already considered in \cite{james.kinderlehrer}. The normalization hypothesis \eqref{eqn:normalization} states that, for any fixed magnetization, the reference configuration is a natural state and, by frame indifference, any rigid motion does not increase the elastic energy. Growth conditions from below an regularity assumptions involving the distance from the set of rotations have been firstly considered in \cite{friesecke.james.mueller1,friesecke.james.mueller2} in the context of dimension reduction problems. Our mixed-growth assumption \eqref{eqn:mixed_growth} states that the energy density $W$ has global $p$-growth, but quadratic growth close to $SO(3)$. This fact supports the second-order approximation of $W$  in \eqref{eqn:taylor-expansion}. We require the function $W$ to be regular in the first argument in a uniform way with respect to the second one. Indeed, the  differentiability set of $W(\cdot,\boldsymbol{\nu})$, determined by the constant $\delta>0$ in \eqref{eqn:local-smooth}, is the same for every $\boldsymbol{\nu}\in \S^2$. Furthermore, the hypotheses \eqref{eqn:coupling_C} and \eqref{eqn:coupling_omega} on the coupling are assumed.

The second term in \eqref{eqn:energy_Fh} represents the \emph{exchange energy} {\MMM with the exchange constant $\alpha>0$.}
The third term in \eqref{eqn:energy_Fh} is given by the \emph{magnetostatic energy}. This last term involves the function $\psi_{\boldsymbol{m}} \colon \R^3 \to \R$, {\MMM the \emph{stray field potential}}, which is a weak solution of the Maxwell equation:
\begin{equation}
    \label{eqn:maxwell_w}
    \Delta \psi_{\boldsymbol{m}}=\div(\chi_{\Omega_h^{\boldsymbol{w}}}\boldsymbol{m}) \hskip 4pt \mbox{in} \hskip 3pt \R^3,
\end{equation}
where $\chi_{\Omega_h^{\boldsymbol{w}}}\boldsymbol{m}$ denotes the extension by zero of $\boldsymbol{m}$ to the whole space.
Such a weak solution exists and it is unique up to an additive constant (see Lemma \ref{lem:maxwell} below), so that the magnetostatic energy is well defined. 

As already mentioned in Section \ref{sec:intro}, in the present work we neglect other contributions in the magnetostatic energy, like the \emph{anisotropy energy} \cite{alouges.difratta,rybka.luskin} or the {\MMM \emph{Dzyaloshinskii–Moriya interaction energy} \cite{davoli.difratta,davoli.difratta.praetorius.ruggeri}}. 

\subsection{Change of variables and rescaling} 
We perform a change of variables in order to deal with energy functionals defined on a fixed domain. We introduce the two functions $\boldsymbol{z}_h$ and $\boldsymbol{z}_0$ defined on the whole space by setting $\boldsymbol{z}_h(\boldsymbol{x}) \coloneqq((\boldsymbol{x}')^\top,h x_3)^\top$ and $\boldsymbol{z}_0(\boldsymbol{x})\coloneqq((\boldsymbol{x}')^\top,0)^\top$  for every $\boldsymbol{x}\in \mathbb{R}^3$. Set $\Omega:= S \times I$. To any deformation $\boldsymbol{w}:\Omega_h\to \mathbb{R}^3$, we associate the map  $\boldsymbol{y}\coloneqq\boldsymbol{w}\circ \boldsymbol{z}_h|_{\Omega}$. Applying the change-of-variable formula, we obtain 
\begin{equation*}
    \frac{1}{h} {\MMM \mathcal{E}_h}(\boldsymbol{w},\boldsymbol{m})=\frac{1}{h^\beta}\int_\Omega W^\inc(\nabla_h\boldsymbol{y},\boldsymbol{m}\circ\boldsymbol{y})\,\d \boldsymbol{x}+ \frac{\alpha}{h}\int_{\Omega^{\boldsymbol{y}}} |\nabla \boldsymbol{m}|^2\,\d \boldsymbol{\xi}+\frac{1}{2h}\int_{\R^3}|\nabla\psi_{\boldsymbol{m}}|^2\,\d \boldsymbol{\xi},
\end{equation*}
where we define the scaled gradient as  $\nabla_h\coloneqq((\nabla')^\top,h^{-1}\partial_3)^\top$. In particular, recalling \eqref{eqn:def-config-reduced-domain}, we have $\Omega^{\boldsymbol{y}}=\Omega^{\boldsymbol{w}}_h$, so that the Maxwell equation \eqref{eqn:maxwell_w} can be trivially rewritten as
\begin{equation}
    \label{eqn:maxwell_y}
    \Delta \psi_{\boldsymbol{m}}=\div(\chi_{\Omega^{\boldsymbol{y}}}\boldsymbol{m}) \hskip 4pt \mbox{in} \hskip 3pt \R^3.
\end{equation}
Thus, we define the magnetoelastic energy functional as
\begin{equation}
    \label{eqn:energy_Eh}
    E_h(\boldsymbol{y},\boldsymbol{m}) \coloneqq \frac{1}{h^\beta}\int_\Omega W^\inc(\nabla_h\boldsymbol{y},\boldsymbol{m}\circ\boldsymbol{y})\,\d \boldsymbol{x} + \frac{{\MMM \alpha}}{h}\int_{\Omega^{\boldsymbol{y}}} |\nabla \boldsymbol{m}|^2\,\d \boldsymbol{\xi} + \frac{1}{2h}\int_{\R^3}|\nabla \psi_{\boldsymbol{m}}|^2\,\d \boldsymbol{\xi},
\end{equation}
where the function ${\psi}_{\boldsymbol{m}}$ is a weak solution of \eqref{eqn:maxwell_y}. We denote the three terms at the right-hand side of \eqref{eqn:energy_Eh} by $E^{\text{el}}_h(\boldsymbol{y},\boldsymbol{m})$, $E^{\text{exc}}_h(\boldsymbol{y},\boldsymbol{m})$ and $E^{\text{mag}}_h(\boldsymbol{y},\boldsymbol{m})$, respectively.

The class of admissible deformation is given by 
\begin{equation*}
    \label{eqn:classY}
    \mathcal{Y}:=\left \{ \boldsymbol{y} \in W^{1,p}(\Omega;\R^3): \; \mbox{${\MMM \det \nabla} \boldsymbol{y}>0$  a.e. in $\Omega$},\;{\MMM \text{$\boldsymbol{y}$ a.e. injective in $\Omega$}} \right \},
\end{equation*}
where $p>3$.
To such deformations we associate  corresponding magnetizations $\boldsymbol{m} \in W^{1,2}(\Omega^{\boldsymbol{y}};\S^2)$. Therefore the class of the admissible states is defined as
\begin{equation*}
    \mathcal{Q}:=\left \{(\boldsymbol{y},\boldsymbol{m}):\; \boldsymbol{y}\in \mathcal{Y},\;\boldsymbol{m}\in W^{1,2}(\Omega^{\boldsymbol{y}};\S^2) \right\}.
\end{equation*}

\begin{remark}[Invariance by rigid motions]
	\label{rem:invariance}
We will show in Propositions \ref{prop:comp-def}, \ref{prop:comp-mag} and \ref{prop:liminf-mag} that the energy in \eqref{eqn:energy_Eh} is invariant with respect to rigid motions. More precisely, given an admissible state $(\boldsymbol{y},\boldsymbol{m})\in \mathcal{Q}$ and a rigid motion $\boldsymbol{T}\colon \R^3 \to \R^3$ of the form $\boldsymbol{T}(\boldsymbol{\xi})\coloneqq \boldsymbol{R}\,\boldsymbol{\xi}+\boldsymbol{d}$ for every $\boldsymbol{\xi}\in \R^3$, where $\boldsymbol{R}\in SO(3)$ and $\boldsymbol{d}\in \R^3$, if we set $\widetilde{\boldsymbol{y}}\coloneqq\boldsymbol{T}\circ \boldsymbol{y}$ and $\widetilde{\boldsymbol{m}}\coloneqq \boldsymbol{R}\,\boldsymbol{m}\circ \boldsymbol{T}^{-1}$, then $(\widetilde{\boldsymbol{y}},\widetilde{\boldsymbol{m}})\in\mathcal{Q}$ and $E_h(\widetilde{\boldsymbol{y}},\widetilde{\boldsymbol{m}})=E_h(\boldsymbol{y},\boldsymbol{m})$. This property follows by \eqref{eqn:frame_indifference} and by the structure of the Maxwell equation.
\end{remark}

\subsection{Main results}
\label{subsec:main-results}
In order to state our two main results, we introduce the limiting energy functional. {\MMM Similarly to \cite{conti.dolzmann},} for every $\boldsymbol{H} \in \rtwtw$ and $\boldsymbol{\nu} \in \S^2$, we set
\begin{equation}
    \label{eqn:Q2-inc}
    \begin{split}
    Q_2^\inc(\boldsymbol{H},\boldsymbol{\nu})\coloneqq \min \Bigg \{  Q_3 &\left( \left(\renewcommand\arraystretch{1.2} 
    \begin{array}{@{}c|c@{}}   \boldsymbol{H}  & \boldsymbol{0}'\\ \hline  (\boldsymbol{0}')^{\top} & 0 \end{array} \right)+\boldsymbol{c}\otimes \boldsymbol{e}_3+\boldsymbol{e}_3 \otimes \boldsymbol{c}, \boldsymbol{\nu} \right): \:\:\boldsymbol{c} \in \R^3,\;\\
    \tr &\left(\left( \renewcommand\arraystretch{1.2}\begin{array}{@{}c|c@{}}   \boldsymbol{H}  & \boldsymbol{0}'\\ \hline  (\boldsymbol{0}')^\top & 0 \end{array} \right)+\boldsymbol{c} \otimes \boldsymbol{e}_3+\boldsymbol{e}_3 \otimes \boldsymbol{c}\right)=0 \Bigg \}.
    \end{split}
\end{equation}
The limiting functional is defined on {\MMM $W^{1,2}(S;\R^2)\times  W^{2,2}(S)\times W^{1,2}(S;\S^2)$} by
\begin{equation}
    \label{eqn:energy_E}
    	\begin{split}
    		E({\MMM\boldsymbol{u},}v,\boldsymbol{\lambda}) &\coloneqq
    		\frac{1}{2}\int_S Q_2^\inc(\sym \nabla'\boldsymbol{u},\boldsymbol{\lambda})\,\d\boldsymbol{x}'+\frac{1}{24}\int_S Q_2^\inc((\nabla')^2 v,\boldsymbol{\lambda})\,\d\boldsymbol{x}'\\
    		&\,+{\MMM \alpha}\int_S |\nabla'\boldsymbol{\lambda}|^2\,\d\boldsymbol{x}'+\frac{1}{2}\int_S |\lambda^3|^2\,\d \boldsymbol{x}'.	
    	\end{split}
\end{equation}
We denote the sum of the first two terms at the right-hand side of \eqref{eqn:energy_E} by $E^{\text{el}}(\boldsymbol{u},v,\boldsymbol{\lambda})$, and the remaining two terms by $E^{\text{exc}}(\boldsymbol{\lambda})$ and $E^{\text{mag}}(\boldsymbol{\lambda})$, respectively.

Our main results are given by Theorems \ref{thm:comp-lower-bound} and \ref{thm:optimality-lower-bound} and describe the asymptotic behaviour of the energies $E_h$ in \eqref{eqn:energy_Eh}, as $h \to 0^+$. {\MMM In order to state them, we  introduce the {\em horizonal} and {\em vertical averaged displacements} \cite{friesecke.james.mueller2}. For $\boldsymbol{y}\in \mathcal{Y}$, these are given by $\mathcal{U}_h(\boldsymbol{y})\colon S \to \R^2$ and $\mathcal{V}_h(\boldsymbol{y})\colon S \to \R$, respectively defined by setting
\begin{equation}
	\label{eqn:averaged-displacements}
	\mathcal{U}_h(\boldsymbol{y})(\boldsymbol{x}')\coloneqq \frac{1}{h^{\beta/2}}\int_I (\boldsymbol{y}'(\boldsymbol{x}',x_3)-\boldsymbol{x}')\,\d x_3, \quad \mathcal{V}_h(\boldsymbol{y})(\boldsymbol{x}')\coloneqq \frac{1}{h^{\beta/2-1}}\int_I y^3(\boldsymbol{x}',x_3)\,\d x_3,
\end{equation}
for every $\boldsymbol{x}' \in S$.}
The first theorem states that the functional $E$ in \eqref{eqn:energy_E} provides a lower bound for the asymptotic behavior of $E_h$ in
\eqref{eqn:energy_Eh}, as $h \to 0^+$.

\begin{theorem}[Compactness and lower bound]
\label{thm:comp-lower-bound}
Assume $p>3$ and $\beta>2p$. Suppose that $W$ satisfies \eqref{eqn:frame_indifference}-\eqref{eqn:local-smooth} and \eqref{eqn:coupling_C}-\eqref{eqn:coupling_omega}. Let $((\boldsymbol{y}_h,\boldsymbol{m}_h))_h \subset \mathcal{Q}$ be such that $E_h(\boldsymbol{y}_h,\boldsymbol{m}_h)\leq C$ for every $h>0$. Then, there exist a sequence of rotations $(\boldsymbol{Q}_h)\subset SO(3)$ and a sequence of translation vectors  $(\boldsymbol{c}_h)\subset \R^3$ such that, setting $\widetilde{\boldsymbol{y}}_h \coloneqq\boldsymbol{T}_h \circ \boldsymbol{y}_h$ and $\widetilde{\boldsymbol{m}}_h\coloneqq \boldsymbol{Q}_h^\top \boldsymbol{m}_h \circ \boldsymbol{T}_h^{-1}$, where $\boldsymbol{T}_h \colon \R^3 \to \R^3$ is the rigid motion defined by $\boldsymbol{T}_h(\boldsymbol{\xi})\coloneqq\boldsymbol{Q}_h^\top \boldsymbol{\xi}-\boldsymbol{c}_h$ for every $\boldsymbol{\xi} \in \R^3$, we have, up to subsequences, as $h \to 0^+$:
\begin{align}
        \label{eqn:lb-comp-u}
        &{\MMM\text{$\widetilde{\boldsymbol{u}}_h \coloneqq \mathcal{U}_h(\widetilde{\boldsymbol{y}}_h) \wk \boldsymbol{u}$ for some $\boldsymbol{u}\in W^{1,2}(S;\R^2$);}}\\
        \label{eqn:lb-comp-v}
        &\text{$\widetilde{v}_h\coloneqq \mathcal{V}_h(\widetilde{\boldsymbol{y}}_h) \to v$ in $W^{1,2}(S)$ for some $v \in W^{2,2}(S)$;}\\
        \label{eqn:lb-comp-my}
        &\text{$\widetilde{\boldsymbol{m}}_h \circ \widetilde{\boldsymbol{y}}_h \to \boldsymbol{\lambda}$ in $L^r(\Omega;\R^3)$ for every $1\leq r<\infty$ for some $\boldsymbol{\lambda} \in W^{1,2}(S;\S^2)$.}
    \end{align}
    Moreover, the following inequality holds:
    \begin{equation}
    \label{eqn:lb-liminf-ineq}
    E({\MMM \boldsymbol{u},}v,\boldsymbol{\lambda}) \leq \liminf\limits_{h \to 0^+}E_h(\boldsymbol{y}_h,\boldsymbol{m}_h) . 
    \end{equation}
\end{theorem}
The second theorem shows that the lower bound identified in Theorem \ref{thm:comp-lower-bound} is optimal.

\begin{theorem}[Optimality of the lower bound]
\label{thm:optimality-lower-bound}
Assume $p>3$ and $\beta>2p$. Suppose that $W$ satisfies \eqref{eqn:frame_indifference}-\eqref{eqn:local-smooth} and \eqref{eqn:coupling_C}-\eqref{eqn:coupling_omega}. Then, {\MMM for every $\boldsymbol{u}\in W^{1,2}(S;\R^2)$,}  $v\in W^{2,2}(S)$ and  $\boldsymbol{\lambda}\in W^{1,2}(S;\S^2)$, there exists $((\boldsymbol{y}_h,\boldsymbol{m}_h))_h \subset \mathcal{Q}$  such that, as $h \to 0^+$, we have:
    \begin{align}
        \label{eqn:opt-lb-u}
        &{\MMM\text{$\boldsymbol{u}_h\coloneqq \mathcal{U}_h(\boldsymbol{y}_h)\to \boldsymbol{u}$ in $W^{1,2}(S;\R^3)$;}}\\
        \label{eqn:opt-lb-v}
        &\text{$v_h\coloneqq \mathcal{V}_h(\boldsymbol{y}_h) \to v$ in $W^{1,2}(S)$;}\\
        \label{eqn:opt-lb-m-comp-y}
        &\text{${\boldsymbol{m}}_h \circ {\boldsymbol{y}}_h \to \boldsymbol{\lambda}$ in $L^r(\Omega;\R^3)$ for every $1 \leq r<\infty$.}
    \end{align}
    Moreover, the following equality holds:
    \begin{equation}
        \label{eqn:opt-lb-limit}
         E({\MMM \boldsymbol{u},}v,\boldsymbol{\lambda}) = \lim_{h \to 0^+} E_h({\boldsymbol{y}}_h,{\boldsymbol{m}}_h). 
    \end{equation}     
\end{theorem}

\subsection{$\boldsymbol{\Gamma}$-convergence}
\label{subsec:gamma-convergence}
The results of Theorems \ref{thm:comp-lower-bound} and \ref{thm:optimality-lower-bound} can be reformulated within the framework of $\Gamma$-convergence \cite{braides,dalmaso}. {\MMM Recall the notation introduced in \eqref{eqn:averaged-displacements} . Given an admissible state $(\boldsymbol{y},\boldsymbol{m}) \in \mathcal{Q}$ and $h>0$, we define  $\mathcal{M}_h(\boldsymbol{y},\boldsymbol{m})\colon \R^3 \to \R^3$ by setting} 
\begin{equation}
\label{eqn:gamma-notation-M}
\mathcal{M}_h(\boldsymbol{y},\boldsymbol{m})\coloneqq (\chi_{\Omega^{\boldsymbol{y}}}\boldsymbol{m})\circ \boldsymbol{z}_h.
\end{equation}

We introduce the functionals
\begin{equation*}
    {\MMM \mathcal{I}_h, \mathcal{I} \colon W^{1,p}(\Omega;\R^3) \times W^{1,2}(S;\R^2) \times W^{1,2}(S) \times L^2(\R^3;\R^3) \to \R \cup \{+\infty\}}
\end{equation*}
given by {\MMM
\begin{equation}
\label{eqn:gamma-energy-Eh}
\mathcal{I}_h(\boldsymbol{y},\boldsymbol{u},v,\boldsymbol{\mu})\coloneqq
\begin{cases}
E_h(\boldsymbol{y},\boldsymbol{m}) &\begin{aligned}
	&\text{if $\boldsymbol{u}=\mathcal{U}_h(\boldsymbol{y})$, $v=\mathcal{V}_h(\boldsymbol{y})$ \text{and} $\boldsymbol{\mu}=\mathcal{M}_h(\boldsymbol{y},\boldsymbol{m})$ }\\
	&\text{for some $\boldsymbol{m}\in W^{1,2}(\Omega^{\boldsymbol{y}};\S^2)$,}
\end{aligned}\\
+ \infty & \text{otherwise,}
\end{cases}
\end{equation}
} and {\MMM
\begin{equation}
\label{eqn:gamma-energy-E}
    \mathcal{I}(\boldsymbol{y},\boldsymbol{u},v,\boldsymbol{\mu})\coloneqq \begin{cases}
    E(\boldsymbol{u},v,\boldsymbol{\lambda}) & \begin{aligned}
    	&\text{if $\boldsymbol{y}=\boldsymbol{z}_0$, $v \in W^{2,2}(S)$ \text{and} $\boldsymbol{\mu}=\chi_{\Omega}\boldsymbol{\lambda}$}\\
    	&\text{for some $\boldsymbol{\lambda}\in W^{1,2}(S;\S^2)$,}
    \end{aligned}\\
    + \infty & \text{otherwise}.
    \end{cases}
\end{equation}
}
The following result is inspired by \cite[Corollary 2.4]{lewicka.mora.pakzad}.

\begin{corollary}[$\boldsymbol{\Gamma}$-convergence]
\label{cor:gamma-convergence}
{\MMM The functionals $(\mathcal{I}_h)$ $\Gamma$-converge to $\mathcal{I}$, as $h \to 0^+$, with respect  to the strong product topology. Namely, we have the following:
\begin{itemize}
    \item[\em (i)] \textbf{(Liminf inequality)} for every sequence $((\boldsymbol{y}_h,\boldsymbol{u}_h,v_h,\boldsymbol{\mu}_h))_h \subset W^{1,p}(\Omega;\R^3) \times W^{1,2}(S;\R^2)\times W^{1,2}(S)\times L^2(\R^3;\R^3)$ such that $(\boldsymbol{y}_h,\boldsymbol{u}_h,v_h,\boldsymbol{\mu}_h)\to(\boldsymbol{y},\boldsymbol{u},v,\boldsymbol{\mu})$ in $W^{1,p}(\Omega;\R^3) \times W^{1,2}(S;\R^2)\times W^{1,2}(S)\times L^2(\R^3;\R^3)$ for some $(\boldsymbol{y},\boldsymbol{u},v,\boldsymbol{\mu}) \in W^{1,p}(\Omega;\R^3) \times W^{1,2}(S;\R^2)\times W^{1,2}(S)\times L^2(\R^3;\R^3)$, we have:
    \begin{equation}
        \label{eqn:liminf-inequality}
        \mathcal{I}(\boldsymbol{y},\boldsymbol{u},v,\boldsymbol{\mu}) \leq \liminf_{h \to 0^+} \mathcal{I}_h(\boldsymbol{y}_h,\boldsymbol{u}_h,v_h,\boldsymbol{\mu}_h);
    \end{equation}
    \item[\em (ii)] \textbf{(Recovery sequence)} for every $(\boldsymbol{y},\boldsymbol{u},v,\boldsymbol{\mu}) \in W^{1,p}(\Omega;\R^3) \times W^{1,2}(S;\R^2)\times W^{1,2}(S)\times L^2(\R^3;\R^3)$ there exists a sequence  $((\boldsymbol{y}_h,\boldsymbol{u}_h,v_h,\boldsymbol{\mu}_h))_h \subset W^{1,p}(\Omega;\R^3) \times W^{1,2}(S;\R^2)\times W^{1,2}(S)\times L^2(\R^3;\R^3)$ such that $(\boldsymbol{y}_h,\boldsymbol{u}_h,v_h,\boldsymbol{\mu}_h)\to(\boldsymbol{y},\boldsymbol{u},v,\boldsymbol{\mu})$ in $W^{1,p}(\Omega;\R^3) \times W^{1,2}(S;\R^2)\times W^{1,2}(S)\times L^2(\R^3;\R^3)$ and we have:
    \begin{equation}
    \label{eqn:recovery-sequence}
        \mathcal{I}(\boldsymbol{y},\boldsymbol{u},v,\boldsymbol{\mu}) = \lim_{h \to 0^+} \mathcal{I}_h(\boldsymbol{y}_h,\boldsymbol{u}_h,v_h,\boldsymbol{\mu}_h).
    \end{equation}
\end{itemize}
}
The same result holds also with the weak product topology in place of the strong one.
\end{corollary}

However, it must be said that Corollary \ref{cor:gamma-convergence} provides less information than Theorems \ref{thm:comp-lower-bound} and \ref{thm:optimality-lower-bound}. Indeed, we cannot deduce the convergence of minimizers from the $\Gamma$-convergence result, since the sequence of functionals {\MMM $(\mathcal{I}_h)$} does not satisfy any suited coercivity assumption \cite[Theorem 1.21]{braides}.

{\MMM
\subsection{Convergence of almost minimizers}
\label{subsec:convergence-minimizers}
From Theorems \ref{thm:comp-lower-bound} and \ref{thm:optimality-lower-bound}, we deduce the convergence of almost minimizers of the energy $E_h$ in \eqref{eqn:energy_Eh} to minimizers of the energy $E$ in \eqref{eqn:energy_E} by straightforward arguments. However, the same conclusion becomes more delicate to be obtained when applied loads are included. In our case, these involve both Lagrangian and Eulerian energy contributions.

For $h>0$, we consider $\boldsymbol{f}_h \in L^2(S;\R^2)$, $g_h \in L^2(S)$ and $\boldsymbol{h}_h \in L^2(\R^3;\R^3)$ representing a tangential force, a normal force and an external magnetic field, respectively. We assume the following:
\begin{align}
	\label{eqn:load-f}
	&\text{$\frac{1}{h^{\beta/2}}\boldsymbol{f}_h \wk \boldsymbol{f}$ in $L^2(S;\R^2)$ for some $\boldsymbol{f}\in L^2(S;\R^2)$,}\\
	\label{eqn:load-g}
	&\text{$\frac{1}{h^{\beta/2+1}}g_h \wk g$ in $L^2(S)$ for some $g\in L^2(S)$,}\\
	\label{eqn:load-h}
	&\text{$\boldsymbol{h}_h\circ \boldsymbol{z}_h \wk \boldsymbol{h}$ in $L^2(\R^3;\R^3)$ for some $\boldsymbol{h}\in L^2(\R^2;\R^3)$.}
\end{align}
Note that, in \eqref{eqn:load-h}, we a priori assume that the limiting field $\boldsymbol{h}$ does not depend on the variable $x_3$.
The effect of the applied loads at the bulk level is described by the functional $L_h \colon \mathcal{Q}\to \R$ defined by
\begin{equation}
	\label{eqn:applied-loads-h}
	L_h(\boldsymbol{y},\boldsymbol{m})\coloneqq \frac{1}{h^\beta} \int_{\Omega} \boldsymbol{f}_h \cdot (\boldsymbol{y}'-\boldsymbol{x}')\,\d \boldsymbol{x}+\frac{1}{h^\beta}\int_\Omega g_h\,y^3\,\d\boldsymbol{x}+\frac{1}{h}\int_{\Omega^{\boldsymbol{y}}} \boldsymbol{h}_h \cdot \boldsymbol{m}\,\d\boldsymbol{\xi},
\end{equation}
so that the corresponding total energy $F_h \colon \mathcal{Q}\to \R$ reads
\begin{equation*}
	F_h(\boldsymbol{y},\boldsymbol{m})\coloneqq E_h(\boldsymbol{y},\boldsymbol{m})-L_h(\boldsymbol{y},\boldsymbol{m}).
\end{equation*}
Note that, while the first two terms in \eqref{eqn:applied-loads-h} are given by integrals on the reference configuration, the last term in \eqref{eqn:applied-loads-h} is written as an integral on the deformed set.

In the limit, the action of the applied loads is determined by the functional $L \colon W^{1,2}(S;\R^2)\times W^{2,2}(S)\times W^{1,2}(S;\S^2)\to \R$ given by
\begin{equation*}
	L(\boldsymbol{u},v,\boldsymbol{\lambda})\coloneqq \int_S \boldsymbol{f}\cdot \boldsymbol{u}\,\d\boldsymbol{x}'+\int_S g\,v\,\d\boldsymbol{x}'+\int_S \boldsymbol{h}\cdot \boldsymbol{\lambda}\d \boldsymbol{x}'
\end{equation*}
and the resulting energy $F \colon W^{1,2}(S;\R^2)\times W^{2,2}(S)\times W^{1,2}(S;\S^2) \to \R$ takes the form
\begin{equation*}
	F(\boldsymbol{u},v,\boldsymbol{\lambda})\coloneqq E(\boldsymbol{u},v,\boldsymbol{\lambda})-L(\boldsymbol{u},v,\boldsymbol{\lambda}).
\end{equation*}
Additionaly, we impose Dirichlet boundary conditions by rescriting ourselves to the class of admissible deformations
\begin{equation}
	\label{eqn:dirichlet-y}
	\mathcal{Y}_h \coloneqq \left \{\boldsymbol{y}\in \mathcal{Y}:\:\text{$\boldsymbol{y}=\boldsymbol{z}_h$ on $\partial S \times I$} \right \}
\end{equation}
and to the corresponding admissible states
\begin{equation}
	\label{eqn:dirichlet-q}
	\mathcal{Q}_h \coloneqq \left \{(\boldsymbol{y},\boldsymbol{m})\in \mathcal{Q}:\:\boldsymbol{y}\in \mathcal{Y}_h,\:\boldsymbol{m}\in W^{1,2}(\Omega^{\boldsymbol{y}};\S^2) \right \}.
\end{equation}
Consequently, the class of limiting admissible states reduces to the set
\begin{equation*}
	\begin{split}
		\mathcal{A}	&\coloneqq  \{(\boldsymbol{u},v,\boldsymbol{\lambda})\in W^{1,2}(S;\R^2)\times W^{2,2}(S)\times W^{1,2}(S;\S^2):\\
		&\hspace*{8,5mm}\text{$\boldsymbol{u}=\boldsymbol{0}'$ on $\partial S$, $v=0$ on $\partial S$, $\nabla ' v=\boldsymbol{0}'$ on $\partial S$}   \}.
	\end{split}
\end{equation*} 
The next result claims that, under the boundary conditions in \eqref{eqn:dirichlet-y}--\eqref{eqn:dirichlet-q}, almost minimizers of the energy $F_h$ actually converge to minimizers of the energy $F$. As a byproduct, we deduce the existence of minimizers for the functional $F$ in $\mathcal{A}$. 

\begin{corollary}[Convergence of almost minimizers]
	\label{cor:convergence-minimizers}
	Assume $p>3$ and $\beta>2p$. Suppose that $W$ satisfies \eqref{eqn:frame_indifference}--\eqref{eqn:local-smooth} and \eqref{eqn:coupling_C}--\eqref{eqn:coupling_omega}, and that the applied loads satisfy \eqref{eqn:load-f}-\eqref{eqn:load-h}. Let  $((\boldsymbol{y}_h,\boldsymbol{m}_h))\subset \mathcal{Q}$ be such that $(\boldsymbol{y}_h,\boldsymbol{m}_h) \in \mathcal{Q}_h$ for every $h>0$. Suppose that
	\begin{equation*}
		\lim_{h \to 0^+} \left \{F_h(\boldsymbol{q}_h)-\inf_{\mathcal{Q}_h}F_h  \right\}=0.
	\end{equation*}
	Then, we have, up to subsequences, as $h \to 0^+$:
	\begin{align}
		\label{eqn:cm-u}
		&\text{$\boldsymbol{u}_h\coloneqq \mathcal{U}_h(\boldsymbol{y}_h)\to \boldsymbol{u}$ in $W^{1,2}(S;\R^3)$;}\\
		\label{eqn:cm-v}
		&\text{$v_h\coloneqq \mathcal{V}_h(\boldsymbol{y}_h) \to v$ in $W^{1,2}(S)$;}\\
		\label{eqn:cm-my}
		&\text{${\boldsymbol{m}}_h \circ {\boldsymbol{y}}_h \to \boldsymbol{\lambda}$ in $L^r(\Omega;\R^3)$ for every $1\leq r <\infty$.}
	\end{align}
	Moreover, $(\boldsymbol{u},v,\boldsymbol{\lambda})\in \mathcal{A}$ is a minimizer of $F$ in $\mathcal{A}$. 
\end{corollary}
The proof of Corollary \ref{cor:convergence-minimizers} is omitted and will be included in a forthcoming paper. Here, we limit ourselves to provide some comments.

The main difficulty in proving Corollary \ref{cor:convergence-minimizers} is to deduce the bound $E^{\text{el}}_h(\boldsymbol{y}_h,\boldsymbol{m}_h)\leq C$ starting from $F_h(\boldsymbol{y}_h,\boldsymbol{m}_h)\leq C$. The situation is analogous to the one in \cite{lecumberry.mueller}, but our case is simpler as we are dealing with the linearized regime. In contrast with \eqref{eqn:lb-comp-u}-\eqref{eqn:lb-comp-my}, in \eqref{eqn:cm-u}-\eqref{eqn:cm-my} compactness is obtained without composing with rigid motions. This is possibile thanks to the Dirichlet boundary conditions in \eqref{eqn:dirichlet-y}-\eqref{eqn:dirichlet-q} by means of the techniques developed in \cite{lecumberry.mueller}. This improved compactness allows to pass to the limit in the last term in \eqref{eqn:applied-loads-h}, which is defined in the actual space, exploiting the convergence of $\mathcal{M}_h(\boldsymbol{y}_h,\boldsymbol{m}_h)$ defined in \eqref{eqn:gamma-notation-M}. The construction of the recovery sequence requires some care, since deformations have to satisfy both the incompressibility constraint and the boundary condition in \eqref{eqn:dirichlet-y}. In our case, where the boundary condition has to be imposed on the whole $\partial S \times I$, it is sufficient to consider limiting displacements $\boldsymbol{u}$ and $v$ having compact support. Then, the deformations of the recovery sequence constructed in the proof of Theorem \ref{thm:optimality-lower-bound} automatically satisfy the desired boundary condition. If, as in \cite{lecumberry.mueller}, the boundary condition is imposed on the set $\Gamma \times I$ for some $\Gamma \subset \partial S$, then the existence of a recovery sequence is a delicate issue. For more general boundary data, even the existence of incompressible deformations  belonging to the class in \eqref{eqn:dirichlet-y} is not guaranteed.

In analogy with \cite[Theorem 2]{friesecke.james.mueller2}, for normal forces only, i.e. for $\boldsymbol{f}_h=\boldsymbol{0}'$, having null average on $S$ and in absence of external magnetic fields, i.e. for $\boldsymbol{h}_h=\boldsymbol{0}$, it is possible to prove the convergence of minimizers of $F_h$  without imposing any boundary condition, but considering  a limiting functional that additionaly depends on a constant rotation. In this case, one can show that \eqref{eqn:cm-u} holds with $\boldsymbol{u}=\boldsymbol{0}'$. This could justify the choice, analogous to the one adopted in \cite{friesecke.james.mueller2}, to define the limiting energy $E$ in \eqref{eqn:energy_E} as a functional depending on $v$ and $\boldsymbol{\lambda}$ only.
}

\section{Compactness and lower bound}
\label{sec:comp-lb}

In this section we show that sequences of admissible states with equi-bounded energies enjoy suitable compactness properties and that the functional \eqref{eqn:energy_Eh} provides a lower bound for the asymptotic behavior of the magnetoelastic energies as $h\to 0^+$.

\subsection{Compactness}
\label{subsec:comp}

The proof of the compactness of deformations with  equi-bounded energies relies on the method of approximation by rotations in thin domains developed in \cite{friesecke.james.mueller2}, suitably adapted to our growth assumption \eqref{eqn:mixed_growth}. This  technique allows to approximate the scaled gradient of deformations with maps taking values in $SO(3)$. These are constructed explicitly, first locally, by means of mollification, and then globally, using  partitions of unity and projecting onto the set of proper rotations. Here, the rigidity estimate given by Proposition \ref{prop:rigidity-estimate} is fundamental in order to ensure that the maps obtained from the local approximation are in an appropriate neighborhood of the set of rotations for $h$ small enough. The invariance property of the rigidity constant (see Remark \ref{rem:rigidity-constant}) is also essential in the argument.

\begin{lemma}[Approximation by rotations]
\label{lem:approx-rot}
Let $\boldsymbol{y}\in W^{1,p}(\Omega;\R^3)$ and set $\boldsymbol{F}_h \coloneqq\nabla_h\boldsymbol{y}$. Define
\begin{equation*}
    \kappa_h\coloneqq\int_\Omega \dist^2(\boldsymbol{F}_h;SO(3))\vee \dist^p(\boldsymbol{F}_h;SO(3))\,\d \boldsymbol{x}.
\end{equation*}
Suppose that $\kappa_h/h^p\to 0$, as $h \to 0^+$. Then, there  exist a map $\boldsymbol{R}_h \in W^{1,p}(S;SO(3))\cap C^\infty (S;SO(3))$ and a constant rotation $\boldsymbol{Q}_h \in SO(3)$ such that, for $q \in \{2,p\}$ and $h \ll 1$, we have the following:
\begin{align}
    \label{eqn:approx-rot1} 
    \int_\Omega |\boldsymbol{F}_h-\boldsymbol{R}_h|^q\,\d\boldsymbol{x}\leq C \kappa_h, \hspace{2 cm} \int_S |\nabla '\boldsymbol{R}_h|^q\,\d\boldsymbol{x}'\leq C \frac{\kappa_h}{h^q},\\
    \label{eqn:approx-rot2} \int_S |\boldsymbol{R}_h-\boldsymbol{Q}_h|^q\,\d\boldsymbol{x}'\leq C \frac{\kappa_h}{h^q}, \hspace{1,6 cm} \int_\Omega |\boldsymbol{F}_h-\boldsymbol{Q}_h|^q\,\d\boldsymbol{x} \leq C \frac{\kappa_h}{h^q}.
\end{align}
\end{lemma}
\begin{proof}
Consider a tubular neighborhood $\mathcal{O}$ of $SO(3)$ in $\rtt$ such that the nearest-point projection $\boldsymbol{\Pi} \colon \mathcal{O}\to SO(3)$ is defined and smooth. 

Arguing as in \cite[Theorem 6]{friesecke.james.mueller2}, we construct a map $\widehat{\boldsymbol{R}}_h\in W^{1,p}(S:\rtt)\cap C^\infty(S;\rtt)$ such that, for $q \in \{2,p\}$and $h\ll1$, the following estimates hold:
\begin{equation}
    \label{eqn:R-tilda-2}
    \int_\Omega |\boldsymbol{F}_h-\widehat{\boldsymbol{R}}_h|^q\,\d\boldsymbol{x}\leq  C \,\kappa_h, \hspace{1,6cm} \int_S |\nabla' \widehat{\boldsymbol{R}}_h|^q\,\d\boldsymbol{x}'\leq C\, \frac{\kappa_h}{h^q}.
\end{equation}
{\MMM
To prove \eqref{eqn:R-tilda-2}, we apply the argument of \cite[Theorem 6]{friesecke.james.mueller2} twice, once with $q=2$ and again with $q=p$. At this point, we stress that the constant rotation obtained by applying the rigidity estimate given by Proposition \ref{prop:rigidity-estimate} locally in cubes of side of order $h$ (see \cite[formula $(68)$]{friesecke.james.mueller2}) does not depend on the value of the exponent $q$ (see Remark \ref{rem:constant-rotation}). Regardless, by construction, the resulting map $\widehat{\boldsymbol{R}}_h$ does not depend on the value of $q$.}

Note that, by \cite[formula $(77)$]{friesecke.james.mueller2}, we also have
\begin{equation}
    \label{eqn:R-tilda-distance}
    \sup_{\boldsymbol{x}' \in S} \dist^2(\widehat{\boldsymbol{R}}_h (\boldsymbol{x}') ;SO(3)) \leq C \frac{\kappa_h}{h^2}.
\end{equation}
From this, since the right hand side tends to zero, as $h \to 0^+$, we deduce that, for $h\ll1$, $\widehat{\boldsymbol{R}}_h(\boldsymbol{x}') \in \mathcal{O}$ for every $\boldsymbol{x}'\in S$. Hence, we can define the smooth map $\boldsymbol{R}_h \coloneqq \boldsymbol{\Pi}\circ \widehat{\boldsymbol{R}}_h$ which, by definition, takes values in $SO(3)$. 

By construction $|\widehat{\boldsymbol{R}}_h-\boldsymbol{R}_h|=\dist(\widehat{\boldsymbol{R}}_h;SO(3))$ in $S$. Using this fact and the firs estimate in \eqref{eqn:R-tilda-2}, we obtain
\begin{equation*}
    \begin{split}
        \int_\Omega |\boldsymbol{F}_h-\boldsymbol{R}_h|^q\,\d\boldsymbol{x}&\leq C \int_\Omega |\boldsymbol{F}_h-\widehat{\boldsymbol{R}}_h|^q\,\d\boldsymbol{x} + C \int_\Omega \dist^q(\widehat{\boldsymbol{R}}_h;SO(3))\,\d\boldsymbol{x}\\
        & \leq C \int_\Omega |\boldsymbol{F}_h-\widehat{\boldsymbol{R}}_h|^q\,\d\boldsymbol{x} + C \int_\Omega \dist^q(\boldsymbol{F}_h;SO(3))\,\d\boldsymbol{x} \leq C \kappa_h.
    \end{split}
\end{equation*}
Since $|\nabla \boldsymbol{\Pi}|\leq 1$ in $\rtt$, by the second estimate in  \eqref{eqn:R-tilda-2}, we have 
\begin{equation*}
    \int_S |\nabla' \boldsymbol{R}_h|^q \,\d\boldsymbol{x}' \leq \int_S |\nabla' \widehat{\boldsymbol{R}}_h|^q \,\d\boldsymbol{x}'\leq C \frac{\kappa_h}{h^q}.
\end{equation*}
Thus \eqref{eqn:approx-rot1} is proven.
Denote by $\boldsymbol{U}_h \in \rtt$ the integral average of $\boldsymbol{R}_h$ over $S$. From the second estimate in \eqref{eqn:approx-rot1}, using the Poincaré inequality, we obtain
\begin{equation}
\label{eqn:Rh-Uh}
    \int_S |\boldsymbol{R}_h-\boldsymbol{U}_h|^q\,\d\boldsymbol{x}'\leq C \int_S |\nabla'\boldsymbol{R}_h|^q\,\d\boldsymbol{x}'\leq C \frac{\kappa_h}{h^q}.
\end{equation}
This implies that there exists at least one point $\boldsymbol{x}_0'\in S$ such that $|\boldsymbol{R}_h(\boldsymbol{x}'_0)-\boldsymbol{U}_h|\leq C h^{-q}\kappa_h$.
Since $\boldsymbol{R}_h(\boldsymbol{x}'_0) \in SO(3)$, we deduce  
\begin{equation*}
    \dist(\boldsymbol{U}_h;SO(3))\leq |\boldsymbol{R}_h(\boldsymbol{x}'_0)-\boldsymbol{U}_h| \leq C h^{-q}\kappa_h,  
\end{equation*}
so that $\boldsymbol{U}_h \in \mathcal{O}$ for $h \ll1$, because the right-hand side tends to zero, as $h \to 0^+$. 
Define $\boldsymbol{Q}_h \coloneqq \boldsymbol{\Pi}(\boldsymbol{U}_h)$. Recalling that $\boldsymbol{\Pi}$ is Lipschitz continuous with $\mathrm{Lip}(\boldsymbol{\Pi})\leq 1$, we have $|\boldsymbol{R}_h-\boldsymbol{Q}_h|\leq|\boldsymbol{R}_h-\boldsymbol{U}_h|$ in $S$. Then, by \eqref{eqn:Rh-Uh}, we obtain
\begin{equation*}
    \int_S |\boldsymbol{R}_h-\boldsymbol{Q}_h|^q\,\d\boldsymbol{x}'\leq  \int_S |\boldsymbol{R}_h-\boldsymbol{U}_h|^q\,\d\boldsymbol{x}'  \leq C \frac{\kappa_h}{h^q},
\end{equation*}
which is the first estimate in \eqref{eqn:approx-rot2}. The latter combined with the first estimate in \eqref{eqn:approx-rot1} gives the second estimate in \eqref{eqn:approx-rot2}.
\end{proof}

In the next result we show how to use Lemma \ref{lem:approx-rot} to obtain the convergence, up to rigid motions, for sequences of deformations with equi-bounded energy. {\MMM Recall the definition of $E_h$ in \eqref{eqn:energy_Eh} and the notation introduced in \eqref{eqn:averaged-displacements}.}

\begin{proposition}[Compactness of deformations]
\label{prop:comp-def}
Let $((\boldsymbol{y}_h,\boldsymbol{m}_h))_h\subset \mathcal{Q}$ be such that $E_h(\boldsymbol{y}_h,\boldsymbol{m}_h)\leq C$ for every $h>0$. Then, there exist a sequence of rotations $(\boldsymbol{Q}_h)\subset SO(3)$ and a sequence of translations vectors $(\boldsymbol{c}_h)\subset \R^3$ such that, setting $\widetilde{\boldsymbol{y}}_h \coloneqq \boldsymbol{T}_h \circ \boldsymbol{y}_h$, where $\boldsymbol{T}_h \colon \R^3 \to \R^3$ is the rigid motion given by $\boldsymbol{T}_h(\boldsymbol{\xi})\coloneqq \boldsymbol{Q}_h^\top \boldsymbol{\xi}-\boldsymbol{c}_h$ for every $\boldsymbol{\xi}\in \R^3$, we have, up to subsequences, as $h \to 0^+$:
\begin{align}
    \label{eqn:comp-y}&\text{$\widetilde{\boldsymbol{y}}_h \to \boldsymbol{z}_0$  in $W^{1,p}(\Omega;\R^3)$;}\\
    \label{eqn:comp-u}&{\MMM
    	\text{$\widetilde{\boldsymbol{u}}_h\coloneqq \mathcal{U}_h(\widetilde{\boldsymbol{y}}_h) \wk \boldsymbol{u}$  weakly  in $W^{1,2}(S;\R^2)$ for some $\boldsymbol{u} \in W^{1,2}(S;\R^2)$;}
    }\\
    \label{eqn:comp-v}&{\MMM
    \text{$\widetilde{v}_h\coloneqq \mathcal{V}_h(\widetilde{\boldsymbol{y}}_h) \to v$ in $W^{1,2}(S)$ for some $v \in W^{2,2}(S)$.}
	}
\end{align}
\end{proposition}
\begin{proof}
We argue along the lines of \cite[Lemma 1]{friesecke.james.mueller2}. For every $h>0$, define $\boldsymbol{F}_h \coloneqq \nabla_h \boldsymbol{y}_h$. Due to the growth assumption \eqref{eqn:mixed_growth} and the scaling $\beta>2p$, we are in a position  to apply Lemma \ref{lem:approx-rot}. Thus, for every $h \ll 1$, we obtain a map $\boldsymbol{R}_h \in W^{1,p}(S;SO(3))$ and a constant rotation $\boldsymbol{Q}_h \in SO(3)$ such that \eqref{eqn:approx-rot1} and \eqref{eqn:approx-rot2} hold. Consider the rigid motion $\boldsymbol{T}_h$ given by $\boldsymbol{T}_h(\boldsymbol{\xi}) \coloneqq \boldsymbol{Q}_h^\top \boldsymbol{\xi}-\boldsymbol{c}_h$ for every $\boldsymbol{\xi} \in \R^3$, where $\boldsymbol{c}_h \in \R^3$. Define  $\widetilde{\boldsymbol{y}}_h \coloneqq \boldsymbol{T}_h\circ \boldsymbol{y}_h$, so that $\widetilde{\boldsymbol{F}}_h \coloneqq \nabla_h \widetilde{\boldsymbol{y}}_h=\boldsymbol{Q}_h^\top \boldsymbol{F}_h$, and set $\widetilde{\boldsymbol{R}}_h \coloneqq\boldsymbol{Q}_h^\top \boldsymbol{R}_h$.  For $q \in \{2,p\}$, by \eqref{eqn:approx-rot1} and \eqref{eqn:approx-rot2}, we have the following estimates:
\begin{align}
    \label{eqn:approx-rot-tilda1} \int_\Omega |\widetilde{\boldsymbol{F}}_h-\widetilde{\boldsymbol{R}}_h|^q\,\d\boldsymbol{x}\leq C h^\beta,& \hspace{1,6 cm} \int_S |\nabla '\widetilde{\boldsymbol{R}}_h|^q\,\d\boldsymbol{x}'\leq C h^{\beta-q},\\
    \label{eqn:approx-rot-tilda2} \int_S |\widetilde{\boldsymbol{R}}_h-\boldsymbol{I}|^q\,\d\boldsymbol{x}'\leq C h^{\beta -q},& \hspace{15 mm} \int_\Omega |\widetilde{\boldsymbol{F}}_h-\boldsymbol{I}|^q\,\d\boldsymbol{x} \leq C h^{\beta-q}.
\end{align}
We choose $\boldsymbol{c}_h$ so that $\widetilde{\boldsymbol{y}}_h - \boldsymbol{z}_h$ has zero average over $\Omega$, namely $\boldsymbol{c}_h \coloneqq \dashint_{\Omega}(\boldsymbol{Q}^\top_h \boldsymbol{y}_h-\boldsymbol{z}_h)\,\d\boldsymbol{x}$. By the second estimate in \eqref{eqn:approx-rot-tilda2} with $q=p$, we have
\begin{equation*}
    \int_\Omega |\nabla \widetilde{\boldsymbol{y}}_h-\nabla\boldsymbol{z}_h|^p\,\d\boldsymbol{x}\leq \int_\Omega |\nabla_h \widetilde{\boldsymbol{y}}_h-\nabla_h\boldsymbol{z}_h|^p\,\d\boldsymbol{x}=\int_\Omega |\widetilde{\boldsymbol{F}}_h-\boldsymbol{I}|^p\,\d\boldsymbol{x}\leq C h^{\beta - p},
\end{equation*}
and applying the Poincaré inequality we deduce that $\widetilde{\boldsymbol{y}}_h-\boldsymbol{z}_h\to 0$ in $W^{1,p}(\Omega;\R^3)$. Thus \eqref{eqn:comp-y} is proven.

We claim that there exists a map $\boldsymbol{A} \in W^{1,2}(S;\rtt)$ such that, up to subsequences, we have, as $h \to 0^+$:
\begin{align}
    \label{eqn:A_h-conv}
    \boldsymbol{A}_h \coloneqq \frac{1}{h^{\beta/2-1}}(\widetilde{\boldsymbol{R}}_h-\boldsymbol{I}) \wk \boldsymbol{A} \hskip 5pt \mbox{in $W^{1,2}(S;\rtt)$}; \hspace{7mm}\\
    \label{eqn:B_h-conv}
    \boldsymbol{B}_h\coloneqq\frac{1}{h^{\beta/2-1}}(\widetilde{\boldsymbol{F}}_h-\boldsymbol{I})\to \boldsymbol{A} \hskip 5pt \mbox{in $L^2(\Omega;\rtt)$}; \hspace{10mm}\\
    \label{eqn:C_h-conv}
    \boldsymbol{C}_h\coloneqq\frac{1}{h^{\beta-2}} \sym(\widetilde{\boldsymbol{R}}_h-\boldsymbol{I})\to \frac{1}{2}\boldsymbol{A}^2 \hskip 5pt \mbox{in $L^2(S;\rtt)$} \hspace{3mm}.
\end{align}
By the second estimate in \eqref{eqn:approx-rot-tilda1} and the first one in \eqref{eqn:approx-rot-tilda2}, both for $q=2$, we have $||\boldsymbol{A}_h||_{W^{1,2}(S;\R^3)}\leq C$ for every $h>0$ and hence \eqref{eqn:A_h-conv} holds. To see \eqref{eqn:B_h-conv}, note that $\boldsymbol{B}_h=h^{-\beta/2+1}(\widetilde{\boldsymbol{F}}_h-\widetilde{\boldsymbol{R}}_h)+\boldsymbol{A}_h$ and use the first estimate in \eqref{eqn:approx-rot-tilda1} with $q=2$ and \eqref{eqn:A_h-conv}. For the proof of \eqref{eqn:C_h-conv}, we first check that $-h^{\beta/2-1}\,\boldsymbol{A}_h^\top \boldsymbol{A}_h=\boldsymbol{A}_h^\top+\boldsymbol{A}_h$. After extracting a further subsequence, so that  $\boldsymbol{A}_h\to\boldsymbol{A}$ almost everywhere in $S$, we pass to the limit in the previous identity and we deduce that $\boldsymbol{A}\in \mathrm{Skew}(3)$ almost everywhere in $S$. Then, applying the Dominated convergence Theorem, we obtain
\begin{equation*}
    \boldsymbol{C}_h=\frac{1}{2h^{\beta/2-1}}(\boldsymbol{A}_h^\top+\boldsymbol{A}_h)=-\frac{1}{2}\boldsymbol{A}_h^\top \boldsymbol{A}_h \to -\frac{1}{2}\boldsymbol{A}^\top  \boldsymbol{A}=\frac{1}{2}\boldsymbol{A}^2 \hskip 5pt \mbox{in $L^2(S;\rtt)$}.
\end{equation*}
We can now prove the compactness of horizontal and vertical displacements. Note that, by the choice of $\boldsymbol{c}_h$, both $\widetilde{\boldsymbol{u}}_h$ and $\widetilde{v}_h$ have zero average on $S$. From the first estimate in \eqref{eqn:approx-rot-tilda1} with $q=2$ and \eqref{eqn:C_h-conv}, we obtain
\begin{equation*}
    \begin{split}
        ||\sym \nabla'\widetilde{{\boldsymbol{u}}}_h||_{L^2(S;\rtwtw)} & \leq \frac{1}{h^{\beta/2}} ||\sym \widetilde{\boldsymbol{F}}_h''-\boldsymbol{I}''||_{L^2(S;\rtwtw)} \\
        & \leq \frac{1}{h^{\beta/2}} ||\sym \,\widetilde{ \boldsymbol{F}}_h''-\sym\, \widetilde{\boldsymbol{R}}_h''||_{L^2(S;\rtwtw)} + \frac{1}{h^{\beta/2}} ||\sym\,\widetilde{\boldsymbol{R}}''_h-\boldsymbol{I}''||_{L^2(S;\rtwtw)}\\
        &\leq \frac{1}{h^{\beta/2}} ||\widetilde{\boldsymbol{F}}_h-\widetilde{\boldsymbol{R}}_h||_{L^2(\Omega;\rtt)} + h^{\beta/2-2} ||\boldsymbol{C}_h||_{L^2(S;\rtt)} \leq C.
    \end{split}
\end{equation*}
Therefore, applying the Korn and the Poincaré inequalities, we prove \eqref{eqn:comp-u}. Finally, we have 
\begin{equation*}
    \begin{split}
     ||\nabla'\widetilde{v}_h - (A^{31},A^{32})^\top||_{L^2(S;\R^2)}&\leq ||h^{-\beta/2+1}\nabla'\widetilde{y}_h^{\,3}-(A^{31},A^{32})^\top||_{L^2(S;\R^2)}\\   
    &=||(B_h^{31},B_h^{32})^\top-(A^{31},A^{32})^\top||_{L^2(S;\R^2)},
    \end{split}
\end{equation*}
and from \eqref{eqn:B_h-conv} we conclude that $\nabla' \widetilde{v}_h \to (A^{31},A^{32})^\top$ in $L^2(S;\R^2)$. Hence, applying the Poincaré inequality and recalling that $\boldsymbol{A}\in W^{1,2}(S;\rtt)$, we deduce \eqref{eqn:comp-v}.
\end{proof}

We fix some notation that will be used in the following. Given any $a>0$, we define the two sets $S^a \coloneqq\{\boldsymbol{x}'\in S:\;\dist(\boldsymbol{x}';\partial S)>a\}$ and $S^{-a} \coloneqq \{\boldsymbol{x'}\in \R^2:\,\dist(\boldsymbol{x}',S) < a\}$. Then, for every  $b>0$, we set $\Omega^a_b\coloneqq S^a \times bI$ and $\Omega^{-a}_b\coloneqq S^{-a} \times bI$.

Our next result guarantees the  compactness of magnetizations for sequences of admissible states with equi-bounded energies. As in Proposition \ref{prop:comp-def}, compactness is obtained up to rigid motions. {\MMM Recall the definition of $E_h$ in \eqref{eqn:energy_Eh} and the notation introduced in \eqref{eqn:gamma-notation-M}.}

\begin{proposition}[Compactness of magnetizations]
\label{prop:comp-mag}
Let $((\boldsymbol{y}_h,\boldsymbol{m}_h))_h \subset \mathcal{Q}$ be such that $E^{\text{\em el}}_h(\boldsymbol{y}_h,\boldsymbol{m}_h)+E^{\text{\em exc}}_h(\boldsymbol{y}_h,\boldsymbol{m}_h)\leq C$ for every $h>0$. Set $\widetilde{\boldsymbol{y}}_h \coloneqq \boldsymbol{T}_h \circ \boldsymbol{y}_h$ and $\widetilde{\boldsymbol{m}}_h \coloneqq \boldsymbol{Q}_h^\top \boldsymbol{m}_h \circ \boldsymbol{T}_h^{-1}$, where $\boldsymbol{Q}_h\in SO(3)$ and $\boldsymbol{T}_h\colon \R^3 \to \R^3$ are given by Proposition \ref{prop:comp-def}. Then, there exists a map $\boldsymbol{\lambda} \in W^{1,2}(S;\S^2)$ such that, up to subsequences, as $h \to 0^+$:
\begin{align}
    \label{eqn:conv-m-comp-y}
    &\text{$\widetilde{\boldsymbol{m}}_h \circ \widetilde{\boldsymbol{y}}_h \to \boldsymbol{\lambda}$ in $L^r(\Omega;\R^3)$ for every $1 \leq r < \infty$,}\\
    \label{eqn:conv-mu}
    &{\MMM
    	\text{$\widetilde{\boldsymbol{\mu}}_h\coloneqq \mathcal{M}_h(\widetilde{\boldsymbol{y}}_h,\widetilde{\boldsymbol{m}}_h)\to \chi_\Omega \boldsymbol{\lambda}$ in $L^r(\R^3;\R^3)$ for every $1 \leq r < \infty$.}
    }
\end{align}
\end{proposition}

\begin{proof}
For convenience of the reader, we subdivide the proof into four steps.

\textbf{Step 1.}
{\MMM By definition of $\mathcal{Y}$, the map $\boldsymbol{y}_h \in \mathcal{Y}$ is almost everywhere injective and, since $E^{\text{el}}_h(\boldsymbol{y}_h,\boldsymbol{m}_h)\leq C$, we deduce that $\det \nabla \boldsymbol{y}_h=h$ almost everywhere in $\Omega$. Consider $\widetilde{\boldsymbol{y}}_h\coloneqq \boldsymbol{T}_h \circ \boldsymbol{y}_h$, where $\boldsymbol{T}_h$ is the rigid motion given by Proposition \ref{prop:comp-def}. By construction, $\widetilde{\boldsymbol{y}}_h$ is also almost everywhere injective with $\det \nabla \widetilde{\boldsymbol{y}}_h=h$ almost everywhere in $\Omega$.}  

Using the second estimate in \eqref{eqn:approx-rot-tilda2} with $q=p$ and applying the Poincaré inequality and the Morrey embedding, we obtain the following key estimate:
\begin{equation}
    \label{eqn:key-estimate}
    ||\widetilde{\boldsymbol{y}}_h-\boldsymbol{z}_h||_{C^0(\closure{\Omega};\R^3)}\leq C \, h^{\beta/p-1}=:\delta_h.
\end{equation}
Note that $\delta_h /h \to 0$, as $h \to 0^+$, since $\beta>2p$. Fix $\eps>0$  and $0<\vth<1$. For every $\boldsymbol{\xi} \in \Omega^\eps_{\vth h}$, we have
\begin{equation*}
    \dist(\boldsymbol{\xi};\partial \Omega_h)=\dist(\boldsymbol{\xi}';\partial S) \wedge (h/2-|\xi_3|)>\eps \wedge ((1-\vth)/2)\,h. 
\end{equation*}
For $h \ll 1$, depending only on $\eps$ and $\vth$, the right-hand side is strictly bigger than $\delta_h$, so that
\begin{equation*}
    ||\widetilde{\boldsymbol{y}}_h-\boldsymbol{z}_h||_{C^0(\closure{\Omega};\R^3)}\leq \delta_h<\dist(\boldsymbol{\xi};\partial \Omega_h)=\dist(\boldsymbol{\xi};\boldsymbol{z}_h(\partial \Omega)).
\end{equation*}
By the stability property of the degree, we have {\MMM $\boldsymbol{\xi} \notin \widetilde{\boldsymbol{y}}_h(\partial \Omega)$} and  $\deg(\widetilde{\boldsymbol{y}}_h,\Omega,\boldsymbol{\xi})=\deg(\boldsymbol{z}_h,\Omega,\boldsymbol{\xi})=1$ so that, by the solvability property, we deduce $\boldsymbol{\xi} \in \widetilde{\boldsymbol{y}}_h(\Omega)$. Therefore, we have the following (see Figure \ref{fig:subcylinder}):
\begin{equation}
    \label{eqn:subcylinder}
    \forall\, \eps>0,\:\forall\, 0<\vth<1,\:\:\exists\, \bar{h}(\eps,\vth)>0:\:\:\forall\, 0<h\leq \bar{h}(\eps,\vth), \:\:\Omega^\eps_{\vth h} \subset \Omega^{\widetilde{\boldsymbol{y}}_h}.
\end{equation}

\begin{figure}
\centering
\begin{tikzpicture}
\node[anchor=south west,inner sep=0] at (0,0) {\includegraphics[width=\textwidth]{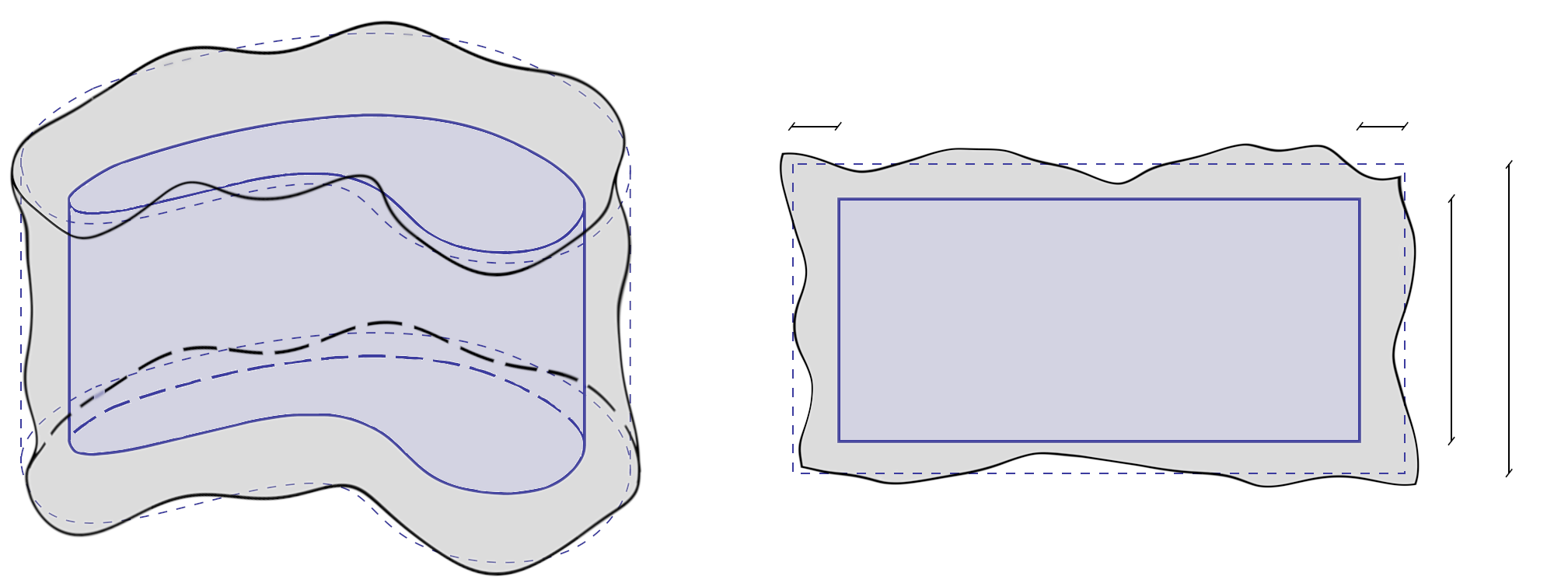}};
\node at (5.4,3.9) {\large \color{bblue}  ${\Omega^{\varepsilon}_{\vartheta h}}$};
\node at (5.75,4.75) {\large $\Omega^{\widetilde{\boldsymbol{y}}_h}$};
\node at (6.8,4.1) {\large \color{bblue} $\Omega_h$};
\node at (8.35,4.85) {$\varepsilon$};
\node at (14.15,4.85) {$\varepsilon$};
\node at (15.1,2.85) {$\vartheta h$};
\node at (15.6,2.85) {$h$};
\end{tikzpicture}
\caption{On the left, the cylinder $\Omega^\varepsilon_{\vartheta h}$ is contained in the deformed configuration $\Omega^{\widetilde{\boldsymbol{y}}_h}$. On the right, the corresponding section.}
\label{fig:subcylinder}
\end{figure}


\textbf{Step 2.} Define $\widetilde{\boldsymbol{m}}_h \coloneqq \boldsymbol{Q}_h^\top \boldsymbol{m}_h \circ \boldsymbol{T}_h^{-1}$, where $\boldsymbol{Q}_h\in SO(3)$ is given by Proposition \ref{prop:comp-def}. By the chain-rule, $\nabla \widetilde{\boldsymbol{m}}_h=\boldsymbol{Q}_h^\top \nabla \boldsymbol{m}_h \circ \boldsymbol{T}_h^{-1} \boldsymbol{Q}_h$, and applying the change-of-variable formula we deduce
\begin{equation}
\label{eqn:exc-invariance}
\begin{split}
E^\mathrm{exc}_h(\widetilde{\boldsymbol{y}}_h,\widetilde{\boldsymbol{m}}_h)&=\frac{\MMM \alpha}{h}\int_{\Omega^{\widetilde{\boldsymbol{y}}_h}} |\nabla \widetilde{\boldsymbol{m}}_h|^2\,\d\boldsymbol{\xi}=\frac{\MMM \alpha}{h}\int_{\boldsymbol{T}_h(\Omega^{\boldsymbol{y}_h})} |\nabla \boldsymbol{m}_h \circ \boldsymbol{T}_h^{-1}|^2 \,\d\boldsymbol{\xi}\\
&=    \frac{\MMM \alpha}{h}\int_{\Omega^{\boldsymbol{y}_h}} |\nabla \boldsymbol{m}_h|^2\,\d\boldsymbol{\xi}=E^{\mathrm{exc}}_h(\boldsymbol{y}_h,\boldsymbol{m}_h).
\end{split}
\end{equation}
Fix $\eps>0$ and $0<\vth<1$. By \eqref{eqn:subcylinder}, for $h\leq \bar{h}(\eps,\vth)$ the composition  $\widehat{\boldsymbol{m}}_h \coloneqq \widetilde{\boldsymbol{m}}_h \circ \boldsymbol{z}_h \restr{\Omega^\eps_\vth}$ is meaningful and defines a map in $W^{1,2}(\Omega^\eps_\vth;\S^2)$. Using \eqref{eqn:subcylinder} and the change-of-variable formula, we have
\begin{equation}
\label{eqn:cylinder-energy}
\begin{split}
C &\geq E_h^\mathrm{exc}(\widetilde{\boldsymbol{y}}_h,\widetilde{\boldsymbol{m}}_h)=\frac{\MMM \alpha}{h}\int_{\Omega^{\widetilde{\boldsymbol{y}}}_h}|\nabla \widetilde{\boldsymbol{m}}_h|^2\,\d \boldsymbol{\xi}\geq \frac{\MMM \alpha}{h}\int_{\Omega^\eps_{\vth h}} |\nabla \widetilde{\boldsymbol{m}}_h|^2\,\d \boldsymbol{\xi}\\
&\geq \frac{\MMM \alpha}{h}\int_{\Omega^\eps_{\vth h}}|\nabla_h \widehat{\boldsymbol{m}}_h \circ \boldsymbol{z}_h^{-1}|^2\,\d \boldsymbol{\xi}={\MMM \alpha}\int_{\Omega^\eps_\vth}|\nabla_h \widehat{\boldsymbol{m}}_h|^2\,\d\boldsymbol{x}.
\end{split}
\end{equation}
From this estimate we deduce two facts. First $||\nabla \widehat{\boldsymbol{m}}_h||_{L^2(\Omega^\eps_\vth;\R^3)} \leq ||\nabla_h \widehat{\boldsymbol{m}}_h||_{L^2(\Omega^\eps_\vth;\R^3)} \leq C$ so that, up to subsequences, we have $\widehat{\boldsymbol{m}}_h \wk \boldsymbol{\lambda}$ in $W^{1,2}(\Omega^\eps_\vth;\R^3)$ for some $\boldsymbol{\lambda} \in W^{1,2}(\Omega^\eps_\vth;\S^2)$. Second, $||\partial_3 \widehat{\boldsymbol{m}}_h/h||_{L^2(\Omega^\eps_\vth;\R^3)}\leq C$ and hence, up to subsequences, we have $\partial_3\widehat{\boldsymbol{m}}_h \to 0$ in $L^2(\Omega^\eps_\vth;\R^3)$ and $\partial_3\widehat{\boldsymbol{m}}_h/h \wk {\MMM\boldsymbol{k}}$ in $L^2(\Omega^\eps_\vth;\R^3)$  for some $\boldsymbol{k}\in L^2(\Omega^\eps_\vth;\R^3)$. In particular, $\boldsymbol{\lambda}$ does not depend on the variable $x_3$. In both cases, since the bounds in \eqref{eqn:cylinder-energy} are uniform in $ \eps$ and $\vth$, the limits depend neither on $\eps$ nor on $\vth$. Thus, $\boldsymbol{\lambda}\in W^{1,2}_\loc(S;\S^2)$ and ${\MMM \boldsymbol{k}}\in L^2_\loc(\Omega;\R^3)$. Moreover, by lower semicontinuity, from \eqref{eqn:cylinder-energy} we obtain
\begin{equation}
    \label{eqn:comp-m-comp-y-0}
    C\geq \liminf_{h \to 0^+} \int_{\Omega^\eps_\vth}|\nabla_h \widehat{\boldsymbol{m}}_h|^2\,\d\boldsymbol{x} \geq \int_{\Omega^\eps_\vth} |\nabla'\boldsymbol{\lambda}|^2\,\d\boldsymbol{x}' + \int_{\Omega^\eps_\vth} |{\MMM\boldsymbol{k}}|^2\,\d\boldsymbol{x} \geq \vth  \int_{S^\eps}|\nabla'\boldsymbol{\lambda}|^2\,\d\boldsymbol{x}',
\end{equation}
and letting $\eps \to 0^+$ and $\vth \to 1^-$ we conclude that $\nabla' \boldsymbol{\lambda} \in L^2(S;\R^2)$, so that $\boldsymbol{\lambda} \in W^{1,2}(S;\S^2)$.

\textbf{Step 3.} In this step, we prove \eqref{eqn:conv-m-comp-y}. First, we show that $\widetilde{\boldsymbol{m}}_h \circ \widetilde{\boldsymbol{y}}_h \wk \boldsymbol{\lambda}$ in $L^2_\loc(\Omega;\R^3)$. Consider a test function $\boldsymbol{\varphi} \in L^2(\Omega';\R^3)$ where $\Omega'\subset \subset \Omega$ is measurable. Choose $\eps\ll1$ and $\vth$ sufficiently close to one  in order to have $\Omega'\subset \Omega^\eps_\vth$, and consider $h \leq \bar{h}(\eps,\vth)$. Denote by $\overline{\boldsymbol{\varphi}}$ the extension of $\boldsymbol{\varphi}$ by zero on $\Omega^\eps_\vth \setminus \Omega'$. We compute
\begin{equation}
\label{eqn:comp-m-comp-y-1}
\int_{\Omega'} (\widetilde{\boldsymbol{m}}_h \circ \widetilde{\boldsymbol{y}}_h-\boldsymbol{\lambda})\,\boldsymbol{\varphi}\,\d\boldsymbol{x} = \int_{\Omega^\eps_\vth} (\widetilde{\boldsymbol{m}}_h \circ \widetilde{\boldsymbol{y}}_h-\widehat{\boldsymbol{m}}_h)\,\overline{\boldsymbol{\varphi}}\,\d\boldsymbol{x}+\int_{\Omega^\eps_\vth} (\widehat{\boldsymbol{m}}_h-\boldsymbol{\lambda})\,\overline{\boldsymbol{\varphi}}\,\d\boldsymbol{x}.
\end{equation}
Since $\widehat{\boldsymbol{m}}_h \wk \boldsymbol{\lambda}$ in $L^2(\Omega^\eps_\vth;\R^3)$, the second integral on the right-hand side tends to zero. For the first integral, we set $A^\eps_{\vth,h} \coloneqq \widetilde{\boldsymbol{y}}_h^{-1}(\Omega^\eps_{\vth h})$ and we split it as
\begin{equation}
\label{eqn:comp-m-comp-y-2}
\int_{\Omega^\eps_\vth} (\widetilde{\boldsymbol{m}}_h \circ \widetilde{\boldsymbol{y}}_h-\widehat{\boldsymbol{m}}_h)\,\overline{\boldsymbol{\varphi}}\,\d\boldsymbol{x}=\int_{\Omega^\eps_\vth \cap A^\eps_{\vth,h}} (\widetilde{\boldsymbol{m}}_h \circ \widetilde{\boldsymbol{y}}_h-\widehat{\boldsymbol{m}}_h)\,\overline{\boldsymbol{\varphi}}\,\d\boldsymbol{x} + \int_{\Omega^\eps_\vth \setminus A^\eps_{\vth,h}} (\widetilde{\boldsymbol{m}}_h \circ \widetilde{\boldsymbol{y}}_h-\widehat{\boldsymbol{m}}_h)\,\overline{\boldsymbol{\varphi}}\,\d\boldsymbol{x}.
\end{equation}

Recall that $\widetilde{\boldsymbol{y}}_h$ is almost everywhere injective with $\det \nabla \widetilde{\boldsymbol{y}}_h=h$ almost averywhere in $\Omega$. By the area formula, we have
\begin{equation*}
    \leb(\Omega^\eps_{\vth h})=\leb(\widetilde{\boldsymbol{y}}_h(A^\eps_{\vth,h}))=\int_{A^\eps_{\vth,h}} \det \nabla \widetilde{\boldsymbol{y}}_h\,\d\boldsymbol{x}=h\, \leb(A^\eps_{\vth,h}),
\end{equation*}
from which we deduce
\begin{equation}
    \leb(A^\eps_{\vth,h})=h^{-1}\leb(\Omega^\eps_{\vth h})=\vth \lebt(S^\eps).
\end{equation}
Since both  sequences $\widetilde{\boldsymbol{m}}_h$ and $\widehat{\boldsymbol{m}}_h$ take values in $\mathbb{S}^2$, the second integral on the right-hand side of \eqref{eqn:comp-m-comp-y-2} can be estimated in the following way
\begin{equation}
\label{eqn:comp-m-comp-y-3}
\begin{split}
\int_{\Omega^\eps_\vth \setminus A^\eps_{\vth,h}} (\widetilde{\boldsymbol{m}}_h \circ \widetilde{\boldsymbol{y}}_h-\widehat{\boldsymbol{m}}_h)\,\overline{\boldsymbol{\varphi}}\,\d\boldsymbol{x} &\leq 2\, ||\boldsymbol{\varphi}||_{L^2(\Omega';\R^3)} \left (\leb(\Omega^\eps_\vth \setminus A^\eps_{\vth,h}) \right)^{1/2}\\
& \leq 2\, ||\boldsymbol{\varphi}||_{L^2(\Omega';\R^3)} \left (\leb(\Omega \setminus A^\eps_{\vth,h}) \right)^{1/2}\\
& \leq 2  \,||\boldsymbol{\varphi}||_{L^2(\Omega';\R^3)} \left (\lebt(S)-\vth \lebt(S_\eps) \right )^{1/2}. 
\end{split}
\end{equation}

To estimate the first integral on the right-hand side of \eqref{eqn:comp-m-comp-y-2}, we proceed as follows. Note that if $\boldsymbol{x}\in \Omega^\eps_\vth \cap A^\eps_{\vth,h}$, then $\widetilde{\boldsymbol{y}}_h(\boldsymbol{x}) \in \Omega^\eps_{\vth h}$ and $\boldsymbol{z}_h(\boldsymbol{x}) \in \Omega^\eps_{\vth h}$. Consider the map $\widehat{\boldsymbol{m}}_h\in W^{1,2}(\Omega^\eps_{\vth h};\S^2)$. Since $\Omega^\eps_{\vth}$ is a Lipschitz domain, at least for $\eps \ll 1$ and $\vth$ sufficiently close to one, this map  admits an extension $\widehat{\boldsymbol{M}}_h \in W^{1,2}(\R^3;\R^3)$, which may depend on $\eps$ and $\vth$, and satisfies
\begin{equation*}
    ||\widehat{\boldsymbol{M}}_h||_{W^{1,2}(\R^3;\R^3)} \leq C(\eps,\vth) ||\widehat{\boldsymbol{m}}_h||_{W^{1,2}(\Omega^\eps_\vth;\R^3)}.
\end{equation*}
In particular,  we have
\begin{equation}
    \label{eqn:grad-M-hat}
    \int_{\R^3} |\nabla \widehat{\boldsymbol{M}}_h|^2\,\d\boldsymbol{x}\leq C(\eps,\vth) \left ( \int_{\Omega^\eps_\vth} |\widehat{\boldsymbol{m}}_h|^2\,\d\boldsymbol{x} +\int_{\Omega^\eps_\vth} |\nabla \widehat{\boldsymbol{m}}_h|^2 \,\d\boldsymbol{x} \right)\leq C(\eps,\vth),
\end{equation}
where we used that $\widehat{\boldsymbol{m}}_h$ takes values in $\S^2$ and that $||\nabla \widehat{\boldsymbol{m}}_h||_{L^2(\Omega^\eps_\vth;\rtt)}\leq C$. Define $\widetilde{\boldsymbol{M}}_h \coloneqq \widehat{\boldsymbol{M}}_h \circ \boldsymbol{z}_h^{-1}$. By construction, $\widetilde{\boldsymbol{M}}_h \restr{\Omega^\eps_{\vth h}}=\widetilde{\boldsymbol{m}}_h \restr{\Omega^\eps_{\vth h}}$. Moreover, from  \eqref{eqn:grad-M-hat},  using the change-of-variables formula we obtain
\begin{equation}
    \label{eqn:grad-M-tilde}
    \begin{split}
        \int_{\R^3}|\nabla \widetilde{\boldsymbol{M}}_h|^2\,\d\boldsymbol{\xi}&=\int_{\R^3}|\nabla_h \widehat{\boldsymbol{M}}_h \circ \boldsymbol{z}_h^{-1}|^2\,\d\boldsymbol{\xi}     \leq \frac{1}{h^2}\int_{\R^3}|\nabla \widehat{\boldsymbol{M}}_h \circ \boldsymbol{z}_h^{-1}|^2\,\d\boldsymbol{\xi}\\
        &=\frac{1}{h}\int_{\R^3}|\nabla \widehat{\boldsymbol{M}}_h|^2\,\d\boldsymbol{x}\leq \frac{C(\eps,\vth)}{h}.
    \end{split}
\end{equation}
We now apply Proposition \ref{prop:lipschitz-truncation} to the map $\widetilde{\boldsymbol{M}}_h$. For every $\lambda>0$, we find a set $F_{\lambda,h} \subset \R^3$ and a constant $C>0$ independent on $h$, $\eps$ and $\vth$, such that for every $\boldsymbol{\xi},\widetilde{\boldsymbol{\xi}} \in F_{\lambda,h}$ we have
\begin{equation}
    \label{eqn:lip-prop}
    |\widetilde{\boldsymbol{M}}_h(\boldsymbol{\xi})-\widetilde{\boldsymbol{M}}_h(\widetilde{\boldsymbol{\xi}})|\leq C \, \lambda\, |\boldsymbol{\xi}-\widetilde{\boldsymbol{\xi}}|.
\end{equation}
Using \eqref{eqn:grad-M-tilde}, we estimate the measure of the complement of $F_{\lambda,h}$ as follows
\begin{equation}
    \label{eqn:complement-F_lambda}
    \leb(\R^3\setminus F_{\lambda,h})\leq \frac{C}{\lambda^2}\int_{\{|\nabla \widetilde{\boldsymbol{M}}_h|>\lambda /2\}} |\nabla \widetilde{\boldsymbol{M}}_h|^2 \,\d\boldsymbol{\xi}\leq \frac{C(\eps,\vth)}{\lambda^2h}.
\end{equation}
Set $U_{\lambda,h} \coloneqq \boldsymbol{z}_h^{-1}(F_{\lambda,h}\cap \Omega^\eps_{\vth h})$ and $V_{\lambda,h} \coloneqq \widetilde{\boldsymbol{y}}_h^{-1}(F_{\lambda,h}\cap \Omega^\eps_{\vth h})$, so that $U_{\lambda,h}\subset \Omega^\eps_\vth$ and $V_{\lambda,h}\subset A^\eps_{\vth,h}$. We split the first integral on the right-hand side of \eqref{eqn:comp-m-comp-y-2} as
\begin{equation}
    \label{eqn:lip-truncation-1}
    \begin{split}
    \int_{\Omega^\eps_\vth \cap A^\eps_{\vth, h}} (\widetilde{\boldsymbol{m}}_h \circ \widetilde{\boldsymbol{y}}_h - \widehat{\boldsymbol{m}}_h)\,\overline{\boldsymbol{\varphi}}\,\d \boldsymbol{x} &= \int_{(\Omega^\eps_\vth \cap A^\eps_{\vth, h})\cap (U_{\lambda,h}\cap V_{\lambda,h})}  (\widetilde{\boldsymbol{m}}_h \circ \widetilde{\boldsymbol{y}}_h - \widehat{\boldsymbol{m}}_h)\,\overline{\boldsymbol{\varphi}}\,\d \boldsymbol{x}\\
    & + \int_{(\Omega^\eps_\vth \cap A^\eps_{\vth, h})\setminus (U_{\lambda,h}\cap V_{\lambda,h})}  (\widetilde{\boldsymbol{m}}_h \circ \widetilde{\boldsymbol{y}}_h - \widehat{\boldsymbol{m}}_h)\,\overline{\boldsymbol{\varphi}}\,\d \boldsymbol{x}.
    \end{split}
\end{equation}

The first integral on the right-hand side of \eqref{eqn:lip-truncation-1} is estimated by  exploiting the Lipschitz  continuity \eqref{eqn:lip-prop}. Indeed, if $\boldsymbol{x}\in (\Omega^\eps_\vth\cap A^\eps_{\vth,h}) \cap (U_{\lambda,h}\cap V_{\lambda,h})$, then $\widetilde{\boldsymbol{y}}_h(\boldsymbol{x})\in F_{\lambda,h}$ and $\boldsymbol{z}_h(\boldsymbol{x})\in F_{\lambda,h}$. Recalling \eqref{eqn:key-estimate}, we obtain
\begin{equation}
    \label{eqn:lip-truncation-3}
    \begin{split}
    \int_{(\Omega^\eps_\vth \cap A^\eps_{\vth, h})\cap (U_{\lambda,h}\cap V_{\lambda,h})}  &(\widetilde{\boldsymbol{m}}_h \circ \widetilde{\boldsymbol{y}}_h-\widehat{\boldsymbol{m}}_h)\,\overline{\boldsymbol{\varphi}}\,\d\boldsymbol{x}\\
    &=\int_{(\Omega^\eps_\vth \cap A^\eps_{\vth, h})\cap (U_{\lambda,h}\cap V_{\lambda,h})}  (\widetilde{\boldsymbol{M}}_h \circ \widetilde{\boldsymbol{y}}_h-\widetilde{\boldsymbol{M}}_h\circ \boldsymbol{z}_h)\,\overline{\boldsymbol{\varphi}}\,\d\boldsymbol{x}\\
    &\leq 
    \int_{(\Omega^\eps_\vth \cap A^\eps_{\vth, h})\cap (U_{\lambda,h}\cap V_{\lambda,h})} C \lambda\, |\widetilde{\boldsymbol{y}}_h-\boldsymbol{z}_h|\,|\overline{\boldsymbol{\varphi}}|\,\d\boldsymbol{x}\\
    & \leq C \lambda ||\widetilde{\boldsymbol{y}}_h-\boldsymbol{z}_h||_{L^\infty(\Omega';\R^3)} \,  ||\boldsymbol{\varphi}||_{L^2(\Omega';\R^3)}\\
    & \leq C \lambda  \,  ||\boldsymbol{\varphi}||_{L^2(\Omega';\R^3)}\, h^{\beta/p-1}.
    \end{split}
\end{equation}

For the second integral on the right-hand side of \eqref{eqn:lip-truncation-1}, first note that 
\begin{equation}
\label{eqn:rem-set}
    (\Omega^\eps_\vth \cap A^\eps_{\vth, h})\setminus (U_{\lambda,h}\cap V_{\lambda,h}) \subset (\Omega^\eps_\vth \setminus U_{\lambda,h}) \cup (A^\eps_{\vth, h} \setminus V_{\lambda,h}).
\end{equation} 
Using the area formula and \eqref{eqn:complement-F_lambda}, we have
\begin{equation}
    \label{eqn:diff-meas-1}
    \begin{split}
    \leb(\Omega^\eps_\vth \setminus U_{\lambda,h})&=\leb(\boldsymbol{z}_h^{-1}(\Omega^\eps_{\vth h}\setminus F_{\lambda,h}))=\int_{\Omega^\eps_{\vth h}\setminus F_{\lambda,h}} \det \nabla \boldsymbol{z}_h^{-1}\,\d\boldsymbol{\xi}\\
    & =h^{-1}\leb(\Omega^\eps_{\vth h}\setminus F_{\lambda,h}) \leq h^{-1}\leb(\R^3 \setminus F_{\lambda,h})\leq \frac{C(\eps,\vth)}{\lambda^2 h^2}.
    \end{split}
\end{equation}
Similarly, by the area formula 
\begin{equation*}
    \leb(\Omega^\eps_{\vth h} \setminus F_{\lambda,h})=\leb(\widetilde{\boldsymbol{y}}_h(A^\eps_{\vth,h}\setminus V_{\lambda,h}))=\int_{A^\eps_{\vth,h}\setminus V_{\lambda,h}} \det \nabla \widetilde{\boldsymbol{y}}_h\,\d\boldsymbol{x}=h \, \leb(A^\eps_{\vth,h}\setminus V_{\lambda,h}), 
\end{equation*}
so that
\begin{equation}
\label{eqn:diff-meas-2}
\leb(A^\eps_{\vth,h}\setminus V_{\lambda,h})=h^{-1} \leb(\Omega^\eps_{\vth h} \setminus F_{\lambda,h}) \leq h^{-1} \leb(\R^3 \setminus F_{\lambda,h}) \leq \frac{C(\eps,\vth)}{\lambda^2h^2}.  
\end{equation}
Using \eqref{eqn:rem-set}-\eqref{eqn:diff-meas-2}, since both $\widetilde{\boldsymbol{m}}_h$ and $\widehat{\boldsymbol{m}}_h$ take values in $\S^2$, we  compute
\begin{equation}
\label{eqn:lip-truncation-2}
\begin{split}
\int_{(\Omega^\eps_\vth \cap A^\eps_{\vth, h})\setminus (U_{\lambda,h}\cap V_{\lambda,h})}  &(\widetilde{\boldsymbol{m}}_h \circ \widetilde{\boldsymbol{y}}_h-\widehat{\boldsymbol{m}}_h)\,\overline{\boldsymbol{\varphi}}\,\d\boldsymbol{x}\\
&\leq 2 ||\boldsymbol{\varphi}||_{L^2(\Omega';\R^3)} \left ( \leb(\Omega^\eps_\vth \setminus U_{\lambda,h}) + \leb(A^\eps_{\vth,h}\setminus V_{\lambda,h})\right)^{1/2}\\
& \leq  ||\boldsymbol{\varphi}||_{L^2(\Omega';\R^3)} \frac{C(\eps,\vth)}{\lambda h}.
\end{split}
\end{equation}
Choosing $\lambda=h^{-s}$ for some $s>0$, and combining \eqref{eqn:lip-truncation-3} with \eqref{eqn:lip-truncation-2}, we obtain
\begin{equation}
    \label{eqn:lip-truncation-4}
     \int_{\Omega^\eps_\vth \cap A^\eps_{\vth, h}} (\widetilde{\boldsymbol{m}}_h \circ \widetilde{\boldsymbol{y}}_h-\widehat{\boldsymbol{m}}_h)\,\overline{\boldsymbol{\varphi}}\,\d\boldsymbol{x} \leq C(\eps,\vth) \, ||\boldsymbol{\varphi}||_{L^2(\Omega';\R^3)}\, \left ( h^{s-1}+h^{\beta/p-s-1} \right)
\end{equation}
so that, for $1<s<\beta/p-1$, the right-hand side goes to zero, as $h \to 0^+$. 

Therefore, we conclude the argument as follows. For every $\rho>0$, we choose $\eps \ll 1$  and $\vth$ sufficiently close to one, according to $\Omega'$, $\boldsymbol{\varphi}$ and $\rho$, so that $\Omega' \subset \Omega^\eps_\vth$ and the right-hand side of \eqref{eqn:comp-m-comp-y-3} is smaller that $\rho/3$. Then, for every $h\ll 1$, according to $\eps$ and $\vth$ chosen before, the first integrals  on  the right-hand side of \eqref{eqn:comp-m-comp-y-1} and  \eqref{eqn:lip-truncation-4} are both smaller than $\rho/3$, so that
\begin{equation*}
   \limsup_{h\to 0^+} \left |\int_{\Omega'} (\widetilde{\boldsymbol{m}}_h \circ \widetilde{\boldsymbol{y}}_h-\boldsymbol{\lambda})\,\boldsymbol{\varphi}\,\d\boldsymbol{x} \right | < \rho.
\end{equation*}
From the arbitrariness of $\rho$, $\Omega'$ and $\boldsymbol{\varphi}$, we conclude that $\widetilde{\boldsymbol{m}}_h \circ \widetilde{\boldsymbol{y}}_h \wk \boldsymbol{\lambda}$ in $L^2_\loc(\Omega;\R^3)$. 

Now, for every measurable set  $\Omega' \subset \subset \Omega$, there holds $\widetilde{\boldsymbol{m}}_h \circ \widetilde{\boldsymbol{y}}_h \wk \boldsymbol{\lambda}$ in $L^2(\Omega';\R^3)$ and, since $|\widetilde{\boldsymbol{m}}_h \circ \widetilde{\boldsymbol{y}}_h|=|\boldsymbol{\lambda}|=1$ almost everywhere in $\Omega'$, we  have $||\widetilde{\boldsymbol{m}}_h \circ \widetilde{\boldsymbol{y}}_h||_{L^2(\Omega';\R^3)}=||\boldsymbol{\lambda}||_{L^2(\Omega';\R^3)}$. As a result, $\widetilde{\boldsymbol{m}}_h \circ \widetilde{\boldsymbol{y}}_h \to \boldsymbol{\lambda}$ in $L^2(\Omega';\R^3)$ so that, up to subsequences, $\widetilde{\boldsymbol{m}}_h \circ \widetilde{\boldsymbol{y}}_h \to \boldsymbol{\lambda}$ almost everywhere in $\Omega'$. Considering a sequence of compact subsets invading $\Omega$, we find a  subsequence such that $\widetilde{\boldsymbol{m}}_h \circ \widetilde{\boldsymbol{y}}_h \to \boldsymbol{\lambda}$ almost everywhere in $\Omega$, and finally, by the Dominated Convergence Theorem, we  obtain \eqref{eqn:conv-m-comp-y}. 

{\MMM
\textbf{Step 4.} We are left to prove \eqref{eqn:conv-mu}.
Note that, by \eqref{eqn:key-estimate}, we  have
\begin{equation*}
	\Omega^{\widetilde{\boldsymbol{y}}_h}\subset \widetilde{\boldsymbol{y}}_h(\Omega) \subset \boldsymbol{z}_h(\Omega) + \closure{B}(\boldsymbol{0},\delta_h) = \Omega_h + \closure{B}(\boldsymbol{0},\delta_h),
\end{equation*}
where $\delta_h=C h^{\beta/p-1}$, from which we  deduce the following (see Figure \ref{fig:supcylinder}):
\begin{equation}
	\label{eqn:supercylinder}
	\forall \eps>0,\:\:\forall \ell>1,\:\:\exists\, \underline{h}(\eps,\ell)>0:\:\:\forall\,0<h\leq \underline{h}(\eps,\ell),\:\:\Omega^{\widetilde{\boldsymbol{y}}_h} \subset \Omega^{-\eps}_{\ell h}.
\end{equation}

\begin{figure}
	\centering
	\begin{tikzpicture}
		\node[anchor=south west,inner sep=0] at (0,0) {\includegraphics[width=\textwidth]{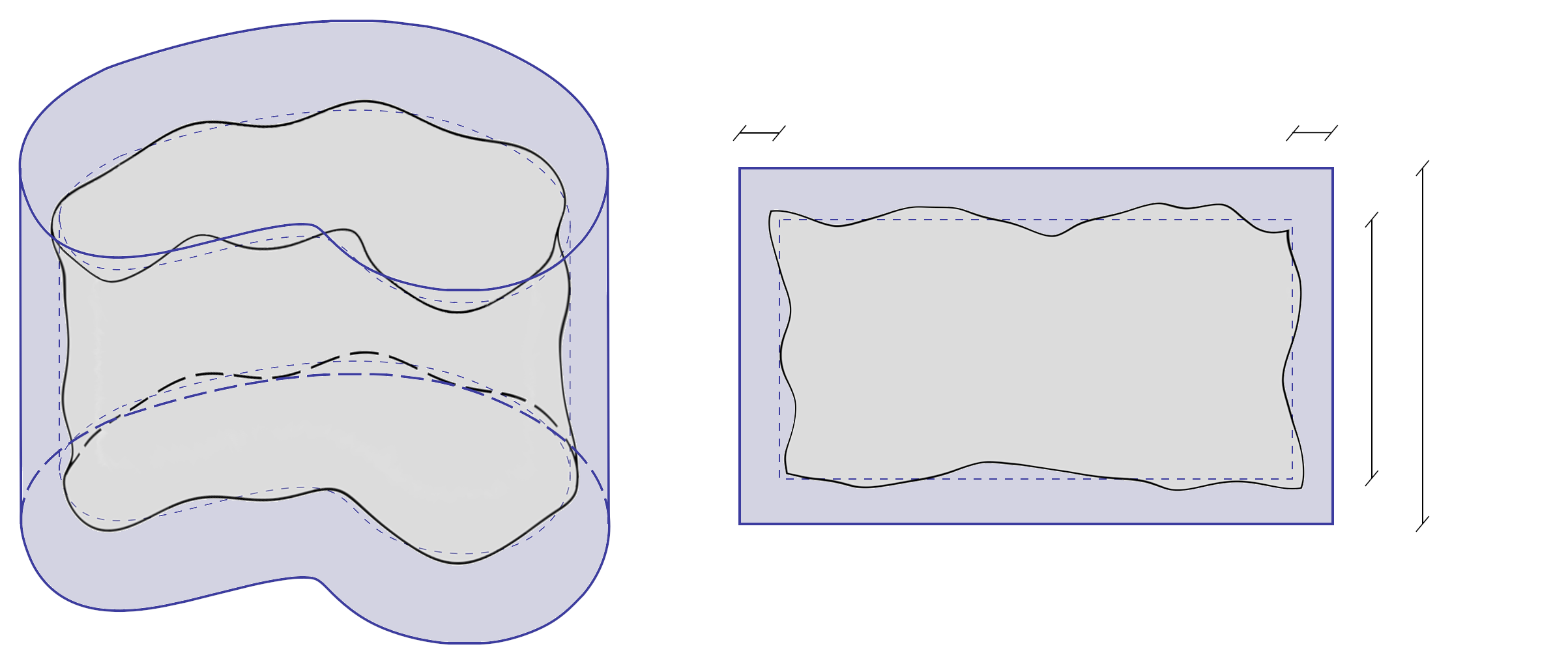}};
		\node at (6.25,6) {\large \color{bblue}  ${\Omega^{-\varepsilon}_{\ell h}}$};
		\node at (5.3,4.8) {\large $\Omega^{\widetilde{\boldsymbol{y}}_h}$};
		\node at (4.65,5.65) {\large \color{bblue} $\Omega_h$};
		\node at (7.8,5.5) {$\varepsilon$};
		\node at (13.45,5.5) {$\varepsilon$};
		\node at (14.2,3.15) {$h$};
		\node at (15.2,3.15) {$\ell h$};
	\end{tikzpicture}
	\caption{On the left, the cylinder $\Omega^{-\varepsilon}_{\ell h}$ contains the deformed configuration $\Omega^{\widetilde{\boldsymbol{y}}_h}$. On the right, the corresponding section.}
	\label{fig:supcylinder}
\end{figure}

Using \eqref{eqn:subcylinder} and \eqref{eqn:supercylinder}, we  infer that $\chi_{\Omega^{\widetilde{\boldsymbol{y}}_h}} \circ \boldsymbol{z}_h=\chi_{\boldsymbol{z}_h^{-1}(\Omega^{\widetilde{\boldsymbol{y}}_h})} \to \chi_\Omega$ almost everywhere in $\R^3$. Indeed,  any point in $\Omega$ is contained in $\Omega^\eps_{\vth}$ for some $\eps>0$ and $0<\vth<1$ and, by \eqref{eqn:subcylinder}, for $h \ll 1$ we have $\Omega^\eps_\vth \subset  \boldsymbol{z}_h^{-1}(\Omega^{\widetilde{\boldsymbol{y}}_h})$.  Similarly, any point $\R^3 \setminus \closure{\Omega}$ is contained in  $\R^3 \setminus \Omega^{-\eps}_{\ell}$ for some $\eps>0$ and $\ell>1$, and, by \eqref{eqn:supercylinder}, for $h \ll 1$ we have $\boldsymbol{z}_h^{-1}(\Omega^{\widetilde{\boldsymbol{y}}_h}) \subset \Omega^{-\eps}_{\ell}$, so that $\R^3 \setminus \Omega^{-\eps}_{\ell} \subset \R^3 \setminus \boldsymbol{z}_h^{-1}(\Omega^{\widetilde{\boldsymbol{y}}_h})$. By Step 2, we can assume that, for almost every $\boldsymbol{x}\in \Omega$, we have $\widetilde{\boldsymbol{m}}_h (\boldsymbol{z}_h(\boldsymbol{x})) \to \boldsymbol{\lambda}(\boldsymbol{x}')$, where the composition $\widetilde{\boldsymbol{m}}_h \circ \boldsymbol{z}_h$ is defined only for $h \ll 1$ depending on $\boldsymbol{x}\in \Omega$. Thus $\widetilde{\boldsymbol{\mu}}_h=\chi_{\boldsymbol{z}_h^{-1}(\Omega^{\widetilde{\boldsymbol{y}}_h})}(\widetilde{\boldsymbol{m}}_h \circ \boldsymbol{z}_h) \to \chi_\Omega \boldsymbol{\lambda}$ almost everywhere in $\R^3$. Since these two maps are bounded and supported in a compact set by \eqref{eqn:supercylinder}, applying the Dominated Convergence Theorem we obtain \eqref{eqn:conv-mu}.
}
\end{proof}

\subsection{Lower bound}
\label{subsec:lb}

In this subsection, we show that the energy functional $E$ defined in \eqref{eqn:energy_E} provides a lower bound for the asymptotic behavior of the energies $E_h$ defined in  \eqref{eqn:energy_Eh}. In the proof, we will analyze the three terms of the energy separately. 

We begin by showing that the lower bound for the exchange energies follows immediately from the argument used to prove the compactness of magnetizations. Recall the definitions of $E_h$ and $E$ in \eqref{eqn:energy_Eh} and \eqref{eqn:energy_E}, respectively.

\begin{proposition}[Lower bound for the exchange energy]
\label{prop:liminf-exc}
Let $((\boldsymbol{y}_h,\boldsymbol{m}_h))_h \subset \mathcal{Q}$ be such that $E^{\text{\em el}}_h(\boldsymbol{y}_h,\boldsymbol{m}_h)+E^{\text{\em exc}}_h(\boldsymbol{y}_h,\boldsymbol{m}_h)\leq C$ for every $h>0$. Given the map $\boldsymbol{\lambda} \in W^{1,2}(S;\S^2)$ identified in Proposition \ref{prop:comp-mag}, we have, up to subsequences, as $h \to 0^+$:
\begin{equation}
    \label{eqn:lower-bound-exc}
    E^\mathrm{exc}(\boldsymbol{\lambda})\leq \liminf_{h \to 0^+} E^\mathrm{exc}_h({\boldsymbol{y}}_h, {\boldsymbol{m}}_h).
\end{equation}
\end{proposition}
\begin{proof}
We follow the notation of Propositions \ref{prop:comp-def} and \ref{prop:comp-mag}.  Recall that $\widetilde{\boldsymbol{y}}_h=\boldsymbol{T}_h \circ \boldsymbol{y}_h$ and $\widetilde{\boldsymbol{m}}_h=\boldsymbol{Q}_h^\top \boldsymbol{m}_h \circ \boldsymbol{T}_h^{-1}$. Here, $\boldsymbol{T}_h$ is the rigid motion given by $\boldsymbol{T}_h(\boldsymbol{\xi})=\boldsymbol{Q}_h^\top \boldsymbol{\xi}-\boldsymbol{c}_h$ for every $\boldsymbol{\xi}\in \R^3$, where $\boldsymbol{Q}_h \in SO(3)$ and $\boldsymbol{c}_h \in \R^3$. By \eqref{eqn:exc-invariance}, $E^\text{exc}_h(\boldsymbol{y}_h,\boldsymbol{m}_h)=E^\text{exc}_h(\widetilde{\boldsymbol{y}}_h,\widetilde{\boldsymbol{m}}_h)$. 
We  argue as in \eqref{eqn:comp-m-comp-y-0} and \eqref{eqn:comp-m-comp-y-1}. More explicitly, we fix $\eps>0$ and $0<\vth<1$, so that, as in Step 2 of the proof of Proposition \ref{prop:comp-mag}, we have $\widehat{\boldsymbol{m}}_h \wk \boldsymbol{\lambda}$ in $W^{1,2}(\Omega^\eps_\vth;\R^3)$ for some $\boldsymbol{\lambda}\in W^{1,2}(S;\S^2)$ and $\partial_3 \widehat{\boldsymbol{m}}_h/h \wk {\MMM \boldsymbol{k}}$ in $L^2(\Omega^\eps_\vth;\R^3)$ for some ${\MMM\boldsymbol{k}}\in L^2_\loc(\Omega;\R^3)$. Then, taking into account \eqref{eqn:subcylinder}, by lower semicontinuity, we have
\begin{equation*}
    \begin{split}
     \liminf_{h \to 0^+} E_h^\mathrm{exc}(\widetilde{\boldsymbol{y}}_h,\widetilde{\boldsymbol{m}}_h) & \geq  \liminf_{h \to 0^+} \frac{\MMM \alpha}{h} \int_{\Omega^\eps_{\vth h}} |\nabla \widetilde{\boldsymbol{m}}_h |^2 \,\d \boldsymbol{\xi} \geq \liminf_{h \to 0^+} {\MMM \alpha} \int_{\Omega^\eps_\vth} |\nabla_h \widehat{\boldsymbol{m}}|^2\,\d\boldsymbol{x}\\
     & \geq {\MMM \alpha} \int_{\Omega^\eps_\vth} |\nabla'\boldsymbol{\lambda}|^2\,\d\boldsymbol{x} + {\MMM \alpha} \int_{\Omega^\eps_\vth} |{\MMM\boldsymbol{k}}|^2\,\d\boldsymbol{x} \geq {\MMM \alpha} \int_{\Omega^\eps_\vth} |\nabla'\boldsymbol{\lambda}|^2\,\d\boldsymbol{x} = {\MMM \alpha}\,\vth \int_{S^\eps} |\nabla'\boldsymbol{\lambda}|^2\,\d\boldsymbol{x}'.
    \end{split}
\end{equation*}
Thus, letting $\eps \to 0^+$ and $\vth \to 1^-$, we obtain \eqref{eqn:lower-bound-exc}.
\end{proof}

The proof of the lower bound for the elastic energy requires more work. In order to deal with the incompressibility constraint, we make use of a technique developed in \cite{conti.dolzmann}  (see also \cite{li.chermisi}). The idea consists in approximating the incompressible energy from below by adding to the compressible energy a penalization term that forces the determinant to be close to one. Namely, for every  $k \in \N$, we define an approximated energy density as
\begin{equation*}
    W^k(\boldsymbol{F},\boldsymbol{\nu}) \coloneqq W(\boldsymbol{F},\boldsymbol{\nu})+\frac{k}{2}(\det \boldsymbol{F}-1)^2
\end{equation*}
for every $\boldsymbol{F}\in \rtt$ and $\boldsymbol{\nu}\in \S^2$. Note that $W^\inc \geq W^k \geq W$. Moreover, $W^k$ satisfies frame-indifference and normalization properties analogous to \eqref{eqn:frame_indifference} and \eqref{eqn:normalization}. Therefore, as in \eqref{eqn:taylor-expansion}, we have the second-order Taylor expansion 
\begin{equation}
    \label{eqn:taylor-expansion-k}
    W^k(\boldsymbol{I}+\boldsymbol{G},\boldsymbol{\nu})=\frac{1}{2}Q_3^k(\boldsymbol{G},\boldsymbol{\nu})+\omega^k(\boldsymbol{G},\boldsymbol{\nu})
\end{equation}
for every $\boldsymbol{G}\in \rtt$ with $|\boldsymbol{G}|<\delta$, where $\delta>0$ was introduced in \eqref{eqn:local-smooth}, and for every $\boldsymbol{\nu}\in \S^2$. In particular
\begin{equation*}
    Q^k_3(\boldsymbol{G},\boldsymbol{\nu})=Q_3(\boldsymbol{G},\boldsymbol{\nu})+k (\tr \boldsymbol{G})^2
\end{equation*}
and $\omega^k(\boldsymbol{G},\boldsymbol{\nu})=\omega(\boldsymbol{G},\boldsymbol{\nu})+k \, \gamma(|\boldsymbol{G}|^2)$ with $\gamma(t)=o(t^2)$ as $t\to 0^+$, where $Q_3$ and $\omega$ were introduced in \eqref{eqn:taylor-expansion} and \eqref{eqn:definition-Q3}, respectively.
We recall \eqref{eqn:Q2-inc} and, for every $\boldsymbol{H} \in \rtwtw$ and $\boldsymbol{\nu} \in \S^2$, we define
\begin{equation*}
    Q^k_2(\boldsymbol{G},\boldsymbol{\nu}) \coloneqq \min \left \{Q_3^k \left (\left( \renewcommand\arraystretch{1.2}\begin{array}{@{}c|c@{}}   \boldsymbol{H}  & \boldsymbol{0}'\\ \hline  (\boldsymbol{0}')^\top & 0 \end{array} \right)+\boldsymbol{c}\otimes \boldsymbol{e}_3+\boldsymbol{e}_3 \otimes \boldsymbol{c},\boldsymbol{\nu} \right ) :\,\boldsymbol{c}\in \R^3 \right \}.
\end{equation*}
For every $k \in \N$, the following comparison estimate holds
\begin{equation}
    \label{eqn:comparison-estimate}
    Q^k_2(\boldsymbol{H},\boldsymbol{\nu}) \geq Q^\inc_2(\boldsymbol{H},\boldsymbol{\nu})-\frac{C}{\sqrt{k}} |\boldsymbol{H}|^2
\end{equation}
for every $\boldsymbol{H}\in \rtwtw$ and $\boldsymbol{\nu} \in \S^2$. This is proved analogously to \cite[Lemma 2.1]{conti.dolzmann}.

We are now in a position to establish a lower bound for the elastic energy. Recall the definitions of $E_h$ and $E$ in \eqref{eqn:energy_Eh} and \eqref{eqn:energy_E}, respectively.

\begin{proposition}[Lower bound for the elastic energy]
\label{prop:liminf-el}
Let $((\boldsymbol{y}_h,\boldsymbol{m}_h))_h \subset \mathcal{Q}$ be such that $E^{\text{\em el}}_h(\boldsymbol{y}_h,\boldsymbol{m}_h)+E^{\text{\rm {exc}}}_h(\boldsymbol{y}_h,\boldsymbol{m}_h)\leq C$ for every $h>0$. Then, given the maps $\boldsymbol{u}\in W^{1,2}(S;\R^2)$, $v \in W^{2,2}(S)$ and $\boldsymbol{\lambda}\in W^{1,2}(S;\S^2)$ identified in Propositions \ref{prop:comp-def} and \ref{prop:comp-mag}, we have, up to subsequences, as $h \to 0^+$:
\begin{equation}
    E^\mathrm{el}(\boldsymbol{u},v,\boldsymbol{\lambda}) \leq \liminf_{h \to 0^+}  E^\mathrm{el}_h({\boldsymbol{y}}_h, {\boldsymbol{m}}_h).
\end{equation}
\end{proposition}
\begin{proof}
We follow the notation of Propositions \ref{prop:comp-def} and \ref{prop:comp-mag}. Recall that $\widetilde{\boldsymbol{y}}_h=\boldsymbol{T}_h \circ \boldsymbol{y}_h$ and $\widetilde{\boldsymbol{m}}_h=\boldsymbol{Q}_h^\top \boldsymbol{m}_h \circ \boldsymbol{T}_h^{-1}$, where $\boldsymbol{T}_h$ is the rigid motion given by $\boldsymbol{T}_h(\boldsymbol{\xi})=\boldsymbol{Q}_h^\top \boldsymbol{\xi}-\boldsymbol{c}_h$ for every $\boldsymbol{\xi}\in \R^3$, with $\boldsymbol{Q}_h \in SO(3)$ and $\boldsymbol{c}_h \in \R^3$. Set $\boldsymbol{F}_h \coloneqq \nabla_h \boldsymbol{y}_h$ and $\widetilde{\boldsymbol{F}}_h\coloneqq\nabla_h \widetilde{\boldsymbol{y}}_h$, so that $\widetilde{\boldsymbol{F}}_h=\boldsymbol{Q}_h^\top \boldsymbol{F}_h$, and set $\widetilde{\boldsymbol{R}}_h\coloneqq \boldsymbol{Q}_h^\top \boldsymbol{R}_h$, where $\boldsymbol{R}_h$ is the map given by Lemma \ref{lem:approx-rot}. Define
\begin{equation*}
\boldsymbol{G}_h \coloneqq \frac{1}{h^{\beta/2}}(\widetilde{\boldsymbol{R}}_h^\top \widetilde{\boldsymbol{F}}_h-\boldsymbol{I}).    
\end{equation*}
By the first estimate in \eqref{eqn:approx-rot-tilda1} with $q=2$, we have $||\boldsymbol{G}_h||_{L^2(\Omega;\rtt)}\leq C$. Hence, up to subsequences,  $\boldsymbol{G}_h \wk \boldsymbol{G}$ in $L^2(\Omega;\rtt)$ for some $\boldsymbol{G}\in L^2(\Omega;\rtt)$. Proceeding as in \cite[Lemma 2]{friesecke.james.mueller2}, we prove that  
\begin{equation}
    \label{eqn:G-affine-in-x3}
\boldsymbol{G}''(\boldsymbol{x})=\boldsymbol{K}(\boldsymbol{x}')+\boldsymbol{L}(\boldsymbol{x}')x_3
\end{equation}
for almost every $\boldsymbol{x} \in \Omega$, where we set $\boldsymbol{K} \coloneqq\sym \nabla'\boldsymbol{u} \in \rtwtw$ and $\boldsymbol{L} \coloneqq - (\nabla')^2 v \in \rtwtw$. {\MMM Here, $\boldsymbol{u}\in W^{1,2}(S;\R^2)$ and $v\in W^{2,2}(S)$ are the limiting averaged displacements identified in \eqref{eqn:comp-u} and \eqref{eqn:comp-v}, respectively.} Define $X_h\coloneqq \{|\boldsymbol{G}_h|\leq h^{-\beta/4}\}$ and set $\chi_h \coloneqq \chi_{X_h}$. By the first estimate in \eqref{eqn:approx-rot-tilda1} with $q=2$, we have $h^{\beta/4}\boldsymbol{G}_h\to \boldsymbol{O}$ in $L^2(\Omega;\rtt)$ so that, up to subsequences, $h^{\beta/4}\boldsymbol{G}_h\to \boldsymbol{O}$ almost everywhere in $\Omega$. Thus, $\chi_h \to 1$ almost everywhere in $\Omega$ and, by the Dominated Convergence Theorem, $\chi_h \to 1$ in $L^1(\Omega)$. From this we deduce that $\chi_h \boldsymbol{G}_h \wk \boldsymbol{G}$ in $L^2(\Omega;\rtt)$.

For simplicity, set $\boldsymbol{\nu}_h\coloneqq \boldsymbol{m}_h \circ \boldsymbol{y}_h$ and $\widetilde{\boldsymbol{\nu}}_h \coloneqq \widetilde{\boldsymbol{m}}_h \circ \widetilde{\boldsymbol{y}}_h$, so that $\widetilde{\boldsymbol{\nu}}_h=\boldsymbol{Q}_h^\top \boldsymbol{\nu}_h$. Note that, since $E^{\text{el}}_h({\boldsymbol{y}}_h,{\boldsymbol{m}}_h)\leq C$, we have $\det \boldsymbol{F}_h=1$ and, in turn, $W^\inc({\boldsymbol{F}}_h,\boldsymbol{\nu}_h)=W({\boldsymbol{F}}_h,\boldsymbol{\nu}_h)=W^k({\boldsymbol{F}}_h,\boldsymbol{\nu}_h)$ for every $k \in \N$. By frame indifference and by \eqref{eqn:taylor-expansion-k}, we obtain
\begin{equation}
    \begin{split}
   \chi_h W^\inc({\boldsymbol{F}}_h,\boldsymbol{\nu}_h)&=\chi_h W^k({\boldsymbol{F}}_h,\boldsymbol{\nu}_h )=\chi_h W^k(\boldsymbol{Q}_h^\top{\boldsymbol{F}}_h,\boldsymbol{Q}_h^\top \boldsymbol{\nu}_h )=\chi_h W^k(\widetilde{\boldsymbol{F}}_h,\widetilde{\boldsymbol{\nu}}_h)\\
   &=\chi_h W^k(\widetilde{\boldsymbol{R}}_h^\top \widetilde{\boldsymbol{F}}_h, \widetilde{\boldsymbol{R}}_h^\top \widetilde{\boldsymbol{\nu}}_h)=\chi_h W^k(\boldsymbol{I}+h^{\beta/2} \boldsymbol{G}_h,\widetilde{\boldsymbol{R}}_h^\top \widetilde{\boldsymbol{\nu}}_h)\\
    &=\frac{1}{2}\chi_h Q^k_3(h^{\beta/2}\boldsymbol{G}_h,\widetilde{\boldsymbol{R}}_h^\top \widetilde{\boldsymbol{\nu}}_h)+\chi_h\, \omega^k(h^{\beta/2}\boldsymbol{G}_h,\widetilde{\boldsymbol{R}}_h^\top \widetilde{\boldsymbol{\nu}}_h)\\
    &=\frac{h^\beta}{2} Q^k_3(\chi_h \boldsymbol{G}_h,\widetilde{\boldsymbol{R}}_h^\top \widetilde{\boldsymbol{\nu}}_h)+ \omega^k(h^{\beta/2}\chi_h \boldsymbol{G}_h,\widetilde{\boldsymbol{R}}_h^\top \widetilde{\boldsymbol{\nu}}_h),
    \end{split}
\end{equation}
where we used that ${\MMM h^{\beta/2}}|\boldsymbol{G}_h|<\delta$ on $X_h$ for $h \ll 1$ depending only on $\delta$.
We compute
\begin{equation}
    \label{eqn:lower-bound-el-1}
    \begin{split}
    E_h^\mathrm{el}({\boldsymbol{y}}_h,{\boldsymbol{m}}_h)&=\frac{1}{h^\beta} \int_\Omega  W^k(\boldsymbol{F}_h,\boldsymbol{\nu}_h) \,\d \boldsymbol{x}    \geq \frac{1}{h^\beta} \int_\Omega  \chi_h \,W^k(\boldsymbol{F}_h,\boldsymbol{\nu}_h) \,\d \boldsymbol{x}\\
    & = \frac{1}{2}\int_\Omega Q^k_3(\chi_h \boldsymbol{G}_h,\widetilde{\boldsymbol{R}}_h^\top \widetilde{\boldsymbol{\nu}}_h)\,\d\boldsymbol{x} + \frac{1}{h^\beta} \int_\Omega \omega^k(h^{\beta/2}\,\chi_h \boldsymbol{G}_h, \widetilde{\boldsymbol{R}}_h^\top \widetilde{\boldsymbol{\nu}}_h)\,\d \boldsymbol{x}.
    \end{split}
\end{equation}
We focus on the first integral on the right-hand side of \eqref{eqn:lower-bound-el-1}. From \eqref{eqn:approx-rot-tilda1} and \eqref{eqn:approx-rot-tilda2}, for $q=p$, using the Morrey embedding, we deduce that $\widetilde{\boldsymbol{R}}_h \to \boldsymbol{I}$ uniformly in $\Omega$. By \eqref{eqn:conv-m-comp-y}, upon extracting a further subsequence, $\widetilde{\boldsymbol{\nu}}_h \to \boldsymbol{\lambda}$ almost everywhere in $\Omega$ and hence, by \eqref{eqn:coupling_C}, we obtain that $\mathbb{C}^{\widetilde{\boldsymbol{R}}_h^\top \widetilde{\boldsymbol{\nu}}_h} \to \mathbb{C}^{\boldsymbol{\lambda}}$ almost everywhere in $\Omega$. Fix $\rho>0$. By the Egorov Theorem, there exists a measurable set $K_\rho \subset \Omega$ with $\leb(K_\rho)<\rho$ such that $\mathbb{C}^{\widetilde{\boldsymbol{R}}_h^\top \widetilde{\boldsymbol{\nu}}_h} \to \mathbb{C}^{\boldsymbol{\lambda}}$ uniformly in $\Omega \setminus K_\rho$.  We write
\begin{equation*}
    \begin{split}
    \int_\Omega Q^k_3(\chi_h \boldsymbol{G}_h,\widetilde{\boldsymbol{R}}_h^\top \widetilde{\boldsymbol{\nu}}_h)\,\d\boldsymbol{x} &\geq \int_{\Omega \setminus K_\rho} Q^k_3(\chi_h \boldsymbol{G}_h,\widetilde{\boldsymbol{R}}_h^\top\widetilde{\boldsymbol{\nu}}_h)\,\d\boldsymbol{x}\\
     &=\int_{\Omega \setminus K_\rho} Q^k_3(\chi_h \boldsymbol{G}_h,\boldsymbol{\lambda})\,\d\boldsymbol{x}+ \int_{\Omega \setminus K_\rho} \left( Q^k_3(\chi_h \boldsymbol{G}_h,\widetilde{\boldsymbol{R}}_h^\top \widetilde{\boldsymbol{\nu}}_h)-Q^k_3(\chi_h \boldsymbol{G}_h, \boldsymbol{\lambda})\right )\,\d \boldsymbol{x}.
    \end{split}
\end{equation*}
Since $Q^k_3(\cdot,\boldsymbol{\lambda})$ is convex, by lower semicontinuity we deduce
\begin{equation}
    \liminf_{h \to 0^+} \int_{\Omega \setminus K_\rho} Q^k_3(\chi_h\boldsymbol{G}_h,\boldsymbol{\lambda})\,\d\boldsymbol{x} \geq \int_{\Omega \setminus K_\rho} Q^k_3(\boldsymbol{G},\boldsymbol{\lambda})\,\d\boldsymbol{x}.
\end{equation}
On the other hand,
\begin{equation}
    \lim_{h \to 0^+} \int_{\Omega \setminus K_\rho} \left( Q^k_3(\chi_h \boldsymbol{G}_h,\widetilde{\boldsymbol{R}}_h^\top \widetilde{\boldsymbol{\nu}}_h)-Q^k_3(\chi_h \boldsymbol{G}_h, \boldsymbol{\lambda})\right )\,\d \boldsymbol{x}=0.
\end{equation}
Indeed, recalling \eqref{eqn:definition-Q3}, 
\begin{equation*}
    \begin{split}
       |Q^k_3(\chi_h \boldsymbol{G}_h,\widetilde{\boldsymbol{R}}_h^\top \widetilde{\boldsymbol{\nu}}_h)-Q^k_3(\chi_h \boldsymbol{G}_h,\boldsymbol{\lambda})|&=|Q_3(\chi_h \boldsymbol{G}_h,\widetilde{\boldsymbol{R}}_h^\top\widetilde{\boldsymbol{\nu}}_h)-Q_3(\chi_h \boldsymbol{G}_h,\boldsymbol{\lambda})|\\
       &=|\mathbb{C}^{\widetilde{\boldsymbol{R}}_h^\top \widetilde{\boldsymbol{\nu}}_h} \,(\chi_h \boldsymbol{G}_h):(\chi_h \boldsymbol{G}_h) - \mathbb{C}^{\boldsymbol{\lambda}} (\chi_h \boldsymbol{G}_h):(\chi_h \boldsymbol{G}_h)|\\
       &\leq |\mathbb{C}^{\widetilde{\boldsymbol{R}}_h^\top \widetilde{\boldsymbol{\nu}}_h} - \mathbb{C}^{\boldsymbol{\lambda}}|\,|\chi_h \boldsymbol{G}_h|^2
    \end{split}
\end{equation*}
so that
\begin{equation*}
    \begin{split}
     \left |  \int_{\Omega \setminus K_\rho} \left( Q^k_3(\chi_h \boldsymbol{G}_h,\widetilde{\boldsymbol{R}}_h^\top \widetilde{\boldsymbol{\nu}}_h)-Q^k_3(\chi_h \boldsymbol{G}_h, \boldsymbol{\lambda})\right )\,\d \boldsymbol{x} \right | &\leq  \int_{\Omega \setminus K_\rho} |\mathbb{C}^{\widetilde{\boldsymbol{R}}_h^\top \widetilde{\boldsymbol{\nu}}_h} - \mathbb{C}^{\boldsymbol{\lambda}}|\,|\chi_h \boldsymbol{G}_h|^2 \,\d\boldsymbol{x}\\
     &\leq ||\mathbb{C}^{\widetilde{\boldsymbol{R}}_h^\top \widetilde{\boldsymbol{\nu}}_h} - \mathbb{C}^{\boldsymbol{\lambda}}||_{L^\infty(\Omega \setminus K_\rho;\rtt)} \,||\chi_h \boldsymbol{G}_h||^2_{L^2(\Omega;\rtt)},
    \end{split}
\end{equation*}
where the right-hand side goes to zero, as $h \to 0^+$. Thus we obtain
\begin{equation*}
    \liminf_{h \to 0^+} \int_\Omega Q_3^k(\chi_h \boldsymbol{G}_h,\widetilde{\boldsymbol{R}}_h^\top \widetilde{\boldsymbol{\nu}}_h)\,\d\boldsymbol{x} \geq \int_{\Omega \setminus K_\rho} Q_3^k(\boldsymbol{G},\boldsymbol{\lambda})\,\d\boldsymbol{x}
\end{equation*}
from which, letting $\rho \to 0^+$ and applying the Monotone Convergence
Theorem, we deduce
\begin{equation}
    \label{eqn:Q3-liminf}
    \liminf_{h \to 0^+} \int_\Omega Q_3^k(\chi_h \boldsymbol{G}_h,\widetilde{\boldsymbol{R}}_h^\top \widetilde{\boldsymbol{\nu}}_h)\,\d\boldsymbol{x} \geq \int_\Omega Q_3^k(\boldsymbol{G},\boldsymbol{\lambda})\,\d\boldsymbol{x}.
\end{equation}
For the second integral on the right-hand side of \eqref{eqn:lower-bound-el-1}, note that $|\omega(\boldsymbol{F},\boldsymbol{\nu})|\leq \overline{\omega}(|\boldsymbol{F}|)\,|\boldsymbol{F}|^2$ for every $\boldsymbol{F}\in \rtt$ and $\boldsymbol{\nu} \in \S^2$, with $\overline{\omega}$ defined as in \eqref{eqn:coupling_omega}. Thus, we obtain
\begin{equation*}
    \begin{split}
    \frac{1}{h^\beta} \int_\Omega \omega^k(h^{\beta/2}\,\chi_h \boldsymbol{G}_h, \widetilde{\boldsymbol{R}}_h^\top \widetilde{\boldsymbol{\nu}}_h)\,\d \boldsymbol{x} &=   \frac{1}{h^\beta} \int_\Omega \left (  \omega(h^{\beta/2}\,\chi_h \boldsymbol{G}_h, \widetilde{\boldsymbol{R}}_h^\top\widetilde{\boldsymbol{\nu}}_h) +
    k\, \gamma(h^\beta\,|\chi_h \boldsymbol{G}_h|^2) \right) \,\d \boldsymbol{x}\\
    &\leq  \int_\Omega \left (\overline{\omega}(h^{\beta/2}|\chi_h \boldsymbol{G}_h|)+k\,\frac{\gamma(h^\beta\,|\chi_h \boldsymbol{G}_h|^2)}{h^\beta\,|\chi_h \boldsymbol{G}_h|^2} \right) |\chi_h \boldsymbol{G}_h|^2\\
    & \leq   \left(\overline{\omega}(h^{\beta/4}) + k\,\overline{\gamma}(h^{\beta/2}) \right)\,||\chi_h \boldsymbol{G}_h||^2_{L^2(\Omega;\rtt)}.
    \end{split}
\end{equation*}
In the formula above, $\overline{\gamma}(t)\coloneqq\sup\{|\gamma(\tau)|/\tau^2:\,0<\tau\leq t\}$. By construction, $\overline{\gamma}(t)\to 0$, as $t \to 0^+$, and, by \eqref{eqn:coupling_omega}, we infer that 
\begin{equation}
\label{eqn:omega-liminf}
\lim_{h \to 0^+} \frac{1}{h^\beta} \int_\Omega \omega^k(h^{\beta/2}\,\chi_h \boldsymbol{G}_h, \widetilde{\boldsymbol{R}}_h^\top \widetilde{\boldsymbol{\nu}}_h)\,\d \boldsymbol{x}=0.
\end{equation}

Combining \eqref{eqn:Q3-liminf} and \eqref{eqn:omega-liminf}, we obtain 
\begin{equation*}
    \liminf_{h \to 0^+} E_h^\mathrm{el}({\boldsymbol{y}}_h,{\boldsymbol{m}}_h) \geq \frac{1}{2}\int_\Omega Q^k_3(\boldsymbol{G},\boldsymbol{\lambda})\,\d\boldsymbol{x} \geq \frac{1}{2}\int_\Omega Q^k_2(\boldsymbol{G}'',\boldsymbol{\lambda})\,\d\boldsymbol{x}.
\end{equation*}
Therefore, by \eqref{eqn:comparison-estimate}, for every $k \in \N$, we have
\begin{equation*}
    \liminf_{h \to 0^+} E_h^\mathrm{el}({\boldsymbol{y}}_h,{\boldsymbol{m}}_h) \geq \frac{1}{2}\int_\Omega Q^\inc_2(\boldsymbol{G}'',\boldsymbol{\lambda})\,\d\boldsymbol{x}-\frac{C}{\sqrt{k}}\int_\Omega |\boldsymbol{G}''|^2\,\d\boldsymbol{x},
\end{equation*}
from which, letting $k \to \infty$, we deduce
\begin{equation*}
    \liminf_{h \to 0^+} E_h^\mathrm{el}({\boldsymbol{y}}_h,{\boldsymbol{m}}_h) \geq \frac{1}{2}\int_\Omega Q^\inc_2(\boldsymbol{G}'',\boldsymbol{\lambda})\,\d\boldsymbol{x}.
\end{equation*}
Finally, recalling \eqref{eqn:G-affine-in-x3}, we conclude
\begin{equation*}
    \int_\Omega Q^\inc_2(\boldsymbol{G}'',\boldsymbol{\lambda})\,\d \boldsymbol{x}=\int_S Q^\inc_2(\boldsymbol{K},\boldsymbol{\lambda})\,\d\boldsymbol{x}'+\frac{1}{12} \int_S Q^\inc_2(\boldsymbol{L},\boldsymbol{\lambda})\,\d\boldsymbol{x}',
\end{equation*}
which gives \eqref{eqn:lower-bound-exc}.
\end{proof}

We now focus on the magnetostatic energy. Recall that, for any $(\boldsymbol{y},\boldsymbol{m}) \in \mathcal{Q}$, the corresponding stray field {\MMM potential} $\psi_{\boldsymbol{m}}$ is a weak solution of \eqref{eqn:maxwell_y}. More explicitly, this means that $\psi_{\boldsymbol{m}} \in V^{1,2}(\R^3)$ and there holds
\begin{equation}
\label{eqn:Maxwell-weak}
    \forall\, \varphi \in V^{1,2}(\R^3),\quad \int_{\R^3} \nabla \psi_{\boldsymbol{m}} \cdot \nabla \varphi\,\d\boldsymbol{\xi}=\int_{\R^3} \chi_{\Omega^{\boldsymbol{y}}}\boldsymbol{m} \cdot \nabla \varphi\,\d\boldsymbol{\xi}.
\end{equation}
In the following result, we claim the existence of such a weak solution and we collect some of its properties that are going to be used later.

\begin{lemma}[Weak solutions of the Maxwell equation]
\label{lem:maxwell}
Let $ (\boldsymbol{y},\boldsymbol{m}) \in \mathcal{Q}$. The Maxwell equation \eqref{eqn:maxwell_y} admits a weak solution $\psi_{\boldsymbol{m}} \in V^{1,2}(\R^3)$ which is unique up to additive constants and satisfies the following stability estimate:
\begin{equation}
    \label{eqn:Maxwell-stability}
    ||\nabla \psi_{\boldsymbol{m}}||_{L^2(\R^3;\R^3)}\leq ||\chi_{\Omega^{\boldsymbol{y}}}\boldsymbol{m}||_{L^2(\R^3;\R^3)}.
\end{equation}
Moreover, such a weak solution admits the following variational characterization:
\begin{equation}
    \label{eqn:Maxwell-variational}
    \psi_{\boldsymbol{m}} \in \mathrm{argmin} \left\{ \int_{\R^3} |\nabla \varphi - \chi_{\Omega^{\boldsymbol{y}}}\boldsymbol{m}|^2 \,\d\boldsymbol{\xi}:\,\varphi\in V^{1,2}(\R^3) \right \}.
\end{equation}
\end{lemma}
For the proof of the existence and the stability of weak solutions, we refer to \cite[Proposition 8.8]{barchiesi.henao.moracorral}. For the proof of  \eqref{eqn:Maxwell-variational}, note that \eqref{eqn:Maxwell-weak} is exactly the weak form of the Euler-Lagrange equation of the convex functional
\begin{equation*}
    \varphi \mapsto \frac{1}{2}\int_{\R^3}|\nabla \varphi - \chi_{\Omega^{\boldsymbol{y}}}\boldsymbol{m}|^2\,\d\boldsymbol{\xi}
\end{equation*}
on $V^{1,2}(\R^3)$.

We now prove that for the magnetostatic energies we actually have the convergence to the corresponding term in the limiting energy. The proof is adapted from \cite[Proposition 4.1]{gioia.james}. Recall the definitions of $E_h$ and $E$ in \eqref{eqn:energy_Eh} and \eqref{eqn:energy_E}, respectively.

\begin{proposition}[Convergence of the magnetostatic energy]
\label{prop:liminf-mag}
Let $((\boldsymbol{y}_h,\boldsymbol{m}_h))_h \subset \mathcal{Q}$ be such that $E_h(\boldsymbol{y}_h,\boldsymbol{m}_h)\leq C$ for every $h>0$. Then, given the map $\boldsymbol{\lambda}\in W^{1,2}(S;\S^2)$ identified in Proposition \ref{prop:comp-mag}, we have, up to subsequences, as $h \to 0^+$:
\begin{equation}
    \label{eqn:convergence-mag}
    \lim_{h \to 0^+}E^\mathrm{mag}_h({\boldsymbol{y}}_h, {\boldsymbol{m}}_h)= E^\mathrm{mag}(\boldsymbol{\lambda}).
\end{equation}
\end{proposition}
\begin{proof}
We follow the notation of Propositions \ref{prop:comp-def} and \ref{prop:comp-mag}. Thus, $\widetilde{\boldsymbol{y}}_h=\boldsymbol{T}_h \circ \boldsymbol{y}_h$ and $\widetilde{\boldsymbol{m}}_h=\boldsymbol{Q}_h^\top \boldsymbol{m}_h \circ \boldsymbol{T}_h^{-1}$, where $\boldsymbol{T}_h$ is the rigid motion given by $\boldsymbol{T}_h(\boldsymbol{\xi})=\boldsymbol{Q}_h^\top \boldsymbol{\xi}-\boldsymbol{c}_h$ for every $\boldsymbol{\xi}\in \R^3$, $\boldsymbol{Q}_h \in SO(3)$ and $\boldsymbol{c}_h \in \R^3$.
For simplicity, denote by $\psi_h$ and $\widetilde{\psi}_h$ the stray field {\MMM potentials} corresponding to $(\boldsymbol{y}_h,\boldsymbol{m}_h)$ and $(\widetilde{\boldsymbol{y}}_h,\widetilde{\boldsymbol{m}}_h)$, respectively. Recall that these are defined up to additive constants and are weak solutions, in the sense of \eqref{eqn:Maxwell-weak}, of the following two equations:
\begin{equation}
\label{eqn:two-maxwell}
 \text{$\Delta \psi_h=\text{div}\left (\chi_{\Omega^{\boldsymbol{y}_h}}\boldsymbol{m}_h\right )$ in $\R^3$, \hspace{8mm} $\Delta \widetilde{\psi}_h=\text{div}(\chi_{\Omega^{\widetilde{\boldsymbol{y}}_h}}\widetilde{\boldsymbol{m}}_h)$ in $\R^3$.} 
\end{equation}

For every $\varphi \in V^{1,2}(\R^3)$, we compute
\begin{equation}
    \label{eqn:weak-maxwell-composition}
    \begin{split}
        \int_{\R^3} \nabla (\psi_h \circ \boldsymbol{T}_h^{-1}) \cdot \nabla \varphi\,\d\boldsymbol{\xi}&=\int_{\R^3} \boldsymbol{Q}_h^\top \nabla \psi_h \circ \boldsymbol{T}_h^{-1} \cdot \nabla \varphi\,\d\boldsymbol{\xi}=\int_{\R^3} \nabla \psi_h \cdot \boldsymbol{Q}_h \nabla \varphi \circ \boldsymbol{T}_h\,\d\boldsymbol{\xi}\\
        &=\int_{\R^3} \nabla \psi_h \cdot \nabla (\varphi \circ \boldsymbol{T}_h)\,\d\boldsymbol{\xi}=\int_{\R^3} \chi_{\Omega^{\boldsymbol{y}_h}}\boldsymbol{Q}_h^\top \boldsymbol{m}_h \cdot  \nabla \varphi \circ \boldsymbol{T}_h\,\d \boldsymbol{\xi}\\
        &=\int_{\R^3} \chi_{\Omega^{\widetilde{\boldsymbol{y}}_h}}\widetilde{\boldsymbol{m}}_h\cdot \nabla \varphi\,\d\boldsymbol{\xi},
    \end{split}
\end{equation}
where we used the chain rule, the first equation in \eqref{eqn:two-maxwell}, the change-of-variable formula and the identity $\chi_{\Omega^{\widetilde{\boldsymbol{y}}_h}}\widetilde{\boldsymbol{m}}_h=(\chi_{\Omega^{\boldsymbol{y}_h}}\boldsymbol{Q}_h^\top \boldsymbol{m}_h)\circ \boldsymbol{T}_h^{-1}$. From \eqref{eqn:weak-maxwell-composition}, we deduce that $\psi_h \circ \boldsymbol{T}_h^{-1}$ is a weak solution of the second equation in \eqref{eqn:two-maxwell}. Therefore we can assume that $\widetilde{\psi}_h=\psi_h \circ \boldsymbol{T}_h^{-1}$. In this case, by the chain rule and change-of-variable formula, we have 
\begin{equation*}
\begin{split}
    E_h^\text{mag}(\widetilde{\boldsymbol{y}}_h,\widetilde{\boldsymbol{m}}_h)&=\frac{1}{2h}\int_{\R^3}|\nabla \widetilde{\psi}_h|^2\,\d\boldsymbol{\xi}=\frac{1}{2h}\int_{\R^3} |\boldsymbol{Q}_h^\top \nabla \psi_h \circ \boldsymbol{T}_h^{-1}|^2\,\d\boldsymbol{\xi}\\
    &= \frac{1}{2h}\int_{\R^3} |\boldsymbol{Q}_h^\top \nabla \psi_h|^2\,\d\boldsymbol{\xi}=\frac{1}{2h}\int_{\R^3} |\nabla \psi_h|^2\,\d\boldsymbol{\xi}=E_h^\text{mag}({\boldsymbol{y}}_h,{\boldsymbol{m}}_h).
\end{split}
\end{equation*}

Testing the weak form \eqref{eqn:Maxwell-weak}  of the second equation in \eqref{eqn:two-maxwell} with $\varphi=\widetilde{\psi}_h$, we obtain
\begin{equation*}
    E_h^{\mathrm{mag}}(\widetilde{\boldsymbol{y}}_h,\widetilde{\boldsymbol{m}}_h)=\frac{1}{2h}\int_{\R^3}|\nabla \widetilde{\psi}_h|^2\,\d\boldsymbol{\xi}=\frac{1}{2h}\int_{\R^3} \chi_{\Omega^{\widetilde{\boldsymbol{y}}_h}}\widetilde{\boldsymbol{m}}_h \cdot \nabla \widetilde{\psi}_h \, \d \boldsymbol{\xi}.
\end{equation*}
Define $\widehat{\psi}_h \coloneqq \widetilde{\psi}_h \circ \boldsymbol{z}_h$ and $\widetilde{\boldsymbol{\mu}}_h \coloneqq (\chi_{\Omega^{\widetilde{\boldsymbol{y}}_h}}\widetilde{\boldsymbol{m}}_h) \circ \boldsymbol{z}_h$. Using the change-of-variable formula, we compute
\begin{equation}
        \label{eqn:gioia-james-energy}
        \frac{1}{2h}\int_{\R^3} \chi_{\Omega^{\widetilde{\boldsymbol{y}}_h}}\widetilde{\boldsymbol{m}}_h \cdot \nabla \widetilde{\psi}_h \, \d \boldsymbol{\xi}=  \frac{1}{2} \int_{\R^3} \widetilde{\boldsymbol{\mu}}_h \cdot \nabla_h \widehat{\psi}_h\,\d\boldsymbol{x}.
\end{equation}
We claim that $\widehat{\psi}_h \to 0$ in $V^{1,2}(\R^3)$, that is, $\widehat{\psi}_h \to 0$ in $L^2_\loc(\R^3)$ and $\nabla \widehat{\psi}_h \to \boldsymbol{0}$ in $L^2(\R^3;\R^3)$. 
By \eqref{eqn:Maxwell-stability}, we have
\begin{equation}
    \int_{\R^3} |\nabla \widetilde{\psi}_h|^2\,\d\boldsymbol{\xi} \leq \int_{\R^3} |\chi_{\Omega^{\widetilde{\boldsymbol{y}}_h
    }}\widetilde{\boldsymbol{m}}_h|^2\,\d\boldsymbol{\xi}=    \leb(\Omega^{\widetilde{\boldsymbol{y}}_h
    })=h\,\lebt(S),
\end{equation}
where we used the fact that $\widetilde{\boldsymbol{m}}_h$ takes values in $\S^2$ and we applied the area formula. From this, using the change-of-variable formula, we obtain
\begin{equation*}
    \int_{\R^3} |\nabla_h \widehat{\psi}_h|^2\,\d \boldsymbol{x}=\int_{\R^3} |\nabla \widetilde{\psi}_h \circ \boldsymbol{z}_h|^2\,\d\boldsymbol{x}=\frac{1}{h}\int_{\R^3} | \nabla \widetilde{\psi}_h|^2\,\d\boldsymbol{\xi}\leq \lebt(S).
\end{equation*}
Therefore $||\nabla_h \widehat{\psi}_h||_{L^2(\R^3;\R^3)}\leq C$. In particular, $||\partial_3 \widehat{\psi}_h/h||_{L^2(\R^3)}\leq C$, so that $\partial_3 \widehat{\psi}_h \to 0$ in $L^2(\R^3)$ and there exists $l\in L^2(\R^3)$ such that, up to subsequences, $\partial_3 \widehat{\psi}_h/h \wk l$ in $L^2(\R^3)$. As $||\nabla \widehat{\psi}_h||_{L^2(\R^3;\R^3)}\leq ||\nabla_h \widehat{\psi}_h||_{L^2(\R^3;\R^3)} \leq C$, there exists $\Psi \in L^2(\R^3;\R^3)$ such that, up to subsequences, $\nabla \widehat{\psi}_h \wk \Psi$ in $L^2(\R^3;\R^3)$. Consider the Hilbert space ${V}^{1,2}(\R^3)/\R$, i.e. the quotient of $V^{1,2}(\R^3)$ with respect to constant functions, whose inner product is given by
\begin{equation*}
    ([\varphi],[\psi])_{{V}^{1,2}(\R^3)/\R}=\int_{\R^3} \nabla \varphi \cdot \nabla \psi\,\d\boldsymbol{\xi}.
\end{equation*}
Thus $||[\widehat{\psi}_h]||_{V^{1,2}(\R^3)/\R}=||\nabla \widehat{\psi}_h||_{L^2(\R^3;\R^3)}\leq C$, so that there exists $\widehat{\psi}\in V^{1,2}(\R^3)$ such that, up to subsequences, $[\widehat{\psi}_h] \wk [\widehat{\psi}]$ in $V^{1,2}(\R^3)/\R$, namely
\begin{equation*}
    \int_{\R^3} \nabla \widehat{\psi}_h \cdot \nabla \varphi\,\d\boldsymbol{\xi} \to \int_{\R^3} \nabla \widehat{\psi} \cdot \nabla \varphi\,\d\boldsymbol{\xi}
\end{equation*}
for every $\varphi \in V^{1,2}(\R^3)$. Testing the weak convergence of $(\nabla \widehat{\psi}_h)_h$ in $L^2(\R^3;\R^3)$ with gradients of functions in $V^{1,2}(\R^3)$, we obtain that $\Psi$ and $\nabla \widehat{\psi}$ differ by a constant vector but, due to their integrability on the whole space, we necessarily have $\Psi=\nabla \widehat{\psi}$. Thus $\nabla \widehat{\psi}_h \wk \nabla\widehat{\psi}$ in $L^2(\R^3;\R^3)$. From $\partial_3 \widehat{\psi}_h \to 0$ in $L^2(\R^3)$, we deduce $\partial_3 \widehat{\psi}=0$. Then, for any  $a,b\in \R$ with $a<b$, we have
\begin{equation*}
    \int_{\R^3} |\nabla \widehat{\psi}|^2\,\d\boldsymbol{\xi}=\int_{\R^3} |\nabla' \widehat{\psi}|^2\,\d\boldsymbol{\xi} \geq \int_a^b \int_{\R^2} |\nabla' \widehat{\psi}|^2\,\d\boldsymbol{\xi}'\,\d\xi_3=(b-a)\,||\nabla'\widehat{\psi}||^2_{L^2(\R^2;\R^2)}.
\end{equation*}
Since $\nabla \widehat{\psi} \in L^2(\R^3;\R^3)$ and $a$ and $b$ are arbitrary, we necessarily have that $\nabla' \widehat{\psi}=\boldsymbol{0}'$ almost everywhere. 

Going back to \eqref{eqn:gioia-james-energy}, we have
\begin{equation*}
    E_h^{\text{mag}}(\widetilde{\boldsymbol{y}}_h,\widetilde{\boldsymbol{m}}_h)=\frac{1}{2}\int_{\R^3} \widetilde{\boldsymbol{\mu}}_h \cdot \nabla_h \widehat{\psi}_h \,\d\boldsymbol{x}= \frac{1}{2} \int_{\R^3} \widetilde{\boldsymbol{\mu}}'_h \cdot \nabla'\widehat{\psi}_h\,\d\boldsymbol{x}+\frac{1}{2}\int_{\R^3} \widetilde{\mu}^3_h \,\frac{\partial_3 \widehat{\psi}_h}{h}\,\d\boldsymbol{x}.
\end{equation*}
Since $\nabla'\widehat{\psi}_h \wk \boldsymbol{0}'$ in $L^2(\R^3;\R^2)$, $\partial_3 \widehat{\psi}_h /h \wk l$ in $L^2(\R^3)$, and \eqref{eqn:conv-mu} holds, passing to the limit, as $h \to 0^+$, we obtain
\begin{equation}
    \label{eqn:gioia-james-convergence}
    \lim_{h \to 0^+} E_h^{\text{mag}}(\widetilde{\boldsymbol{y}}_h,\widetilde{\boldsymbol{m}}_h)= \int_{\R^3} \chi_\Omega \lambda^3\,l\,\d\boldsymbol{x}.
\end{equation}
Therefore, in order to prove \eqref{eqn:convergence-mag}, we are left to show that $l=\chi_\Omega \lambda^3$. {\MMM To do this, we adopt an argument from \cite[Lemma 1]{carbou}.}

Recall that, by \eqref{eqn:Maxwell-variational}, we have 
\begin{equation*}
    \widetilde{\psi}_h \in \mathrm{argmin} \left \{ \int_{\R^3}|\nabla \varphi-\chi_{\Omega^{\widetilde{\boldsymbol{y}}_h}}\widetilde{\boldsymbol{m}}_h|^2\,\d\boldsymbol{\xi}:\,\varphi \in V^{1,2}(\R^3) \right \}.
\end{equation*}
Changing variables, we  deduce that $\widehat{\psi}_h$ admits an analogous variational characterization, namely
\begin{equation}
    \label{eqn:scaled-Maxwell-variational}
    \widehat{\psi}_h \in \mathrm{argmin} \left \{ \int_{\R^3}|\nabla_h \varphi-\widetilde{\boldsymbol{\mu}}_h|^2\,\d\boldsymbol{x}:\,\varphi \in V^{1,2}(\R^3) \right \}.
\end{equation}
{\MMM
The weak form of the Euler-Lagrange equation corresponding to \eqref{eqn:scaled-Maxwell-variational} reads as follows:
\begin{equation*}
	\forall\,\varphi\in V^{1,2}(\R^3),\quad \int_{\R^3} \nabla_h \widehat{\psi}_h \cdot \nabla_h \varphi\,\d\boldsymbol{x}=\int_{\R^3}  \widetilde{\boldsymbol{\mu}}_h\cdot \nabla_h \varphi \,\d\boldsymbol{x}.
\end{equation*}
Expanding the scalar product and multiplying by $h$, we have
\begin{equation*}
	h\int_{\R^3} (\nabla'\widehat{\psi}_h-\widetilde{\boldsymbol{\mu}}_h')\cdot \nabla' \varphi\,\d\boldsymbol{x}+\int_{\R^3} \left (\frac{\partial_3\widehat{\psi}_h}{h}-\widetilde{\mu}_h^3 \right)\,\partial_3 \varphi\,\d\boldsymbol{x}=0,
\end{equation*}
from which, passing to the limit, as $h \to 0^+$, taking into account that $\nabla'\widehat{\psi}_h \wk \boldsymbol{0}'$ in $L^2(\R^3;\R^2)$, $\widetilde{\boldsymbol{\mu}}_h \to \chi_\Omega\boldsymbol{\lambda}$ in $L^2(\R^3;\R^3)$ and $\partial_3 \widehat{\psi}_h\wk l$ in $L^2(\R^3)$, we obtain
\begin{equation*}
	\int_{\R^3} (l-\chi_\Omega\lambda^3)\,\partial_3 \varphi\,\d\boldsymbol{x}=0.
\end{equation*}
From this equation, as $\varphi$ is arbitrary, we deduce that the difference $l-\chi_\Omega\lambda^3$ does not depend on $x_3$. However, since $l \in L^2(\R^3)$ and $\chi_\Omega\lambda^3 \in L^2(\R^3)$, this entails $l=\chi_\Omega\lambda^3$, as claimed.
}
\end{proof}

We are now able to prove our first main result.

{\MMM
\begin{proof}[Proof of Theorem \ref{thm:comp-lower-bound}]
The proof of compactness follows combining Propositions \ref{prop:comp-def} and \ref{prop:comp-mag}. 
More precisely, \eqref{eqn:lb-comp-u} and \eqref{eqn:lb-comp-v} are proved in \eqref{eqn:comp-u} and \eqref{eqn:comp-v}, respectively, while \eqref{eqn:lb-comp-my} is obtained in \eqref{eqn:conv-m-comp-y}. The proof of the lower bound \eqref{eqn:lb-liminf-ineq} follows immediately from Propositions \ref{prop:liminf-exc}, \ref{prop:liminf-el} and \ref{prop:liminf-mag}.
\end{proof}
}

\section{Optimality of the lower bound}
\label{sec:opt}

In this section we show that the lower bound given by the energy \eqref{eqn:energy_E} is optimal by proving the existence of a recovery sequence for any admissible limiting state. 

We employ the common strategy of arguing first for a dense family of  states and then  regain the general statement by an approximation procedure. Note that the space $C^\infty(\closure{S};\S^2)$ is dense in $W^{1,2}(S;\S^2)$, that is, for every $\boldsymbol{\lambda} \in W^{1,2}(S;\S^2)$ there exists a sequence $(\boldsymbol{\lambda}_n)\subset C^\infty(\closure{S};\S^2)$ such that $\boldsymbol{\lambda}_n \to \boldsymbol{\lambda}$ in $W^{1,2}(S;\R^3)$, as $n \to \infty$ \cite[Theorem 2.1]{hajlasz}.

In the next result we show how to construct a recovery sequence for smooth admissible limiting states. In order to deal with the incompressibility constraint we adopt an argument from \cite{li.chermisi} (see also \cite{conti.dolzmann}). Recall the definitions of the energies $E_h$ and $E$ in \eqref{eqn:energy_Eh} and \eqref{eqn:energy_E}, respectively, and the notation introduced in \eqref{eqn:averaged-displacements} and \eqref{eqn:gamma-notation-M}.

\begin{proposition}[Recovery sequence]
\label{prop:recovery-sequence} Let {\MMM $\boldsymbol{u}\in C^\infty(\closure{S};\R^2)$, }
$v \in C^\infty(\closure{S})$ and $\boldsymbol{\lambda} \in  C^\infty(\closure{S};\S^2)$, and let {\MMM $\boldsymbol{a},\boldsymbol{b}\in C^\infty(\closure{S};\R^3)$ satisfy}
\begin{equation}
	\label{eqn:recovery-traceless-a}
	{\MMM
	\tr \left (  \left( \renewcommand\arraystretch{1.2} \begin{array}{@{}c|c@{}}   \sym \nabla' \boldsymbol{u}  & \boldsymbol{0}'\\ \hline  (\boldsymbol{0}')^\top & 0 \end{array} \right) + \boldsymbol{a} \otimes \boldsymbol{e}_3+\boldsymbol{e}_3 \otimes \boldsymbol{a} \right )=\div'\boldsymbol{u}+2 a^3=0 \hskip 5 pt \text{in $S$.}
	}
\end{equation}
\begin{equation}
    \label{eqn:recovery-traceless-b}
    \tr \left ( - \left( \renewcommand\arraystretch{1.2} \begin{array}{@{}c|c@{}}   (\nabla')^2 v  & \boldsymbol{0}'\\ \hline  (\boldsymbol{0}')^\top & 0 \end{array} \right) + \boldsymbol{b} \otimes \boldsymbol{e}_3+\boldsymbol{e}_3 \otimes \boldsymbol{b} \right )=- \Delta'v+2 b^3=0 \hskip 5 pt \text{in $S$.}
\end{equation}
Then, there exists a sequence of admissible states $((\boldsymbol{y}_h,\boldsymbol{m}_h))_h \subset \mathcal{Q}$ such that, as $h \to 0^+$, we have:
\begin{align}
        \label{eqn:recovery-y}
        &\text{$\boldsymbol{y}_h \to \boldsymbol{z}_0$ in $W^{1,p}(\Omega;\R^3)$;} \hspace{66mm}\\
        \label{eqn:recovery-u}
        &{\MMM
        \text{$\boldsymbol{u}_h \coloneqq \mathcal{U}_h(\boldsymbol{y}_n) \to \boldsymbol{u}$ in $W^{1,2}(S;\R^2)$;}
    	}\\
        \label{eqn:recovery-v}
        &{\MMM
        \text{$v_h \coloneqq \mathcal{V}_h(\boldsymbol{y}_h) \to v$ in $W^{1,2}(S)$;}
    	}\\
        \label{eqn:recovery-m-circ-y}
        &\text{${\boldsymbol{m}}_h \circ {\boldsymbol{y}}_h \to \boldsymbol{\lambda}$ in $L^r(\Omega;\R^3)$ for every $1\leq r <\infty$,}\\
        \label{eqn:recovery-mu}
        &{\MMM
        \text{$\boldsymbol{\mu}_h\coloneqq \mathcal{M}_h(\boldsymbol{y}_h,\boldsymbol{m}_h)\to \chi_\Omega \boldsymbol{\lambda}$ in $L^r(\R^3;\R^3)$ for every $1\leq r< \infty$.}
    	}
    \end{align}
Moreover, there hold:
{\MMM
\begin{equation}
	\label{eqn:recovery-elastic}
	\begin{split}
		\lim_{h \to 0^+} E^{\text{\em el}}_h(\boldsymbol{y}_h,\boldsymbol{m}_h)&=\frac{1}{2}\int_S Q_2^{\text{\em inc}}\left (\left( \renewcommand\arraystretch{1.2} \begin{array}{@{}c|c@{}}   \sym \nabla' \boldsymbol{u}  & \boldsymbol{0}'\\ \hline  (\boldsymbol{0}')^\top & 0 \end{array} \right) + \boldsymbol{a} \otimes \boldsymbol{e}_3+\boldsymbol{e}_3 \otimes \boldsymbol{a},\boldsymbol{\lambda} \right )\,\d\boldsymbol{x}'\\
		&+\frac{1}{24}\int_S Q_2^{\text{\em inc}}\left (-\left( \renewcommand\arraystretch{1.2} \begin{array}{@{}c|c@{}}   (\nabla')^2 v  & \boldsymbol{0}'\\ \hline  (\boldsymbol{0}')^\top & 0 \end{array} \right) + \boldsymbol{b} \otimes \boldsymbol{e}_3+\boldsymbol{e}_3 \otimes \boldsymbol{b},\boldsymbol{\lambda} \right )\,\d\boldsymbol{x}';
	\end{split}
\end{equation}
}
\begin{align}
	\label{eqn:recovery-exchange}
	&\hspace*{5mm}\lim_{h \to 0^+} E^{\text{\em exc}}_h(\boldsymbol{y}_h,\boldsymbol{m}_h)=E^{\text{\em exc}} (\boldsymbol{\lambda}); \hspace*{77mm}\\
	\label{eqn:recovery-magnetostatic}
	&\hspace*{5mm}\lim_{h \to 0^+} E^{\text{\em mag}}_h(\boldsymbol{y}_h,\boldsymbol{m}_h)=E^{\text{\em mag}} (\boldsymbol{\lambda}).
\end{align}
\end{proposition}
\begin{proof}
For convenience of the reader we subdivide the proof into six steps.

\textbf{Step 1.} In this step we exhibit the general structure of deformations of the recovery sequence. We can assume that {\MMM $\boldsymbol{u}$,$v$, $\boldsymbol{a}$ and $\boldsymbol{b}$} are all defined and smooth on the whole space. Consider the sequence of compressible deformations $(\overline{\boldsymbol{y}}_h) \subset  C^\infty(\R^3;\R^3)$ defined according to the ansatz of \cite{friesecke.james.mueller2} for the linearized von K\'{a}rm\'{a}n regime, namely
\begin{equation*}
\overline{\boldsymbol{y}}_h \coloneqq \boldsymbol{z}_h + {\MMM h^{\beta/2} \left ( \begin{matrix}  \boldsymbol{u}  \\ 0\end{matrix}   \right )}+ h^{\beta/2-1} \left ( \begin{matrix}  \boldsymbol{0}'  \\ v\end{matrix}   \right )-h^{\beta/2} x_3 \left ( \begin{matrix}  \nabla'v  \\ 0\end{matrix}   \right )+{\MMM2 h^{\beta/2+1} x_3 \boldsymbol{a}}+h^{\beta/2+1} x_3^2 \boldsymbol{b}.
\end{equation*}
Set $\overline{\boldsymbol{F}}_h \coloneqq \nabla_h \overline{\boldsymbol{y}}_h$. We compute
\begin{equation*}
	\begin{split}
		\overline{\boldsymbol{F}}_h&= \boldsymbol{I}+{\MMM h^{\beta/2}\left( \renewcommand\arraystretch{1.2} \begin{array}{@{}c|c@{}}   \nabla'\boldsymbol{u}  & \boldsymbol{0}'\\ \hline  (\boldsymbol{0}')^\top & 0 \end{array} \right)}+  h^{\beta/2-1} \left( \renewcommand\arraystretch{1.2} \begin{array}{@{}c|c@{}}   \mathbf{O}''  & -\nabla' v\\ \hline  \nabla'v^\top & 0 \end{array} \right) -h^{\beta/2} x_3 \left( \renewcommand\arraystretch{1.2} \begin{array}{@{}c|c@{}}   (\nabla')^2 v  & \boldsymbol{0}'\\ \hline  (\boldsymbol{0}')^\top & 0 \end{array} \right)\\
		&+ {\MMM 2h^{\beta/2} \boldsymbol{a} \otimes \boldsymbol{e}_3}+2h^{\beta/2} x_3 \boldsymbol{b} \otimes \boldsymbol{e}_3 + O(h^{\beta/2+1}).
	\end{split}   
\end{equation*}
Using the identity $(\boldsymbol{I}+\boldsymbol{G})^\top (\boldsymbol{I}+\boldsymbol{G})=\boldsymbol{I}+\boldsymbol{G}+\boldsymbol{G}^\top + \boldsymbol{G}^\top \boldsymbol{G}$ for every $\boldsymbol{G}\in \rtt$ and that $\beta>6$, we obtain
\begin{equation}
\label{eqn:y-comp-nonlinear}
\overline{\boldsymbol{F}}_h^\top \overline{\boldsymbol{F}}_h= \boldsymbol{I}+2h^{\beta/2}({\MMM\boldsymbol{A}+x_3\,\boldsymbol{B}})+O(h^{\beta/2+1}),
\end{equation}
where
\begin{equation}
	\label{eqn:matrix-B}
	{\MMM\boldsymbol{A} \coloneqq  \left( \renewcommand\arraystretch{1.2} \begin{array}{@{}c|c@{}}    \sym \nabla'\boldsymbol{u}  & \boldsymbol{0}'\\ \hline  (\boldsymbol{0}')^\top & 0 \end{array} \right) + \boldsymbol{a} \otimes \boldsymbol{e}_3 + \boldsymbol{e}_3 \otimes \boldsymbol{a},} \qquad \boldsymbol{B} \coloneqq - \left( \renewcommand\arraystretch{1.2} \begin{array}{@{}c|c@{}}   (\nabla')^2 v  & \boldsymbol{0}'\\ \hline  (\boldsymbol{0}')^\top & 0 \end{array} \right) + \boldsymbol{b} \otimes \boldsymbol{e}_3 + \boldsymbol{e}_3 \otimes \boldsymbol{b}.
\end{equation}
Then, using the formula $\det(\boldsymbol{I}+\boldsymbol{G})=1+\tr\, \boldsymbol{G}+\tr\,\cof\, \boldsymbol{G} + \det \boldsymbol{G}$ for every $\boldsymbol{G}\in \rtt$, and noting that ${\MMM\tr\,\boldsymbol{A}}=\tr\,\boldsymbol{B}=0$ by \eqref{eqn:recovery-traceless-a} and \eqref{eqn:recovery-traceless-b}, we deduce 
\begin{equation*}
	\det(\overline{\boldsymbol{F}}_h^\top \overline{\boldsymbol{F}}_h)=1+2h^\beta {\MMM(P+x_3Q+x_3^2 R)}+O(h^{\beta/2+1}),
\end{equation*}
where {\MMM$P$, $Q$ and $R$} are polynomials in the variable $\boldsymbol{x}'$ which depend on {\MMM$\boldsymbol{u}$, $v$, $\boldsymbol{a}$ and $\boldsymbol{b}$}. Taking the square root, we obtain
\begin{equation}
    \label{eqn:y-comp-determinant}
    \det \nabla_h \overline{\boldsymbol{y}}_h= 1+h^\beta {\MMM(P+x_3Q+x_3^2 R)}+O(h^{\beta/2+1}).
\end{equation}
Fix $h>0$ and consider a function $\eta_h \colon \Omega \to \R$. Define $\boldsymbol{y}_h(\boldsymbol{x}',x_3) \coloneqq \overline{\boldsymbol{y}}_h(\boldsymbol{x}',\eta_h(\boldsymbol{x}',x_3))$ for $\boldsymbol{x}\in \Omega$. Assuming $\eta_h$ to be differentiable and setting $\nabla_h \boldsymbol{y}_h\coloneqq\boldsymbol{F}_h$, for every $\boldsymbol{x}\in \Omega$, we compute
\begin{equation}
    \label{eqn:y-inc-scal-grad}
    \boldsymbol{F}_h(\boldsymbol{x}',x_3)=  \overline{\boldsymbol{F}}_h(\boldsymbol{x}',\eta_h(\boldsymbol{x}',x_3)) \boldsymbol{N}_h(\boldsymbol{x}',x_3),
\end{equation}
where
\begin{equation}
    \label{eqn:N-h}
    \boldsymbol{N}_h \coloneqq\left( \renewcommand\arraystretch{1.2} \begin{array}{@{}c|c@{}}   \boldsymbol{I}  & \boldsymbol{0}'\\ \hline  
    \begin{matrix}  h \nabla' (\eta_h)^{\top} \end{matrix}
     & \partial_3 \eta_h \end{array} \right).
\end{equation}
In particular
\begin{equation}
    \label{eqn:y-inc-determinant}
    \det \boldsymbol{F}_h(\boldsymbol{x}',x_3)= \det  \overline{\boldsymbol{F}}_h(\boldsymbol{x}',\eta_h(\boldsymbol{x}',x_3))\; \partial_3 \eta_h(\boldsymbol{x}',x_3) 
\end{equation}
for every $\boldsymbol{x} \in \Omega$. Note that, since $\beta>2$, by \eqref{eqn:y-comp-determinant} we have $\det  \overline{\boldsymbol{F}}_h=1+O(h^{\beta/2+1})$ and hence  $1/2 \leq  \det  \overline{\boldsymbol{F}}_h \leq 2$ for $h \ll 1$. Thus we can define
\begin{equation}
\label{eqn:fh-definition}
    {\MMM\Phi_h}(\boldsymbol{x}',x_3) \coloneqq \left ( \det \overline{\boldsymbol{F}}_h(\boldsymbol{x}',x_3) \right)^{-1}= \left ( 1+h^\beta{\MMM \left  ( P(\boldsymbol{x}')+x_3Q(\boldsymbol{x}')+x_3^2R(\boldsymbol{x}')\right )}+O(h^{\beta/2+1})\right)^{-1} 
\end{equation}
for every $(\boldsymbol{x}',x_3) \in S \times (-1,1)$. In view of \eqref{eqn:y-inc-determinant}, if $\eta_h$ is a solution of the following Cauchy problem
\begin{equation}
    \label{eqn:ode}
    \begin{cases} 
    \partial_3 \eta_h(\boldsymbol{x}',x_3)={\MMM\Phi_h}(\boldsymbol{x}',\eta_h(\boldsymbol{x}',x_3)) & \text{in $I$,}\\
    \eta_h(\boldsymbol{x}',0)=0,
    \end{cases}
\end{equation}
then $\det \nabla_h \boldsymbol{y}_h=\det \boldsymbol{F}_h=1$ in $\Omega$. 
Note that, regarding $\boldsymbol{x}'\in S$ as a parameter and $x_3\in I$ as the only variable, we can see the equation in \eqref{eqn:ode} as an ordinary differential equation. 

\textbf{Step 2.} In this step we prove the existence of a unique solution $\eta_h$ of \eqref{eqn:ode} and that this is  of class $C^1$. First note that, by definition, ${\MMM\Phi_h} \in C^\infty(\closure{S}\times [-1,1])$ and $1/2 < f_h < 2$ for $h \ll 1$. Using this fact and \eqref{eqn:y-comp-determinant}, we deduce the following estimates:
\begin{align}
\label{eqn:fh-estimates-1st}
    ||\partial_i {\MMM\Phi_h}||_{C^0(\closure{S}\times [-1,1])} \leq C h^{\beta/2+1}&, \hskip 14 pt ||\partial_3 {\MMM\Phi_h}||_{C^0(\closure{S}\times [-1,1])} \leq C h^{\beta/2+1},\\
\label{eqn:fh-estimates-2nd}
    ||\partial_3 \partial_i {\MMM\Phi_h}||_{C^0(\closure{S}\times [-1,1])} \leq C h^{\beta/2+1}&, \hskip 14 pt ||\partial_3^2 {\MMM\Phi_h}||_{C^0(\closure{S}\times [-1,1])} \leq C h^{\beta/2+1},
\end{align}
for $i \in \{1,2\}$ and $h \ll 1$.
In particular, for every $\boldsymbol{x}' \in S$, all the maps ${\MMM\Phi_h}(\boldsymbol{x}',\cdot)$, $\partial_i {\MMM\Phi_h}(\boldsymbol{x}',\cdot)$ and $\partial_3 {\MMM\Phi_h}(\boldsymbol{x}',\cdot)$ are Lipschitz continuous on $[-1,1]$. Consider the closed set $\mathcal{X} \coloneqq C^0(\closure{\Omega};[-1,1])$ of the Banach space $C^0(\closure{\Omega})$. 
Any solution of \eqref{eqn:ode} is given by a fixed point of the operator $\mathcal{T}_h \colon C^0(\closure{\Omega}) \to C^0(\closure{\Omega})$ defined by setting
\begin{equation*}
     \mathcal{T}_h (\eta)(\boldsymbol{x}',x_3) \coloneqq \int_0^{x_3} {\MMM\Phi_h}(\boldsymbol{x}',\eta(\boldsymbol{x}',t))\,\d t
\end{equation*}
for every $\eta \in C^0(\closure{\Omega})$ and $\boldsymbol{x}\in \Omega$.
Note that, for every $\eta \in \mathcal{X}$ and $\boldsymbol{x}\in \Omega$, we have 
\begin{equation*}
|\mathcal{T}_h (\eta)(\boldsymbol{x}',x_3)|\leq |x_3|\, ||{\MMM\Phi_h}||_{C^0(\closure{S}\times [-1,1])} \leq 1,
\end{equation*}
 so that $\mathcal{T}_h(\mathcal{X})\subset \mathcal{X}$.  Moreover, by the second estimate in \eqref{eqn:fh-estimates-1st}, for every $\eta,\widetilde{\eta}
  \in \mathcal{X}$ and $\boldsymbol{x}\in \Omega$, we have
\begin{equation*}
    \begin{split}
     |\mathcal{T}_h (\eta)(\boldsymbol{x}',x_3)-\mathcal{T}_h(\widetilde{\eta})(\boldsymbol{x}',x_3)|&\leq  \int_0^{x_3} |{\MMM\Phi_h}(\boldsymbol{x}',\eta(\boldsymbol{x}',t)) -{\MMM\Phi_h}(\boldsymbol{x}',\widetilde{\eta}(\boldsymbol{x}',t))|\,\d t\\
    &\leq  C h^{\beta/2+1}\,||\eta -\widetilde{\eta}||_{C^0(\closure{\Omega})}.
    \end{split}
\end{equation*}
Thus, for $h\ll1$, the operator $\mathcal{T}_h$ is a contraction   and, by the Banach Fixed-Point Theorem, it admits a unique fixed point $\eta_h \in \mathcal{X}$. 

By construction, $\eta_h$ is continuously differentiable with respect to the variable $x_3$, namely $\partial_3 \eta_h \in C^0(\closure{\Omega})$. 
We claim that $\eta_h \in C^1(\closure{\Omega})$. Fix $i \in \{1,2\}$. Note that, if $\eta \in \mathcal{X}$ is differentiable with respect to the variable $x_i$ with $\partial_i \eta \in C^0(\closure{\Omega})$, then so is $\mathcal{T}_h(\eta)$. Indeed, by \eqref{eqn:fh-estimates-1st}, we obtain
\begin{equation}
    \label{eqn:partial-Th-eta}
    \partial_i \mathcal{T}_h(\eta)(\boldsymbol{x}',x_3)=\int_0^{x_3} \Big ( \partial_i {\MMM\Phi_h}(\boldsymbol{x}',\eta(\boldsymbol{x}',t))+\partial_3 {\MMM\Phi_h}(\boldsymbol{x}',\eta(\boldsymbol{x'},t))\partial_i \eta(\boldsymbol{x}',t) \Big)\,\d t
\end{equation}
for every $\boldsymbol{x}\in \Omega$, so that $\partial_i \mathcal{T}_h(\eta)\in C^0(\closure{\Omega})$. Define the operator $\mathcal{S}_h \colon C^0(\closure{\Omega}) \times C^0(\closure{\Omega}) \to C^0(\closure{\Omega})$ by setting
\begin{equation*}
    \mathcal{S}_h(\eta,\xi)(\boldsymbol{x}',x_3)\coloneqq \int_0^{x_3} \Big ( \partial_i {\MMM\Phi_h}(\boldsymbol{x}',\eta(\boldsymbol{x}',t))+\partial_3 {\MMM\Phi_h}(\boldsymbol{x}',\eta(\boldsymbol{x}',t))\xi(\boldsymbol{x}',t) \Big)\,d t
\end{equation*}
for every $\eta,\xi \in C^0(\closure{\Omega})$  and $\boldsymbol{x}\in \Omega$. For $h \ll1$, we have $\mathcal{S}_h(\mathcal{X} \times \mathcal{X})\subset \mathcal{X}$. Indeed, by \eqref{eqn:fh-estimates-1st}, for every $\eta,\xi \in \mathcal{X}$ and  $\boldsymbol{x}\in \Omega$, there holds
\begin{equation*}
    |\mathcal{S}_h(\eta,\xi)(\boldsymbol{x}',x_3)|\leq \,|x_3|\,\left (||\partial_i {\MMM\Phi_h}||_{C^0(\closure{S}\times [-1,1])}+||\partial_3 {\MMM\Phi_h}||_{C^0(\closure{S}\times [-1,1])}\,||\xi||_{C^0(\closure{\Omega})} \right ) \leq C h^{\beta/2+1}.
\end{equation*}
Moreover, by \eqref{eqn:partial-Th-eta}, we have
\begin{equation}
    \label{eqn:partial-Th-Sh}
    \partial_i \mathcal{T}_h(\eta)=\mathcal{S}_h(\eta,\partial_i \eta)
\end{equation}
for every $\eta \in \mathcal{X}$ with $\partial_i \eta \in C^0(\closure{\Omega})$. 

Consider two sequences $(\eta^n),(\xi^n) \subset \mathcal{X}$ inductively defined as follows. We set $\eta^1 \coloneqq 0$ and $\xi^1 \coloneqq 0$, and we define $\eta^{n+1} \coloneqq \mathcal{T}_h(\eta^n)$ and $\xi^{n+1} \coloneqq \mathcal{S}_h(\eta^n,\xi^n)$ for every $n \in \N$. By the Banach Fixed-Point Theorem, we already know that
\begin{equation}
    \label{eqn:etan-convergence}
    \text{$\eta^n \to \eta_h$ in $C^0(\closure{\Omega})$,}
\end{equation}
as $n \to \infty$. Besides, for every $n \in \N$ we have
\begin{equation}
\label{eqn:etan-derivative}
    \xi^n=\partial_i \eta^n.
\end{equation}
This is proved by induction.
Indeed, $\xi^1=\partial_i \eta^1$ by definition, and assuming $\xi^n=\partial_i \eta^n$ for some $n \in \N$, by \eqref{eqn:partial-Th-Sh} we compute
\begin{equation*}
    \xi^{n+1}=\mathcal{S}_h(\eta^n,\xi^n)=\mathcal{S}_h(\eta^n,\partial_i \eta^n)=\partial_i \mathcal{T}_h(\eta^n)= \partial_i \eta^{n+1},
\end{equation*}
so that the claim follows. Define the operator $\mathcal{R}_h \colon C^0(\closure{\Omega}) \to C^0(\closure{\Omega})$ by setting $\mathcal{R}_h(\xi)\coloneqq \mathcal{S}_h(\eta_h,\xi)$ for every $\xi \in C^0(\closure{\Omega})$. Note that $\mathcal{R}_h(\mathcal{X})\subset \mathcal{X}$ for $h \ll 1$. By the second estimate in \eqref{eqn:fh-estimates-1st}, for every $\xi,\widetilde{\xi}\in \mathcal{X}$ and  $\boldsymbol{x}\in \Omega$ we have
\begin{equation*}
    \begin{split}
    |\mathcal{R}_h(\xi)(\boldsymbol{x}',x_3)-\mathcal{R}_h(\widetilde{\xi})(\boldsymbol{x}',x_3)|&\leq \int_0^{x_3} |\partial_3 {\MMM\Phi_h}(\boldsymbol{x}',\eta_h(\boldsymbol{x}',t))|\,|\xi(\boldsymbol{x}',t)-\widetilde{\xi}(\boldsymbol{x}',t)|\,\d t\\
    &\leq C h^{\beta/2+1}\,||\xi-\widetilde{\xi}||_{C^0(\closure{\Omega})}.
    \end{split}
\end{equation*}
Thus, for $h \ll 1$, the operator $\mathcal{R}_h$ is a contraction and, by the Banach Fixed-Point Theorem, it admits a unique fixed point $\xi_h \in \mathcal{X}$. By \eqref{eqn:fh-estimates-2nd}, for every $n \in \N$ we compute
\begin{equation*}
    \begin{split}
     \xi^{n+1}-\mathcal{R}_h(\xi^n)&=\int_0^{x_3} \big ((\partial_i {\MMM\Phi_h}(\boldsymbol{x}',\eta^n(\boldsymbol{x}',t))-\partial_i {\MMM\Phi_h}(\boldsymbol{x}',\eta_h(\boldsymbol{x}',t)) \big )\,\d t\\
     &+ \int_0^{x_3} \big(\partial_3 {\MMM\Phi_h}(\boldsymbol{x}',\eta^n(\boldsymbol{x}',t))-\partial_3 {\MMM\Phi_h}(\boldsymbol{x}',\eta_h(\boldsymbol{x}',t)))\xi^n\big )\,\d t\\
     &\leq C h^{\beta/2+1} ||\eta^n - \eta_h||_{C^0(\closure{\Omega})} + C h^{\beta/2+1} ||\eta^n - \eta_h||_{C^0(\closure{\Omega})} \,||\xi^n||_{C^0(\closure{\Omega})}\\
     &\leq C h^{\beta/2+1} \left ( 1 +||\xi^n||_{C^0(\closure{\Omega})} \right ).
    \end{split}
\end{equation*}
Hence, by the Ostrowski Theorem \cite[page 150]{walter}, we deduce that $\xi^n \to \xi_h$ in $C^0(\closure{\Omega})$, as $n \to \infty$. Recalling \eqref{eqn:etan-convergence} and \eqref{eqn:etan-derivative}, this yields $\xi_h=\partial_i \eta_h$ and, in particular, $\partial_i \eta_h \in C^0(\closure{\Omega})$. Therefore $\eta_h \in C^1(\closure{\Omega})$.

\textbf{Step 3.} In this step we prove that
\begin{equation}
    \label{eqn:comp-inc-grad-estimate}
    \boldsymbol{F}_h^\top(\boldsymbol{x}) \boldsymbol{F}_h(\boldsymbol{x})= \overline{\boldsymbol{F}}_h^\top(\boldsymbol{x}) \overline{\boldsymbol{F}}_h(\boldsymbol{x}) + O(h^{\beta/2+1})
\end{equation}
for every $\boldsymbol{x}\in \Omega$, where, we recall, $\boldsymbol{F}_h=\nabla_h \boldsymbol{y}_h$ and $\overline{\boldsymbol{F}}_h= \nabla_h \overline{\boldsymbol{y}}_h$.

The solution $\eta_h$ of \eqref{eqn:ode} satisfies the following estimates:
\begin{equation}
    \label{eqn:eta-estimates}
    ||\partial_3 \eta_h -1||_{C^0(\closure{\Omega})} \leq C h^{\beta/2+1}, \hskip 8 pt ||\eta_h - x_3||_{C^0(\closure{\Omega})}\leq C h^{\beta/2+1}, \hskip 8 pt ||\nabla'\eta_h||_{C^0(\closure{\Omega})} \leq C h^{\beta/2+1}.
\end{equation}
Recall \eqref{eqn:fh-definition}. To check the first estimate in \eqref{eqn:eta-estimates}, we use \eqref{eqn:ode} to compute
\begin{equation*}
	\begin{split}
		|\partial_3 \eta_h - 1|&=\left | \frac{1}{1+h^\beta {\MMM(P+x_3Q+x_3^2R)} + O(h^{\beta/2+1})}-1  \right|\\
		&=\frac{|h^{\beta/2-1}{\MMM(P+x_3Q+x_3^2R)}+O(1)|}{|1+h^\beta {\MMM(P+x_3Q+x_3^2R)} + O(h^{\beta/2+1})|} h^{\beta/2+1} \leq C h^{\beta/2+1}.
	\end{split}
\end{equation*}
To see the second estimate in \eqref{eqn:eta-estimates}, we use the first one and we have
\begin{equation*}
    |\eta_h - x_3|=\left | \int_0^{x_3} (\partial_3 \eta_h (\cdot,t)-1)\,\d t\right | \leq C h^{\beta/2+1}.
\end{equation*}
Finally, to prove the third estimate in \eqref{eqn:eta-estimates}, we use the fixed point property. Using the first estimate in \eqref{eqn:fh-estimates-1st}, for $i \in \{1,2\}$, we write
\begin{equation*}
    \begin{split}
        \partial_i \eta_h&= \partial_i \left( \int_0^{x_3} {\MMM\Phi_h}(\boldsymbol{x}',\eta_h(\boldsymbol{x}',t))\,\d t \right)= \int_0^{x_3} \partial_i \left( {\MMM\Phi_h}(\boldsymbol{x}',\eta_h(\boldsymbol{x}',t))\right)\,\d t\\
        &= \int_0^{x_3} \partial_i \left ( \frac{1}{1+h^\beta {\MMM(P+\eta_h Q+\eta_h^2R)} + O(h^{\beta/2+1})}\right)\,\d t\\
        &= - \int_0^{x_3} \frac{h^\beta {\MMM(\partial_i P+\partial_i \eta_h Q+\eta_h \partial_i Q+2\eta_h \partial_i\eta_h R+\eta_h^2\partial_i R)}}{(1+h^\beta {\MMM(P+\eta_h Q+\eta_h^2R)} + O(h^{\beta/2+1}))^2} \,\d t,
    \end{split}
\end{equation*}
so that
\begin{equation*}
    |\partial_i \eta_h| \leq h^{\beta/2+1} \int_0^{x_3} \frac{|h^{\beta/2-1}{\MMM(\partial_i P+\partial_i \eta_h Q+\eta_h \partial_i Q+2\eta_h \partial_i\eta_h R+\eta_h^2\partial_i R)}|}{(1+h^\beta {\MMM(P+\eta_h Q+\eta_h^2R)} + O(h^{\beta/2+1}))^2}\,\d t \leq C h^{\beta/2+1}.
\end{equation*}

We now use the estimates \eqref{eqn:eta-estimates} to prove \eqref{eqn:comp-inc-grad-estimate}.
By \eqref{eqn:y-inc-scal-grad}, for every $\boldsymbol{x}\in \Omega$ we have
\begin{equation*}
    \boldsymbol{F}_h(\boldsymbol{x})=\widehat{\boldsymbol{F}}_h(\boldsymbol{x})\,\boldsymbol{N}_h(\boldsymbol{x}),
\end{equation*}
where $\boldsymbol{N}_h$ is defined as in \eqref{eqn:N-h} and $\widehat{\boldsymbol{F}}_h(\boldsymbol{x})\coloneqq\overline{\boldsymbol{F}}_h(\boldsymbol{x}',\eta^h (\boldsymbol{x}))$. 
Note that $\boldsymbol{N}_h=O(1)$ thanks to the third estimate in  \eqref{eqn:eta-estimates} and the fact that $\partial_3 \eta_h={\MMM\Phi_h}(\boldsymbol{x}',\eta_h)$ is bounded. By the first and the third estimate in \eqref{eqn:eta-estimates}, we also deduce that $\boldsymbol{N}_h-\boldsymbol{I}=O(h^{\beta/2+1})$. Hence we compute
\begin{equation*}
    \begin{split}
        \boldsymbol{F}_h^\top \boldsymbol{F}_h - \overline{\boldsymbol{F}}_h^\top \overline{\boldsymbol{F}}_h&=\boldsymbol{N}_h^\top \widehat{\boldsymbol{F}}_h^\top \widehat{\boldsymbol{F}}_h \boldsymbol{N}_h-\overline{\boldsymbol{F}}_h^\top \overline{\boldsymbol{F}}_h\\
        &=\boldsymbol{N}_h^\top (\widehat{\boldsymbol{F}}_h^\top \widehat{\boldsymbol{F}}_h-\overline{\boldsymbol{F}}_h^\top \overline{\boldsymbol{F}}_h)\boldsymbol{N}_h+(\boldsymbol{N}_h - \boldsymbol{I})^\top \overline{\boldsymbol{F}}_h^\top \overline{\boldsymbol{F}}_h+\boldsymbol{N}_h^\top \overline{\boldsymbol{F}}_h^\top \overline{\boldsymbol{F}}_h (\boldsymbol{N}_h-\boldsymbol{I}).
    \end{split}
\end{equation*}
By \eqref{eqn:y-comp-nonlinear}, we have
\begin{equation*}
\widehat{\boldsymbol{F}}_h^\top \widehat{\boldsymbol{F}}_h-\overline{\boldsymbol{F}}_h^\top \overline{\boldsymbol{F}}_h=2h^{\beta/2}(\eta_h - x_3) \boldsymbol{B}+O(h^{\beta/2+1})=O(h^{\beta/2+1}),   
\end{equation*}
where we used the second estimate in \eqref{eqn:eta-estimates}. Also, note that $\overline{\boldsymbol{F}}_h^\top \overline{\boldsymbol{F}}_h=O(1)$ due to \eqref{eqn:y-comp-nonlinear}. Thus \eqref{eqn:comp-inc-grad-estimate} follows.

\textbf{Step 4.} We prove here that the deformations $\boldsymbol{y}_h$ are injective.
By construction, $\boldsymbol{y}_h \in C^1(\closure{\Omega};\R^3)$ satisfies $\det \nabla \boldsymbol{y}_h=h$ in $\Omega$. We argue as in \cite[Theorem 5.5-1]{ciarlet}. The expression of $\boldsymbol{y}_h$ is given by 
\begin{equation}
    \label{eqn:y-incompressible}
    \boldsymbol{y}_h={ \MMM\left ( \begin{matrix}\boldsymbol{0}'\\ h\eta_h\end{matrix}\right ) + h^{\beta/2} \left ( \begin{matrix}\boldsymbol{u}\\0\end{matrix}\right )}+h^{\beta/2-1} \left ( \begin{matrix}\boldsymbol{0}'\\v\end{matrix}\right )-h^{\beta/2} \eta_h \left ( \begin{matrix}\nabla'v\\0\end{matrix}\right )+ {\MMM2h^{\beta/2+1} \eta_h \boldsymbol{a}} +h^{\beta/2+1} \eta_h^2 \boldsymbol{b}.
\end{equation}
Set $\boldsymbol{\varphi}_h\coloneqq \boldsymbol{y}_h-\boldsymbol{z}_h$.
We define $\boldsymbol{w}_h \coloneqq \boldsymbol{y}_h \circ \boldsymbol{z}_h^{-1} \restr{\Omega_h}=\boldsymbol{id}\restr{\Omega_h}+\boldsymbol{\psi}_h \in C^1(\closure{\Omega}_h;\R^3)$, where $\boldsymbol{\psi}_h\coloneqq \boldsymbol{\varphi}_h \circ \boldsymbol{z}_h^{-1}\restr{\Omega_h}$. Then, we compute
\begin{equation}
    \label{eqn:grad-phih}
    \begin{split}
     \nabla_h \boldsymbol{\varphi}_h&={\MMM \left(\renewcommand\arraystretch{1.2} \begin{array}{@{}c|c@{}}   \boldsymbol{O}''  & \boldsymbol{0}'\\ \hline  h\nabla'\eta_h & \partial_3\eta_h-1 \end{array} \right)+h^{\beta/2}\left( \renewcommand\arraystretch{1.2} \begin{array}{@{}c|c@{}}   \nabla'\boldsymbol{u}  & \boldsymbol{0}'\\ \hline  (\boldsymbol{0}')^\top & 0 \end{array} \right)}+
     h^{\beta/2-1}\boldsymbol{e}_3 \otimes \left ( \begin{matrix} \nabla' v \\ 0 \end{matrix}\right)\\
     &- h^{\beta/2} \left ( \begin{matrix} \nabla' v \\ 0 \end{matrix} \right) \otimes \nabla_h \eta_h-h^{\beta/2} \eta_h \left(\renewcommand\arraystretch{1.2} \begin{array}{@{}c|c@{}}   (\nabla')^2 v  & \boldsymbol{0}'\\ \hline  (\boldsymbol{0}')^\top & 0 \end{array} \right)+{\MMM2h^{\beta/2+1} \boldsymbol{a}\otimes \nabla_h\eta_h}\\
     &+{\MMM2h^{\beta/2+1} \eta_h \left (\renewcommand\arraystretch{1.2} \begin{array}{@{}c|c@{}} \nabla' \boldsymbol{a}& \boldsymbol{0} \end{array}\right) }+ 2h^{\beta/2+1}\eta_h \boldsymbol{b} \otimes \nabla_h \eta_h + h^{\beta/2+1} \eta_h^2 \left (\renewcommand\arraystretch{1.2} \begin{array}{@{}c|c@{}} \nabla' \boldsymbol{b}& \boldsymbol{0} \end{array}\right).
    \end{split}
\end{equation}
Using \eqref{eqn:eta-estimates}, we deduce 
\begin{equation}
    \label{eqn:grad-psih}
    ||\nabla \boldsymbol{\psi}_h||_{C^0(\closure{\Omega}_h;\rtt)}=||\nabla_h \boldsymbol{\varphi}_h \circ \boldsymbol{z}_h^{-1}||_{C^0(\closure{\Omega}_h;\rtt)}\leq C h^{\beta/2-1}.
\end{equation}
Now, take any two distinct points $\boldsymbol{X}_1,\boldsymbol{X}_2 \in {\Omega}_h$ and set $\boldsymbol{x}_1 \coloneqq \boldsymbol{z}_h^{-1}(\boldsymbol{X}_1)\in \Omega$ and $\boldsymbol{x}_2 \coloneqq \boldsymbol{z}_h^{-1}(\boldsymbol{X}_2)\in \Omega$. Then, there exist a finite number of distinct points $\widetilde{\boldsymbol{x}}_1,\dots,\widetilde{\boldsymbol{x}}_m \in \Omega$ with $\widetilde{\boldsymbol{x}}_1=\boldsymbol{x}_1$ and $\widetilde{\boldsymbol{x}}_m=\boldsymbol{x}_2$ such that each segment connecting  $\widetilde{\boldsymbol{x}}_i$ to $\widetilde{\boldsymbol{x}}_{i+1}$ is entirely contained in $\Omega$ and $\sum_{i=1}^{m-1} |\widetilde{\boldsymbol{x}}_i-\widetilde{\boldsymbol{x}}_{i+1}|\leq C \, |\boldsymbol{x}_1-\boldsymbol{x}_2|$ for some constant $C>0$ depending only on $S$ \cite[page 224]{ciarlet}. Define $\widetilde{\boldsymbol{X}}_i \coloneqq \boldsymbol{z}_h(\widetilde{\boldsymbol{x}}_i)\in \Omega_h$ for every $i \in \{1,\dots,m\}$. By construction, each segment connecting $\widetilde{\boldsymbol{X}}_i$ to $\widetilde{\boldsymbol{X}}_{i+1}$ is entirely contained in $\Omega_h$, $\widetilde{\boldsymbol{X}}_1=\boldsymbol{X}_1$ and $\widetilde{\boldsymbol{X}}_2=\boldsymbol{X}_2$. Since $|\widetilde{\boldsymbol{X}}_i-\widetilde{\boldsymbol{X}}_{i+1}|\leq |\widetilde{\boldsymbol{x}}_i-\widetilde{\boldsymbol{x}}_{i+1}|$ and $|\boldsymbol{x}_1-\boldsymbol{x}_2| \leq h^{-1}|\boldsymbol{X}_1-\boldsymbol{X}_2|$, we deduce $\sum_{i=1}^{m-1} |\widetilde{\boldsymbol{X}}_i-\widetilde{\boldsymbol{X}}_{i+1}|\leq C h^{-1}\, |\boldsymbol{X}_1-\boldsymbol{X}_2|$.
Then, by  the Mean Value Theorem, we have
\begin{equation*}
    \begin{split}
     |\boldsymbol{w}_h(\boldsymbol{X}_1)-\boldsymbol{w}_h(\boldsymbol{X}_2)-(\boldsymbol{X}_1-\boldsymbol{X}_2)|&= |\boldsymbol{\psi_h}(\boldsymbol{X}_1) - \boldsymbol{\psi_h}(\boldsymbol{X}_2)| \leq \sum_{i=1}^{m-1} |\boldsymbol{\psi_h}(\widetilde{\boldsymbol{X}}_i) - \boldsymbol{\psi_h}(\widetilde{\boldsymbol{X}}_{i+1})|\\
     & \leq ||\nabla \boldsymbol{\psi}_h||_{C^0(\closure{\Omega}_h;\rtt)} \sum_{i=1}^{m-1} |\widetilde{\boldsymbol{X}}_i-\widetilde{\boldsymbol{X}}_{i+1}| \leq C h^{\beta/2-2} |\boldsymbol{X}_1-\boldsymbol{X}_2|,
    \end{split}
\end{equation*}
where in the last inequality we used \eqref{eqn:grad-psih}. Since $\beta>4$, for $h \ll 1$, we obtain 
$$ |\boldsymbol{w}_h(\boldsymbol{X}_1)-\boldsymbol{w}_h(\boldsymbol{X}_2)-(\boldsymbol{X}_1-\boldsymbol{X}_2)|<|\boldsymbol{X}_1-\boldsymbol{X}_2|.$$ Thus, as $\boldsymbol{X}_1 \neq \boldsymbol{X}_2$, we necessarily have $\boldsymbol{w}_h(\boldsymbol{X}_1)\neq \boldsymbol{w}_h(\boldsymbol{X}_2)$. Therefore we proved that $\boldsymbol{w}_h$ is injective and, in turn, so is $\boldsymbol{y}_h$. Thus, $\boldsymbol{y}_h \in \mathcal{Y}$. Claims \eqref{eqn:recovery-y}, \eqref{eqn:recovery-u} and \eqref{eqn:recovery-v} follow by direct computations.

\textbf{Step 5.}  This step is devoted to the proof of \eqref{eqn:recovery-elastic}. We can assume that $\boldsymbol{\lambda}$ is defined and smooth on the whole space. By continuity, since $|\boldsymbol{\lambda}|=1$ in $S$, we can find $V \subset \subset \R^2$ open with $S \subset \subset V$ such that $\boldsymbol{\lambda}$ does not vanish on $V$. Therefore, defining $\boldsymbol{\Lambda}\coloneqq \left(|\boldsymbol{\lambda}|^{-1}\boldsymbol{\lambda}\right )\restr{V}$, then  $\boldsymbol{\Lambda} \in C^\infty(\closure{V};\S^2)$ and we have $\boldsymbol{\Lambda}=\boldsymbol{\lambda}$ in $S$. Recall \eqref{eqn:y-incompressible}. Since $\eta_h$ takes values in $[-1,1]$ {\MMM and thanks to the second estimate in \eqref{eqn:eta-estimates}}, we have
\begin{equation}
    \label{eqn:sup-phih}
    ||\boldsymbol{y}_h-\boldsymbol{z}_h||_{C^0(\closure{\Omega};\R^3)}=||\boldsymbol{\varphi}_h||_{C^0(\closure{\Omega};\R^3)}\leq C h^{\beta/2-1} \eqqcolon \bar{\delta}_h.
\end{equation}
By \eqref{eqn:sup-phih}, we deduce $\Omega^{\boldsymbol{y}_h}\subset \boldsymbol{y}_h(\Omega)\subset \Omega_h + B(\boldsymbol{0},\bar{\delta}_h)$. Hence, for $ h \ll 1$, we have $\Omega^{\boldsymbol{y}_h} \subset \subset V \times \R$ and thus we can define $\boldsymbol{m}_h\coloneqq \boldsymbol{\Lambda}\restr{\Omega^{\boldsymbol{y}_h}}$. This gives $(\boldsymbol{y}_h,\boldsymbol{m}_h) \in \mathcal{Q}$.  Note that, for every $\boldsymbol{x} \in \Omega$, we have
\begin{equation*}
    \boldsymbol{m}_h(\boldsymbol{y}_h(\boldsymbol{x}))=\boldsymbol{\Lambda}(\boldsymbol{y}_h(\boldsymbol{x}))\to \boldsymbol{\Lambda}(\boldsymbol{z}_0(\boldsymbol{x}))=\boldsymbol{\Lambda}(\boldsymbol{x}')=\boldsymbol{\lambda}(\boldsymbol{x}'),
\end{equation*}
so that, applying the Dominated Convergence Theorem, we deduce \eqref{eqn:recovery-m-circ-y}. 

We now focus on the elastic energy. For simplicity, we set $\boldsymbol{\nu}_h \coloneqq \boldsymbol{m}_h \circ \boldsymbol{y}_h$. Recall that $\det \boldsymbol{F}_h=1$ in $\Omega$. By the Polar Decomposition Theorem, for every $h>0$ we have $\boldsymbol{F}_h=\boldsymbol{P}_h(\boldsymbol{F}_h^\top \boldsymbol{F}_h)^{1/2}$ for some $\boldsymbol{P}_h\in SO(3)$. Since $\boldsymbol{F}_h\to \boldsymbol{I}$ uniformly in $\Omega$ as we see from \eqref{eqn:y-incompressible} and \eqref{eqn:grad-phih}, passing to the limit in the previous identity we deduce that $\boldsymbol{P}_h \to \boldsymbol{I}$ uniformly in $\Omega$. Then, using \eqref{eqn:frame_indifference}, \eqref{eqn:taylor-expansion}, \eqref{eqn:y-comp-nonlinear} and \eqref{eqn:comp-inc-grad-estimate}, we compute
\begin{equation*}
    \begin{split}
    W^\inc(\boldsymbol{F}_h,\boldsymbol{\nu}_h)&=W(\boldsymbol{F}_h,\boldsymbol{\nu}_h)=W(\boldsymbol{P}_h(\boldsymbol{F}_h^\top \boldsymbol{F}_h)^{1/2},\boldsymbol{\nu}_h)=W((\boldsymbol{F}_h^\top \boldsymbol{F}_h)^{1/2},\boldsymbol{P}^\top_h\boldsymbol{\nu}_h)\\
    &=W((\overline{\boldsymbol{F}}_h^\top \overline{\boldsymbol{F}}_h+O(h^{\beta/2+1}))^{1/2},\boldsymbol{P}^\top_h\boldsymbol{\nu}_h)\\
    &=W((\boldsymbol{I}+2h^{\beta/2}{\MMM(\boldsymbol{A}+x_3\boldsymbol{B})}+O(h^{\beta/2+1}))^{1/2},\boldsymbol{P}^\top_h\boldsymbol{\nu}_h)\\
    &=W(\boldsymbol{I}+h^{\beta/2}{\MMM(\boldsymbol{A}+x_3\boldsymbol{B})}+O(h^{\beta/2+1}),\boldsymbol{P}^\top_h\boldsymbol{\nu}_h)\\
    &=\frac{1}{2}Q_3(h^{\beta/2}{\MMM(\boldsymbol{A}+x_3\boldsymbol{B})}+O(h^{\beta/2+1}),\boldsymbol{P}^\top_h\boldsymbol{\nu}_h) + \omega(h^{\beta/2}{\MMM(\boldsymbol{A}+x_3\boldsymbol{B})}+O(h^{\beta/2+1}),\boldsymbol{P}^\top_h\boldsymbol{\nu}_h)\\
    &=\frac{h^\beta}{2}Q_3({\MMM\boldsymbol{A}+x_3\boldsymbol{B}}+O(h),\boldsymbol{P}^\top_h\boldsymbol{\nu}_h) + \omega(h^{\beta/2}{\MMM(\boldsymbol{A}+x_3\boldsymbol{B})}+O(h^{\beta/2+1}),\boldsymbol{P}^\top_h\boldsymbol{\nu}_h),
    \end{split}
\end{equation*}
so that
\begin{equation}
    \label{eqn:recovery-split}
    \begin{split}
    	E^\mathrm{el}_h(\boldsymbol{y}_h,\boldsymbol{m}_h)&=\frac{1}{2}\int_\Omega Q_3({\MMM\boldsymbol{A}+x_3\boldsymbol{B}}+O(h),\boldsymbol{P}_h^\top\boldsymbol{\nu}_h) \,\d\boldsymbol{x}\\
    	&+ \frac{1}{h^\beta} \int_\Omega \omega(h^{\beta/2}({\MMM\boldsymbol{A}+x_3\boldsymbol{B}})+O(h^{\beta/2+1}),\boldsymbol{P}_h^\top \boldsymbol{\nu}_h)\,\d\boldsymbol{x}.
    \end{split}
\end{equation}
For the first integral on the right-hand side of \eqref{eqn:recovery-split}, applying the Dominated Convergence Theorem we obtain
\begin{equation*}
    \int_\Omega Q_3({\MMM\boldsymbol{A}+x_3\boldsymbol{B}}+O(h),\boldsymbol{P}_h^\top \boldsymbol{\nu}_h) \,\d\boldsymbol{x} \to \int_\Omega Q_3({\MMM \boldsymbol{A}+x_3\boldsymbol{B}},\boldsymbol{\lambda}) \,\d\boldsymbol{x}.
\end{equation*}
For the second integral on the right-hand side of \eqref{eqn:recovery-split}, by \eqref{eqn:coupling_omega} we have
\begin{equation*}
\begin{split}
    \frac{1}{h^\beta} \int_\Omega \omega(h^{\beta/2}({\MMM\boldsymbol{A}+x_3\boldsymbol{B}})&+O(h^{\beta/2+1}),\boldsymbol{P}_h^\top \boldsymbol{\nu}_h)\,\d\boldsymbol{x}\\
    &\leq \int_\Omega \overline{\omega}(|h^{\beta/2}({\MMM\boldsymbol{A}+x_3\boldsymbol{B}})+O(h^{\beta/2+1})|)\,\frac{|h^{\beta/2}({\MMM\boldsymbol{A}+x_3\boldsymbol{B}})+O(h^{\beta/2+1})|^2}{h^\beta}\,\d\boldsymbol{x}\\ 
    & \leq  \overline{\omega}(O(h^{\beta/2})) \int_\Omega |{\MMM\boldsymbol{A}+x_3\boldsymbol{B}}+O(h)|^2\,\d\boldsymbol{x} \leq C \,\overline{\omega}(O(h^{\beta/2})) \to 0.
\end{split}
\end{equation*}
Therefore we proved that
\begin{equation*}
	\begin{split}
		\lim_{h \to 0^+}E_h^\mathrm{el}(\boldsymbol{y}_h,\boldsymbol{m}_h)&=\frac{1}{2}\int_\Omega Q_3({\MMM\boldsymbol{A}+x_3\boldsymbol{B}},\boldsymbol{\lambda})\,\d\boldsymbol{x}={\MMM\frac{1}{2} \int_S Q_3( \MMM\boldsymbol{A},\boldsymbol{\lambda})\,\d\boldsymbol{x}'}+\frac{1}{24} \int_S Q_3( \boldsymbol{B},\boldsymbol{\lambda})\,\d\boldsymbol{x}',
	\end{split}
\end{equation*}
namely \eqref{eqn:recovery-elastic}.

\textbf{Step 6.} We now focus on the remaining terms of the energy. For the exchange energy, changing variables and passing to the limit with the Dominated Convergence Theorem, we obtain
\begin{equation*}
    \begin{split}
        E^\mathrm{exc}_h(\boldsymbol{y}_h,\boldsymbol{m}_h)&=\frac{\MMM \alpha}{h}\int_{\boldsymbol{y}_h(\Omega)}|\nabla \boldsymbol{m}_h|^2\,\d\boldsymbol{\xi}=\frac{\MMM \alpha}{h}\int_{\boldsymbol{y}_h(\Omega)} |\nabla \boldsymbol{\Lambda}|^2\,\d \boldsymbol{\xi}\\
        &=\frac{\MMM \alpha}{h}\int_{\boldsymbol{y}_h(\Omega)} |\nabla' \boldsymbol{\Lambda}|^2\,\d \boldsymbol{\xi}   ={\MMM \alpha}\int_\Omega |\nabla' \boldsymbol{\Lambda}|^2 \circ \boldsymbol{y}_h \,\d\boldsymbol{x}\\
        &\to {\MMM \alpha}\int_\Omega |\nabla' \boldsymbol{\Lambda}|^2 \circ \boldsymbol{z}_0 \,\d\boldsymbol{x}={\MMM \alpha}\int_S |\nabla' \boldsymbol{\Lambda}|^2 \,\d\boldsymbol{x}'={\MMM \alpha}\int_S |\nabla' \boldsymbol{\lambda}|^2 \,\d\boldsymbol{x}'=E^{\mathrm{exc}}(\boldsymbol{\lambda}),
    \end{split}
\end{equation*}
which is \eqref{eqn:recovery-exchange}.
Finally, we deal with the magnetostatic energy. By continuity, for every $\boldsymbol{x} \in \Omega$, we have
\begin{equation*}
    \boldsymbol{m}_h(\boldsymbol{z}_h(\boldsymbol{x}))=\boldsymbol{\Lambda}(\boldsymbol{z}_h(\boldsymbol{x})) \to \boldsymbol{\Lambda}(\boldsymbol{z}_0(\boldsymbol{x}))=\boldsymbol{\Lambda}(\boldsymbol{x}')=\boldsymbol{\lambda}(\boldsymbol{x}'),
\end{equation*}
as $h \to 0^+$.
Recall \eqref{eqn:sup-phih}. From this, after geometric considerations analogous to \eqref{eqn:subcylinder} and \eqref{eqn:supercylinder}, we obtain that $\chi_{\Omega^{\boldsymbol{y}_h}} \circ \boldsymbol{z}_h=\chi_{\boldsymbol{z}^{-1}_h(\Omega^{\boldsymbol{y}_h})} \to \chi_\Omega$ almost everywhere in $\R^3$. Then, applying the Dominated Convergence Theorem, we show \eqref{eqn:recovery-mu} and, arguing as in the proof of Proposition \ref{prop:liminf-mag}, we obtain
\begin{equation*}
    \lim_{h \to 0^+} E^{\mathrm{mag}}_h(\boldsymbol{y}_h,\boldsymbol{m}_h)=E^{\mathrm{mag}}(\boldsymbol{\lambda}),
\end{equation*}
so that \eqref{eqn:recovery-magnetostatic} is proven.
\end{proof}

We conclude this section with the proof of our second main result.

\begin{proof}[Proof of Theorem \ref{thm:optimality-lower-bound}]
By the definition of $Q_2^\inc$, there exists {\MMM$\boldsymbol{a},\boldsymbol{b}\colon S \to \R^3$} such that, given
\begin{equation*}
    {\MMM\boldsymbol{A}\coloneqq  \left( \renewcommand\arraystretch{1.2}\begin{array}{@{}c|c@{}}   \sym \nabla'\boldsymbol{u}  & \boldsymbol{0}'\\ \hline  (\boldsymbol{0}')^\top & 0 \end{array} \right) + \boldsymbol{a} \otimes \boldsymbol{e}_3 + \boldsymbol{e}_3 \otimes \boldsymbol{a},} \qquad \boldsymbol{B}\coloneqq - \left( \renewcommand\arraystretch{1.2}\begin{array}{@{}c|c@{}}   (\nabla')^2 v  & \boldsymbol{0}'\\ \hline  (\boldsymbol{0}')^\top & 0 \end{array} \right) + \boldsymbol{b} \otimes \boldsymbol{e}_3 + \boldsymbol{e}_3 \otimes \boldsymbol{b},
\end{equation*}
we have {\MMM$\tr\boldsymbol{A}=\div'\boldsymbol{u}+2a^3=0$ and $Q_3(\boldsymbol{A},\boldsymbol{\lambda})=Q^\inc_2(\sym\nabla'\boldsymbol{u},\boldsymbol{\lambda})$ in $S$}, as well as $\tr \boldsymbol{B}=-\Delta'v+2b^3=0$  and $Q_3(\boldsymbol{B},\boldsymbol{\lambda})=Q^\inc_2((\nabla')^2v,\boldsymbol{\lambda})$ in $S$. In particular, using the positive definiteness of $Q_3$ and the continuity of $Q_2^\inc$, we  deduce that $\boldsymbol{a},\boldsymbol{b} \in L^2(S;\R^3)$. Let the sequences {\MMM$(\boldsymbol{u}_n)\subset C^\infty(\closure{S};\R^2)$}, $(v_n) \subset C^\infty(\closure{S})$ and $(\boldsymbol{\lambda}_n) \subset C^\infty(\closure{S};\S^2)$ be such that {\MMM$\boldsymbol{u}_n\to\boldsymbol{u}$ in $W^{1,2}(S;\R^2)$}, $v_n \to v$ in $W^{2,2}(S)$ and $\boldsymbol{\lambda}_n \to \boldsymbol{\lambda}$ in $W^{1,2}(S;\R^3)$, as $n \to \infty$. Consider also two sequences $(\boldsymbol{a}_n),(\boldsymbol{b}_n) \subset C^\infty(\closure{S};\R^3)$ such that {\MMM$\boldsymbol{a}_n\to\boldsymbol{a}$ in $L^2(S;\R^3)$} and $\boldsymbol{b}_n \to \boldsymbol{b}$ in $L^2(S;\R^3)$, as $n \to \infty$, and set
\begin{equation*}
{\MMM\boldsymbol{A}_n\coloneqq  \left( \renewcommand\arraystretch{1.2}\begin{array}{@{}c|c@{}}   \sym \nabla'\boldsymbol{u}_n  & \boldsymbol{0}'\\ \hline  (\boldsymbol{0}')^\top & 0 \end{array} \right) + \boldsymbol{a}_n \otimes \boldsymbol{e}_3 + \boldsymbol{e}_3 \otimes \boldsymbol{a}_n,} \qquad \boldsymbol{B}_n\coloneqq - \left( \renewcommand\arraystretch{1.2}\begin{array}{@{}c|c@{}}   (\nabla')^2 v_n  & \boldsymbol{0}'\\ \hline  (\boldsymbol{0}')^\top & 0 \end{array} \right) + \boldsymbol{b}_n \otimes \boldsymbol{e}_3 + \boldsymbol{e}_3 \otimes \boldsymbol{b}_n.
\end{equation*}
We can always choose {\MMM$\boldsymbol{a}_n$} and $\boldsymbol{b}_n$ in order to have ${\MMM\tr \boldsymbol{A}_n}=\tr \boldsymbol{B}_n=0$ for every $n \in \N$. Indeed, given {\MMM$a_n^3\coloneqq -\frac{1}{2}\div'\boldsymbol{u}_n\in C^\infty(\closure{S})$} and  $b_n^3 \coloneqq \frac{1}{2}\Delta'v_n \in C^\infty(\closure{S})$, we have {\MMM$a_n^3\to -\frac{1}{2}\div'\boldsymbol{u}=a^3$ in $L^2(S)$} and
$b_n^3 \to \frac{1}{2}\Delta'v=b^3$ in $L^2(S)$, since {\MMM$\boldsymbol{u}_n \to \boldsymbol{u}$ in $W^{1,2}(S;\R^2)$}, $v_n \to v$ in $W^{2,2}(S)$ and ${\MMM\tr \boldsymbol{A}}=\tr \boldsymbol{B}=0$.

For each $n\in \N$, we can apply Proposition \ref{prop:recovery-sequence}. Therefore, there exists a sequence of admissible states $((\boldsymbol{y}^n_h,\boldsymbol{m}^n_h))_h \subset \mathcal{Q}$ satisfying, as $h \to 0^+$:
\begin{align*}
    &\text{{\MMM$\boldsymbol{u}^n_h\coloneqq \mathcal{U}_h(\boldsymbol{y}_h^n) \to \boldsymbol{u}_n$ in $W^{1,2}(S;\R^2)$}}, \hspace{10 mm} \text{$v_h^n\coloneqq \mathcal{V}_h(\boldsymbol{y}_h^n) \to v_n$ in $W^{1,2}(S)$},\\
     &\hspace{20mm}\text{$\boldsymbol{m}_h^n \circ \boldsymbol{y}_h^n \to \boldsymbol{\lambda}_n$ in $L^s(\Omega;\R^3)$ for every $1\leq s<\infty$,}\\
     &\hspace*{10mm}E_h^{\mathrm{el}}(\boldsymbol{y}^n_h,\boldsymbol{m}^n_h)\to {\MMM\frac{1}{2} \int_S Q_3(\boldsymbol{A}_n,\boldsymbol{\lambda}_n)\,\d\boldsymbol{x}'}+ \frac{1}{24} \int_S Q_3(\boldsymbol{B}_n,\boldsymbol{\lambda}_n)\,\d\boldsymbol{x}',\\
      &\hspace*{15mm}E_h^{\mathrm{exc}}(\boldsymbol{y}^n_h,\boldsymbol{m}^n_h) \to E^{\mathrm{exc}}(\boldsymbol{\lambda}_n), \hspace{6mm} E_h^{\mathrm{mag}}(\boldsymbol{y}^n_h,\boldsymbol{m}^n_h) \to E^{\mathrm{mag}}(\boldsymbol{\lambda}_n).
\end{align*}
Here, we adopted the notation in \eqref{eqn:averaged-displacements}. Using the Dominated Convergence Theorem, we see that
\begin{equation*}
    {\MMM\frac{1}{2} \int_S Q_3(\boldsymbol{A}_n,\boldsymbol{\lambda}_n)\,\d\boldsymbol{x}'}+ \frac{1}{24} \int_S Q_3(\boldsymbol{B}_n,\boldsymbol{\lambda}_n)\,\d\boldsymbol{x}'\to {\MMM\frac{1}{2} \int_S Q_3(\boldsymbol{A},\boldsymbol{\lambda})\,\d\boldsymbol{x}'}+ \frac{1}{24} \int_S Q_3(\boldsymbol{B},\boldsymbol{\lambda})\,\d\boldsymbol{x}',
\end{equation*}
as $n \to \infty$, and, recalling that $\boldsymbol{\lambda}_n \to \boldsymbol{\lambda}$ in $W^{1,2}(S;\R^3)$, we deduce
\begin{equation*}
    E^{\mathrm{exc}}(\boldsymbol{\lambda}_n)\to E^{\mathrm{exc}}(\boldsymbol{\lambda}), \qquad E^{\mathrm{mag}}(\boldsymbol{\lambda}_n)\to E^{\mathrm{mag}}(\boldsymbol{\lambda}),
\end{equation*}
as $n \to \infty$. Then, by a standard diagonal argument, we select an vanishing sequence $(h_n)$ such that we have, as $n \to \infty$:
\begin{align*}
    &\hspace*{10mm}\text{{\MMM$\boldsymbol{u}^n_{h_n} \to \boldsymbol{u}$ in $W^{1,2}(S;\R^2)$}}, \hspace{10 mm} \text{$v_{h_n}^n \to v$ in $W^{1,2}(S)$},\\
    &\hspace*{12mm}\text{$\boldsymbol{m}_{h_n}^n \circ \boldsymbol{y}_{h_n}^n \to \boldsymbol{\lambda}$ in $L^s(\Omega;\R^3)$ for every $1 \leq s < \infty$,}\hspace*{20 mm}\\
    &\hspace*{5mm}E_h^{\mathrm{el}}(\boldsymbol{y}^n_{h_n},\boldsymbol{m}^n_{h_n})\to {\MMM\frac{1}{2} \int_S Q_3(\boldsymbol{A},\boldsymbol{\lambda})\,\d\boldsymbol{x}'}+\frac{1}{24} \int_S Q_3(\boldsymbol{B},\boldsymbol{\lambda})\,\d\boldsymbol{x}',\\
    &E_h^{\mathrm{exc}}(\boldsymbol{y}^n_{h_n},\boldsymbol{m}^n_{h_n}) \to E^{\mathrm{exc}}(\boldsymbol{\lambda}), \hspace{10mm} E_h^{\mathrm{mag}}(\boldsymbol{y}^n_{h_n},\boldsymbol{m}^n_{h_n}) \to E^{\mathrm{mag}}(\boldsymbol{\lambda}).
\end{align*}
Thus \eqref{eqn:opt-lb-u}-\eqref{eqn:opt-lb-m-comp-y} are proved and, recalling our choice of $\boldsymbol{a}$ and $\boldsymbol{b}$, also \eqref{eqn:opt-lb-limit} is proved.
\end{proof}

\section{\texorpdfstring{$\Gamma$}{gamma}-convergence}
\label{sec:gamma}

This section is devoted to the proof of Corollary \ref{cor:gamma-convergence}. Recall the notation introduced in \eqref{eqn:averaged-displacements} and \eqref{eqn:gamma-notation-M}, and the definitions of the functionals $\mathcal{I}_h$ and $\mathcal{I}$ in \eqref{eqn:gamma-energy-Eh} and \eqref{eqn:gamma-energy-E}, respectively.

\begin{proof}[Proof of Corollary \ref{cor:gamma-convergence}]
We prove the result in the case of the strong product topology. The other case works exactly in the same way. 

\textbf{Proof of (i).} Suppose that  $(\boldsymbol{y}_h,\boldsymbol{u}_h,v_h,\boldsymbol{\mu}_h)\to(\boldsymbol{y},\boldsymbol{u},v,\boldsymbol{\mu})$ in $W^{1,p}(\Omega;\R^3)\times W^{1,2}(S;\R^2) \times W^{1,2}(S)\times L^2(\R^3;\R^3)$, as $h \to 0^+$. We can assume that the right-hand side of \eqref{eqn:liminf-inequality} is finite and, up to subsequences,  that the inferior limit is actually a limit. In this case, $\mathcal{I}_h(\boldsymbol{y}_h,\boldsymbol{u}_h,v_h,\boldsymbol{\mu}_h) \leq C$ for $h \ll 1$. Hence, by the definition of the functional $\mathcal{I}_h$, for $h \ll 1$ we have $\boldsymbol{u}_h=\mathcal{U}_h(\boldsymbol{y_h})$, $v_h=\mathcal{V}_h(\boldsymbol{y}_h)$ and $\boldsymbol{\mu}_h=\mathcal{M}_h(\boldsymbol{y}_h,\boldsymbol{m}_h)$ for some $\boldsymbol{m}_h \in W^{1,2}(\Omega^{\boldsymbol{y}_h};\S^2)$. Moreover, there holds
$E_h(\boldsymbol{y}_h,\boldsymbol{m}_h)\leq C$ for $h \ll 1$. Thus, we can apply Proposition \ref{prop:comp-def}. We obtain a sequence $(\boldsymbol{T}_h)$ of rigid motions of the form $\boldsymbol{T}_h(\boldsymbol{\xi})=\boldsymbol{Q}_h^\top \boldsymbol{\xi}-\boldsymbol{c}_h$ for every $\boldsymbol{\xi}\in \R^3$, where $\boldsymbol{Q}_h \in SO(3)$ and $\boldsymbol{c}_h \in \R^3$. Then, we define $\widetilde{\boldsymbol{y}}_h\coloneqq\boldsymbol{T}_h \circ \boldsymbol{y}_h$, and, up to subsequences, we have:
\begin{equation}
    \label{eqn:compactness-variant}
    \text{$\widetilde{\boldsymbol{y}}_h \to \boldsymbol{z}_0$ in $W^{1,p}(\Omega;\R^3)$, \hskip 8 pt $\widetilde{\boldsymbol{u}}_h\coloneqq \mathcal{U}_h(\widetilde{\boldsymbol{y}}_h) \wk \widetilde{\boldsymbol{u}}$ in $W^{1,2}(S;\R^2)$, \hskip 8 pt $\widetilde{v}_h\coloneqq \mathcal{V}_h(\widetilde{\boldsymbol{y}}_h) \to \widetilde{v}$ in $W^{1,2}(S)$,}    
\end{equation}
 for some $\widetilde{\boldsymbol{u}}\in W^{1,2}(S;\R^2)$ and $\widetilde{v} \in W^{2,2}(S)$.  For convenience, set $\boldsymbol{V}_h\coloneqq (h\boldsymbol{u}_h^\top,v_h)^\top$ and $\widetilde{\boldsymbol{V}}_h\coloneqq (h \widetilde{\boldsymbol{u}}_h^\top,\widetilde{v}_h)^\top$. Thus, $\boldsymbol{V}_h\to \boldsymbol{V}$ in $W^{1,2}(S;\R^3)$ and $\widetilde{\boldsymbol{V}}_h \to \widetilde{\boldsymbol{V}}$ in $W^{1,2}(S;\R^3)$, where $V\coloneqq((\boldsymbol{0}')^\top,v)^\top$ and $\widetilde{\boldsymbol{V}}\coloneqq((\boldsymbol{0}')^\top,\widetilde{v})^\top$. We compute
\begin{equation*}
    \begin{split}
    \int_S |\nabla' \widetilde{\boldsymbol{V}}_h-\boldsymbol{Q}_h^\top \nabla'\boldsymbol{V}_h|^2\,\d\boldsymbol{x}' &= \frac{1}{h^{\beta-2}}\int_S \left | \int_I \left( \nabla' \widetilde{\boldsymbol{y}}_h - \left( \renewcommand\arraystretch{1.2} \begin{matrix}
    \boldsymbol{I}'' \\ \hline (\boldsymbol{0}')^\top \end{matrix}
    \right) - \boldsymbol{Q}_h^\top \left (\nabla' \boldsymbol{y}_h - \left( \renewcommand\arraystretch{1.2} \begin{matrix}
    \boldsymbol{I}'' \\ \hline (\boldsymbol{0}')^\top \end{matrix}
    \right) \right) \right) \,\d x_3 \right|^2\,\d\boldsymbol{x}'\\
    &=\frac{1}{h^{\beta-2}} \int_S \left | \int_I (\boldsymbol{Q}_h^\top-\boldsymbol{I})\,\left( \renewcommand\arraystretch{1.2}  \begin{matrix}
    \boldsymbol{I}'' \\ \hline (\boldsymbol{0}')^\top \end{matrix}
    \right)\,\d x_3\right |^2\,\d \boldsymbol{x}'=\frac{C}{h^{\beta-2}}\,|\boldsymbol{Q}_h^\top - \boldsymbol{I}|^2.
    \end{split}
\end{equation*}
Since the left-hand side is bounded in $L^2(S;\R^{3 \times 2})$, because both sequences $(\nabla'\widetilde{\boldsymbol{V}}_h)$ and $(\nabla' \boldsymbol{V}_h)$ are convergent, we deduce that 
\begin{equation}
    \label{eqn:Qh-I-estimate}
    |\boldsymbol{Q}_h^\top - \boldsymbol{I}|\leq Ch^{\beta/2-1}.
\end{equation}
On the other hand, we have
\begin{equation*}
	\boldsymbol{V}_h-\widetilde{\boldsymbol{V}}_h=-\frac{1}{h^{\beta/2-1}} (\boldsymbol{Q}_h^\top-\boldsymbol{I}) \int_I \boldsymbol{y}_h(\cdot,x_3)\,\d x_3+\frac{1}{h^{\beta/2-1}}\boldsymbol{c}_h
\end{equation*}
from which we obtain
\begin{equation*}
	\frac{1}{h^{\beta/2-1}}\boldsymbol{c}_h=\boldsymbol{V}_h-\widetilde{\boldsymbol{V}}_h+\frac{1}{h^{\beta/2-1}} (\boldsymbol{Q}_h^\top-\boldsymbol{I}) \int_I \boldsymbol{y}_h(\cdot,x_3)\,\d x_3.
\end{equation*}
Since the right-hand side is bounded in $L^2(S;\R^3)$ because of \eqref{eqn:Qh-I-estimate} and the fact that both the sequences $(\boldsymbol{V}_h)$, $(\widetilde{\boldsymbol{V}}_h)$  are convergent in $L^2(S;\R^3)$ and $(\boldsymbol{y}_h)$ is convergent in $L^2(\Omega;\R^3)$, we deduce that
\begin{equation}
	\label{eqn:ch-estimate}
	|\boldsymbol{c}_h|\leq C h^{\beta/2-1}.
\end{equation}
From the first estimate in \eqref{eqn:key-estimate} and from \eqref{eqn:Qh-I-estimate} and \eqref{eqn:ch-estimate}, we obtain that
\begin{equation}
	\label{eqn:key-estimate-variant}
	||\boldsymbol{y}_h - \boldsymbol{z}_h||_{C^0(\closure{\Omega};\R^3)}\leq C h^{\beta/p-1}.
\end{equation}
Hence, $\boldsymbol{y}_h \to \boldsymbol{z}_0$ uniformly in $\Omega$, so that, by the Morrey embedding, $\boldsymbol{y}=\boldsymbol{z}_0$. Moreover, exploiting \eqref{eqn:key-estimate-variant} and the bound $E^{\text{exc}}_h(\boldsymbol{y}_h,\boldsymbol{m}_h)\leq C$ for $h \ll 1$, and arguing as in the proof of Proposition \ref{prop:comp-mag}, we show that 
\begin{equation}
	\label{eqn:conv-mu-variant}
	\text{$\boldsymbol{\mu}_h \coloneqq \mathcal{M}_h(\boldsymbol{y}_h,\boldsymbol{m}_h) \to \chi_\Omega \boldsymbol{\lambda}$ in $L^r(\R^3;\R^3)$ for every $1 \leq r < \infty$,}
\end{equation}
for some $\boldsymbol{\lambda}\in W^{1,2}(S;\S^2)$. In particular, we deduce $\boldsymbol{\mu}=\chi_\Omega \boldsymbol{\lambda}$.

Combining \eqref{eqn:Qh-I-estimate} with the second estimate in \eqref{eqn:approx-rot2} with $q=2$, we obtain
\begin{equation}
    \label{eqn:variant-approx-rot}
   \int_\Omega |\nabla_h \boldsymbol{y}_h-\boldsymbol{I}|^2\,\d\boldsymbol{x}\leq C h^{\beta-2}. 
\end{equation}
Thanks to the estimate \eqref{eqn:variant-approx-rot}, we can repeat the argument of the proof of Proposition \ref{prop:comp-def}, but with $\boldsymbol{I}$ in place of $\boldsymbol{Q}_h$. Consider the translation motion $\widecheck{\boldsymbol{T}}_h$ defined by $\widecheck{\boldsymbol{T}}_h(\boldsymbol{\xi})\coloneqq\boldsymbol{\xi}-\widecheck{\boldsymbol{c}}_h$ for every $\boldsymbol{\xi}\in \R^3$, where $\widecheck{\boldsymbol{c}}_h \coloneqq \dashint_\Omega (\boldsymbol{y}_h-\boldsymbol{z}_h)\,\d\boldsymbol{x}$.
We define $\widecheck{\boldsymbol{y}}_h\coloneqq \widecheck{\boldsymbol{T}}_h \circ \boldsymbol{y}_h=\boldsymbol{y}_h-\widecheck{\boldsymbol{c}}_h$ and $\widecheck{\boldsymbol{m}}_h \coloneqq \boldsymbol{m}_h \circ \widecheck{\boldsymbol{T}}_h^{-1}$. 
We show that 
\begin{equation}
	\label{eqn:check-conv}
	\text{$\widecheck{\boldsymbol{u}}_h \coloneqq \mathcal{U}_h(\widecheck{\boldsymbol{y}}_h)\wk \widecheck{\boldsymbol{u}}$ in $W^{1,2}(S;\R^2)$, \hskip 8 pt $\widecheck{v}_h\coloneqq \mathcal{V}_h(\widetilde{\boldsymbol{y}}_h) \to \widecheck{v}$ in $W^{1,2}(S)$,}    
\end{equation}
for some $\widecheck{\boldsymbol{u}}\in W^{1,2}(S;\R^2)$ and $\widecheck{v} \in W^{2,2}(S)$. Then, reasoning as in the proof of Proposition \ref{prop:comp-mag}, we find $\widecheck{\boldsymbol{\lambda}}\in W^{1,2}(S;\S^2)$ such that
\begin{align}
	\label{eqn:check-conv-mag}
	&\text{$\widecheck{\boldsymbol{m}}_h \circ \widecheck{\boldsymbol{y}}_h \to \boldsymbol{\lambda}$ in $L^r(\Omega;\R^3)$ for every $1 \leq r < \infty$,}\\
	\label{eqn:check-mu}
	&\text{$\widecheck{\boldsymbol{\mu}}_h \coloneqq \mathcal{M}_h(\widecheck{\boldsymbol{y}}_h,\widecheck{\boldsymbol{m}_h})\to \chi_\Omega \widecheck{\boldsymbol{\lambda}}$ in $L^r(\Omega;\R^3)$ for every $1 \leq r < \infty$.}
\end{align}
for some $\widecheck{\boldsymbol{\lambda}}\in W^{1,2}(S;\S^2)$. Besides, following the proof of Propositions \ref{prop:liminf-exc}, \ref{prop:liminf-el} and \ref{prop:liminf-mag}, we show that
\begin{equation}
	\label{eqn:liminf-inequality-almost}
	E(\widecheck{\boldsymbol{u}},\widecheck{v},\widecheck{\boldsymbol{\lambda}})\leq \liminf_{h \to 0^+} E_h(\widecheck{\boldsymbol{y}}_h,\widecheck{\boldsymbol{m}}_h).
\end{equation}
Thanks to the invariance of the energy $E_h$ with respect to translations (see Remark \ref{rem:invariance}), we have $E_h(\widecheck{\boldsymbol{y}}_h,\widecheck{\boldsymbol{m}}_h)=E_h(\boldsymbol{y}_h,\boldsymbol{m}_h)=\mathcal{I}_h(\boldsymbol{y}_h,\boldsymbol{u}_h,v_h,\boldsymbol{\mu}_h)$. As $\nabla'\widecheck{\boldsymbol{u}}_h=\nabla'\boldsymbol{u}_h$ and $\nabla'\widecheck{v}_h=\nabla'v_h$, passing to the limit, as $h \to 0^+$, we deduce $\nabla'\widecheck{\boldsymbol{u}}=\nabla'\boldsymbol{u}$ and $\nabla'\widecheck{v}=\nabla'v$. In particular, we have $v \in W^{2,2}(S)$, and, recalling that $\boldsymbol{y}=\boldsymbol{z}_0$ and $\boldsymbol{\mu}=\chi_\Omega \boldsymbol{\lambda}$, there holds $\mathcal{I}(\boldsymbol{y},\boldsymbol{u},v,\boldsymbol{\mu})=E(\boldsymbol{u},v,\boldsymbol{\lambda})=E(\widecheck{\boldsymbol{u}},\widecheck{v},\boldsymbol{\lambda})$.
Therefore, if we show that $\widecheck{\boldsymbol{\lambda}}=\boldsymbol{\lambda}$, then $E(\widecheck{\boldsymbol{u}},\widecheck{v},\widecheck{\boldsymbol{\lambda}})=\mathcal{I}(\boldsymbol{y},\boldsymbol{u},v,\boldsymbol{\mu})$, so that \eqref{eqn:liminf-inequality-almost} yields \eqref{eqn:liminf-inequality}.

{\MMM
We now prove this last claim. Setting $\widecheck{\boldsymbol{V}}_h\coloneqq (h\widecheck{\boldsymbol{u}}_h^\top,\widecheck{v}_h)^\top$ and $\widecheck{\boldsymbol{V}}\coloneqq((\boldsymbol{0}')^\top,\widecheck{v})^\top$, by \eqref{eqn:check-conv}, we have $\widecheck{\boldsymbol{V}}_h \to \widecheck{\boldsymbol{V}}$ in $W^{1,2}(S;\R^3)$. Note that $\boldsymbol{V}_h-\widecheck{\boldsymbol{V}}_h=h^{-\beta/2+1} \widecheck{\boldsymbol{c}}_h$.
Since the left-hand side of the previous equation is bounded in $L^2(S;\R^3)$, we deduce $|\widecheck{\boldsymbol{c}}_h|\leq Ch^{\beta/2-1}$. Let $\boldsymbol{\varphi}\in C^\infty_c(S;\R^3)$ and denote by $\overline{\boldsymbol{\varphi}}$ its extension to the whole space by zero. Using the change-of-variable formula and the identity $\chi_{\Omega^{\widecheck{\boldsymbol{y}}_h}}\widecheck{\boldsymbol{m}}_h=(\chi_{\Omega^{\boldsymbol{y}_h}}\boldsymbol{m}_h)\circ \widecheck{\boldsymbol{T}}_h^{-1}$, we compute
\begin{equation}
	\label{eqn:gamma-final}
	\begin{split}
		\int_{\R^3} \widecheck{\boldsymbol{\mu}}_h \cdot \overline{\boldsymbol{\varphi}}\,\d\boldsymbol{x}&=\int_{\R^3} (\chi_{\Omega^{\boldsymbol{y}_h}}\boldsymbol{m}_h) \circ \widecheck{\boldsymbol{T}}_h^{-1}\circ \boldsymbol{z}_h \cdot \overline{\boldsymbol{\varphi}}\,\d\boldsymbol{x}\\
		&=\frac{1}{h}\int_{\R^3} \chi_{\Omega^{\boldsymbol{y}_h}}\boldsymbol{m}_h \cdot \overline{\boldsymbol{\varphi}} \circ \boldsymbol{z}_h^{-1}\circ \widecheck{\boldsymbol{T}}_h\,\d\boldsymbol{\xi}\\
		&=\int_{\R^3} \boldsymbol{\mu}_h \cdot \overline{\boldsymbol{\varphi}} \circ \boldsymbol{z}_h^{-1}\circ \widecheck{\boldsymbol{T}}_h \circ \boldsymbol{z}_h\,\d\boldsymbol{x}.
	\end{split}
\end{equation}
Note that $	\boldsymbol{z}_h^{-1}\circ \widecheck{\boldsymbol{T}}_h \circ \boldsymbol{z}_h(\boldsymbol{x})=\boldsymbol{x}-((\widecheck{\boldsymbol{c}}_h')^\top,h^{-1}\widecheck{c}_h^{\,3})^\top \to \boldsymbol{x}$ for every $\boldsymbol{x}\in \Omega$, since $|\widecheck{\boldsymbol{c}}_h|\leq Ch^{\beta/2-1}$. Thus, we have $\overline{\boldsymbol{\varphi}} \circ \boldsymbol{z}_h^{-1}\circ \widecheck{\boldsymbol{T}}_h \circ \boldsymbol{z}_h \to \overline{\boldsymbol{\varphi}}$ in $L^2(\R^3;\R^3)$ thanks to the Dominated Convergence Theorem. Recalling \eqref{eqn:conv-mu-variant} and \eqref{eqn:check-mu}, and passing to the limit, as $h \to 0^+$, in \eqref{eqn:gamma-final}, we obtain
\begin{equation*}
	\int_S \widecheck{\boldsymbol{\lambda}} \cdot \boldsymbol{\varphi}\,\d\boldsymbol{x}'=\int_{\Omega} \widecheck{\boldsymbol{\lambda}} \cdot \overline{\boldsymbol{\varphi}}\,\d\boldsymbol{x}=\int_{\Omega} {\boldsymbol{\lambda}} \cdot \overline{\boldsymbol{\varphi}}\,\d\boldsymbol{x}=\int_S {\boldsymbol{\lambda}} \cdot \boldsymbol{\varphi}\,\d\boldsymbol{x}'.
\end{equation*}
Since $\boldsymbol{\varphi}$ is arbitrary, the equality $\widecheck{\boldsymbol{\lambda}}=\boldsymbol{\lambda}$ follows.
}

\textbf{Proof of (ii).}
The existence of recovery sequences follows directly from Theorem \ref{thm:optimality-lower-bound}. Consider $(\boldsymbol{y},\boldsymbol{u},v,\boldsymbol{\mu})\in W^{1,p}(\Omega;\R^3)\times W^{1,2}(S;\R^2)\times W^{1,2}(S)\times L^2(\R^3;\R^3)$. If $\mathcal{I}(\boldsymbol{y},\boldsymbol{u},v,\boldsymbol{\mu})=+\infty$,then the constant sequence $((\boldsymbol{y}_h,\boldsymbol{u}_h,v_h,\boldsymbol{\mu}_h))_h$ with $\boldsymbol{y}_h=\boldsymbol{y}$, $\boldsymbol{u}_h=\boldsymbol{u}$, $v_h=v$ and $\boldsymbol{\mu}_h=\boldsymbol{\mu}$, is a recovery sequence. Suppose that $\mathcal{I}(\boldsymbol{y},\boldsymbol{u},v,\boldsymbol{\mu})<+\infty$. In this case, by the definition of the functional $\mathcal{I}$ in \eqref{eqn:gamma-energy-E}, we have $\boldsymbol{y}=\boldsymbol{z}_0$, $v \in W^{2,2}(S)$ and $\boldsymbol{\mu}=\chi_\Omega \boldsymbol{\lambda}$ for some $\boldsymbol{\lambda} \in W^{1,2}(S;\S^2)$. Also, $\mathcal{I}(\boldsymbol{y},\boldsymbol{u},v,\boldsymbol{\mu})=E(\boldsymbol{u},v,\boldsymbol{\lambda})$. By Theorem \ref{thm:optimality-lower-bound}, we obtain a sequence $((\boldsymbol{y}_h,\boldsymbol{m}_h))_h \subset \mathcal{Q}$ satisfying the following:
\begin{align}
	&\nonumber \text{$\boldsymbol{u}_h\coloneqq \mathcal{U}_h(\boldsymbol{y}_h)\to \boldsymbol{u}$ in $W^{1,2}(S;\R^2$), \quad $v_h\coloneqq\mathcal{V}_h(\boldsymbol{y}_h)\to v$ in $W^{1,2}(S)$,}\\
	&\nonumber \hspace*{15mm}\text{$\boldsymbol{m}_h \circ \boldsymbol{y}_h \to \boldsymbol{\lambda}$ in $L^r(\Omega;\R^3)$ for every $1 \leq r < \infty$,}
\end{align}
Moreover, this also gives
\begin{equation}
	\label{eqn:recovery-sequence-almost}
	E(\boldsymbol{u},v,\boldsymbol{\lambda})=\lim_{h \to 0^+} E_h(\boldsymbol{y}_h,\boldsymbol{m}_h).
\end{equation}	
If we set $\boldsymbol{\mu}_h\coloneqq \mathcal{M}_h(\boldsymbol{y}_h,\boldsymbol{m}_h)$, then $\mathcal{I}_h(\boldsymbol{y}_h,\boldsymbol{u}_h,v_h,\boldsymbol{\mu}_h)=E_h(\boldsymbol{y}_h,\boldsymbol{m}_h)$, so that \eqref{eqn:recovery-sequence} follows from \eqref{eqn:recovery-sequence-almost}.
A direct computation shows that $\boldsymbol{\mu}_h \to \chi_\Omega\boldsymbol{\lambda}$ in $L^r(\R^3;\R^3)$ for every $1 \leq r < \infty$, which concludes the proof.
\end{proof}

\section*{Acknowledgements}
The author is thankful to his advisor Elisa Davoli for proposing this problem to him and for many helpful discussions and suggestions. He is also thankful to Martin Kru\v{z}\'{i}k for discussions on the topic of magnetoelasticity {\MMM and to the anonimous referee for various comments that led to an improvement of the manuscript.} This work has been supported by the Austrian Science Fund (FWF) through the grant I4052-N32 and by the Federal Ministry of Education, Science and Research of Austria (BMBWF) through the OeAD-WTZ project CZ04/2019.

\end{document}